\documentclass[letterpaper,11pt]{article}
\usepackage{color}

\usepackage[small]{titlesec}
\usepackage{theorem,amsmath,amssymb,amscd}
\usepackage[all]{xy}
\usepackage{booktabs}

\usepackage[hyperfootnotes=false, colorlinks, linkcolor={blue}, citecolor={magenta}, filecolor={blue}, urlcolor={blue}]{hyperref}

\usepackage[normalem]{ulem}
\usepackage{cancel}
\usepackage{soul}

\theoremstyle{change}
\allowdisplaybreaks
\newcommand{\R}{{\mathbb R}}
\newcommand{\C}{{\mathbb C}}

\newcommand{\p}{\mathfrak p}
\newcommand{\OF}{{\mathfrak o}}
\newcommand{\GL}{{\rm GL}}

\newcommand{\SL}{{\rm SL}}
\newcommand{\SO}{{\rm SO}}

\newcommand{\GSp}{{\rm GSp}}

\newcommand{\meta}{\widetilde{\rm SL}}

\newcommand{\St}{{\rm St}}
\newcommand{\triv}{1}

\newcommand{\Norm}{{\rm N}}
\newcommand{\Mat}{{\rm M}} 
\newcommand{\Gal}{{\rm Gal}}
\newcommand{\GSO}{{\rm GSO}}
\newcommand{\GO}{{\rm GO}}
\newcommand{\OO}{{\rm O}}

\newcommand{\ind}{{\rm ind}}
\newcommand{\cInd}{\mathrm{c}\text{-}\mathrm{Ind}}
\newcommand{\Trace}{{\rm T}}
\newcommand{\trace}{{\rm tr}}

\newcommand{\Hom}{{\rm Hom}}
\newcommand{\SSp}{{\rm Sp}}

\newcommand{\disc}{{\rm disc}}
\newcommand{\pisw}{\pi_{\scriptscriptstyle SW}}
\newcommand{\pis}{\pi_{\scriptscriptstyle S}}
\newcommand{\piw}{\pi_{\scriptscriptstyle W}}

\newcommand{\sroot}{{\Delta}}
\renewcommand{\emptyset}{\varnothing}

\newcommand{\mat}[4]{{\setlength{\arraycolsep}{0.5mm}\left[
\begin{array}{cc}#1&#2\\#3&#4\end{array}\right]}}
\newcommand{\qed}{\hspace*{\fill}\rule{1ex}{1ex}}
\newcommand{\forget}[1]{}

\def\qdots{\mathinner{\mkern1mu\raise0pt\vbox{\kern7pt\hbox{.}}\mkern2mu
\raise3.4pt\hbox{.}\mkern2mu\raise7pt\hbox{.}\mkern1mu}}

\newenvironment{proof}{\vspace{1ex}\noindent\emph{Proof.}\hspace{0.5em}}
	{\hfill\qed\vspace{2ex}}

\newtheorem{lemma}{Lemma.}[subsection]
\newtheorem{theorem}[lemma]{Theorem.}
\newtheorem{corollary}[lemma]{Corollary.}
\newtheorem{proposition}[lemma]{Proposition.}

\begin{document}

\thispagestyle{empty}
\begin{center}
 {\bf\Large Some results on Bessel functionals for $\GSp(4)$}

 \vspace{2ex}
 Brooks Roberts and Ralf Schmidt\footnote{Supported by NSF grant DMS-1100541.\\2010 \emph{Mathematics Subject Classification}. Primary 11F70 and 22E50.}

 \vspace{5ex}
 \begin{minipage}{80ex}
  \small{\sc Abstract.} We prove that every irreducible, admissible representation $\pi$ of $\GSp(4,F)$, where $F$ is a non-archimedean local field of characteristic zero, admits a Bessel functional, provided $\pi$ is not one-dimensional. Given $\pi$, we explicitly determine the set of all split Bessel functionals admitted by $\pi$, and prove that these functionals are unique. If $\pi$ is not supercuspidal, or in an $L$-packet with a non-supercuspidal representation, we explicitly determine the set of all Bessel functionals admitted by $\pi$, and prove that these functionals are unique.
 \end{minipage}
\end{center}

\vspace{3ex}
\tableofcontents

\section*{Introduction}
\phantomsection
\label{introsec}
\addcontentsline{toc}{section}{Introduction}
Let $F$ be a non-archimedean local field of characteristic zero, and let $\psi$ be a non-trivial character of $F$. Let $\GSp(4,F)$ be the subgroup of $g$ in $\GL(4,F)$ satisfying $^tgJg=\lambda(g)J$ for some scalar $\lambda(g)$ in $F^\times$, where
$$
J=\begin{bmatrix}&&&1\\&&1\\&-1\\-1\end{bmatrix}.
$$
The Siegel parabolic subgroup $P$ of $\GSp(4,F)$ is the subgroup consisting of all matrices whose lower left $2\times2$ block is zero. Let $N$ be the unipotent radical of $P$. The characters $\theta$ of $N$ are in one-to-one correspondence with symmetric $2\times2$ matrices $S$ over $F$ via the formula
$$
 \theta(\mat{1}{X}{}{1})=\psi({\rm tr}(S\mat{}{1}{1}{}X)).
$$
We say that $\theta$ is \emph{non-degenerate} if the matrix $S$ is invertible, and we say that $\theta$ is \emph{split} if $\disc(S)=1$; here $\disc(S)$ is the class of $-\det(S)$ in $F^\times/F^{\times2}$. For a fixed $S$, we define
\begin{equation}\label{Tdefeq2}
 T=\mat{}{1}{1}{}\{g\in \GL(2,F):\:^tgSg=\det(g)S\}\mat{}{1}{1}{}.
\end{equation}
We embed $T$ into $\GSp(4,F)$ via the map
$$
 t\mapsto\mat{t}{}{}{\det(t)t'},
$$
where for a $2\times2$-matrix $g$ we write $g'=\left[\begin{smallmatrix}&1\\1\end{smallmatrix}\right]\,^tg^{-1}\left[\begin{smallmatrix}&1\\1\end{smallmatrix}\right]$. The group $T$ normalizes $N$, so that we can define the semidirect product $D=TN$. This will be referred to as the \emph{Bessel subgroup} corresponding to $S$. For $t$ in $T$ and $n$ in $N$, we have $\theta(tnt^{-1})=\theta(n)$. Thus, if $\Lambda$ is a character of $T$, we can define a character $\Lambda\otimes\theta$ of $D$ by $(\Lambda\otimes\theta)(tn)=\Lambda(t)\theta(n)$. Whenever we regard $\C$ as a one-dimensional representation of $D$ via this character, we denote it by $\C_{\Lambda\otimes\theta}$. Let $(\pi,V)$ be an irreducible, admissible representation of $\GSp(4,F)$. A non-zero element of the space $\Hom_D(\pi,\C_{\Lambda\otimes\theta})$ is called a $(\Lambda,\theta)$-Bessel functional for $\pi$. We say that $\pi$ admits a $(\Lambda,\theta)$-Bessel functional if $\Hom_D(\pi,\C_{\Lambda\otimes\theta})$ is non-zero, and that $\pi$ 
admits a unique $(\Lambda,\theta)$-Bessel functional if $\Hom_D(\pi,\C_{\Lambda\otimes\theta})$ is one-dimensional.

In this paper we investigate the existence and uniqueness of Bessel functionals. We prove three main results about irreducible, admissible representations $\pi$ of $\GSp(4,F)$.

\vspace{2ex}
\hspace{-3ex}
\begin{minipage}{90ex}
\begin{itemize}
 \item If $\pi$ is not one-dimensional, we prove that $\pi$ admits some $(\Lambda,\theta)$-Bessel functional; see Theorem~\ref{existencetheorem}.
 \item If $\theta$ is split, we determine the set of $\Lambda$ for which $\pi$ admits a $(\Lambda,\theta)$-Bessel functional, and prove that such functionals are unique; see Proposition \ref{GSp4genericprop}, Theorem \ref{existencetheorem}, Theorem \ref{mainnonsupercuspidaltheorem} and Theorem \ref{splituniquenesstheorem}.
 \item If $\pi$ is non-supercuspidal, or is in an $L$-packet with a non-supercuspidal representation, we determine the set of $(\Lambda,\theta)$ for which $\pi$ admits a $(\Lambda,\theta)$-Bessel functional, and prove that such functionals are unique; see Theorem \ref{mainnonsupercuspidaltheorem} and Theorem \ref{splituniquenesstheorem}.
\end{itemize}
\end{minipage}

\vspace{3ex}\noindent
We point out that all our results hold independently of the residual characteristic of $F$.

To investigate $(\Lambda,\theta)$-Bessel functionals for $(\pi,V)$ we use the $P_3$-module $V_{Z^J}$, the $G^J$-module $V_{Z^J,\psi}$, and the twisted Jacquet module $V_{N,\theta}$. Here,
$$
 P_3=\GL(3,F)\cap\begin{bmatrix} *&*&*\\ *&*&*\\&&1\end{bmatrix},\;\:
 Z^J=\GSp(4,F)\cap\begin{bmatrix}1&&&*\\&1\\&&1\\&&&1\end{bmatrix},\;\:
 G^J=\GSp(4,F)\cap\begin{bmatrix}1&*&*&*\\&*&*&*\\&*&*&*\\&&&1\end{bmatrix}.
$$
The $P_3$-module $V_{Z^J}$ was computed for all $\pi$ with trivial central character in $\cite{NF}$; in this paper, we note that these results extend to the general case. The $G^J$-module $V_{Z^J,\psi}$ is closely related to representations of the metaplectic group $\meta(2,F)$. The twisted Jacquet module $V_{N,\theta}$ is especially relevant for non-supercuspidal representations. Indeed, we completely calculate twisted Jacquet modules of representations parabolically induced from the Klingen or Siegel parabolic subgroups. These methods suffice to treat most representations; for the few remaining families of representations we use theta lifts.
As a by-product of our investigations we obtain a characterization of non-generic representations. Namely, the following conditions are equivalent: $\pi$ is non-generic; the twisted Jacquet module $V_{N,\theta}$ is finite-dimensional for all non-degenerate $\theta$; the twisted Jacquet module $V_{N,\theta}$ is finite-dimensional for all split $\theta$; the $G^J$-module $V_{Z^J,\psi}$ is of finite length. See Theorem \ref{nongenchartheorem}.

Bessel functionals have important applications, and have been investigated in a number of works. For example, they can be used to define and study $L$-functions, in the case where the representation in question has no Whittaker model; e.g., \cite{Piatetski1997}, \cite{Sugano1984}, \cite{Furusawa1993}, \cite{PitaleSahaSchmidt2011}. As far as we know, the first works investigating Bessel functionals for irreducible, admissible representations of $\GSp(4,F)$ are \cite{NovoPia1973} and \cite{Novodvorski1973}. Both of these papers consider only representations with trivial central character. The main result of the first paper is the uniqueness of $(\Lambda,\theta)$-Bessel functionals for trivial $\Lambda$. In the second paper, this is generalized to arbitrary $\Lambda$.

If an irreducible, admissible representation $\pi$ admits a $(\Lambda,\theta)$-Bessel functional, then $\pi$ has an associated Bessel model. For unramified $\pi$ admitting a $(\Lambda,\theta)$-Bessel functional, the works  \cite{Sugano1984} and \cite{BumpFriedbergFurusawa1997} contain explicit formulas for the spherical vector in such a Bessel model. Other explicit formulas in certain cases of Iwahori-spherical representations appear in \cite{Saha2009}, \cite{Pitale2011} and \cite{PitaleSchmidt2012}. We note that these works show that all the values of a certain vector in the given Bessel model can be expressed in terms of data depending only on the representation and $\Lambda$ and $\theta$; in this situation it follows that the Bessel functional is unique. As far as we know, a detailed proof of uniqueness of Bessel functionals in all cases has not yet appeared in the literature.

In the case of odd residual characteristic, and when $\pi$ appears in a generic $L$-packet, the main local theorem of \cite{PrTa2011} gives an $\varepsilon$-factor criterion for the existence of a $(\Lambda,\theta)$-Bessel functional. There is some overlap between the methods of \cite{PrTa2011} and the present work. However, the goal of this work is to give a complete and ready account of Bessel functionals for all non-supercuspidal representations. We hope these results will be useful for applications where such specific knowledge is needed.

\section{Some definitions}
\label{defsec}
Throughout this work let $F$ be a non-archimedean local field of characteristic zero. Let $\bar F$ be a fixed algebraic closure of $F$. We fix a non-trivial character $\psi:\:F\rightarrow\C^\times$. The symbol $\OF$ denotes the ring of integers of $F$, and $\p$ is the maximal ideal of $\OF$. We let $\varpi$ be a fixed generator of $\p$. We denote by $|\cdot|$ the normalized absolute value on $F$, and by $\nu$ its restriction to $F^\times$. The Hilbert symbol of $F$ will be denoted by $(\cdot,\cdot)_F$. If $\Lambda$ is a character of a group, we denote by $\C_\Lambda$ the space of the one-dimensional representation whose action is given by $\Lambda$. If $x=\left[\begin{smallmatrix} a&b \\ c&d \end{smallmatrix} \right]$ is a $2 \times 2$ matrix, then we set $x^* = \left[\begin{smallmatrix} d&-b \\ -c&a \end{smallmatrix} \right]$.  If $X$ is an $l$-space, as in 1.1 of \cite{BeZe1976}, and $V$ is a complex vector space, then $\mathcal{S}(X,V)$ is the space of locally constant functions $X\to V$ with compact 
support. Let $G$ be an $l$-group, as in \cite{BeZe1976}, and let $H$ be a closed subgroup. If $\rho$ is a smooth representation of $H$, we define the compactly induced representation (unnormalized) $\cInd_H^G(\rho)$ as in 2.22 of \cite{BeZe1976}. If $(\pi,V)$ is a smooth representation of $G$, and if $\theta$ is a character of $H$, we define the twisted Jacquet module $V_{H,\theta}$ as the quotient $V/V(H,\theta)$, where $V(H,\theta)$ is the span of all vectors $\pi(h)v-\theta(h)v$ for all $h$ in $H$ and $v$ in $V$.

\subsection{Groups}
Let
$$
 \GSp(4,F)=\{g\in\GL(4,F):\:^tgJg=\lambda(g)J,\:\lambda(g)\in F^\times\},\qquad J=\begin{bmatrix}&&&1\\&&1\\&-1\\-1\end{bmatrix}.
$$
The scalar $\lambda(g)$ is called the \emph{multiplier} or \emph{similitude factor} of the matrix $g$. The \emph{Siegel parabolic subgroup} $P$ of $\GSp(4,F)$ consists of all matrices whose lower left $2\times2$ block is zero. For a matrix $A\in\GL(2,F)$ set
$$
 A'=\mat{}{1}{1}{}\,^t\!A^{-1}\mat{}{1}{1}{}.
$$
Then the Levi decomposition of $P$ is $P=MN$, where
\begin{equation}\label{Mdefeq}
 M=\{\mat{A}{}{}{\lambda A'}:\:A\in\GL(2,F),\:\lambda\in F^\times\},
\end{equation}
and
\begin{equation}\label{Ndefeq}
 N=\{\begin{bmatrix}
         1&&y&z\\&1&x&y\\&&1\\&&&1
        \end{bmatrix}:\;x,y,z\in F\}.
\end{equation}
Let $Q$ be the \emph{Klingen parabolic subgroup}, i.e.,
\begin{equation}\label{Qdefeq}
 Q=\GSp(4,F)\cap\begin{bmatrix} *&*&*&*\\&*&*&*\\&*&*&*\\&&&*\end{bmatrix}.
\end{equation}
The Levi decomposition for $Q$ is $Q=M_QN_Q$, where
\begin{equation}\label{MQdefeq}
 M_Q=\{\begin{bmatrix}t\\&A\\&&t^{-1}\det(A)\end{bmatrix}:\:A\in\GL(2,F),\:t\in F^\times\},
\end{equation}
and $N_Q$ is the \emph{Heisenberg group}
\begin{equation}\label{NQdefeq}
 N_Q=\{\begin{bmatrix}
         1&x&y&z\\&1&&y\\&&1&-x\\&&&1
        \end{bmatrix}:\;x,y,z\in F\}.
\end{equation}
The subgroup of $Q$ consisting of all elements with $t=1$ and $\det(A)=1$ is called the \emph{Jacobi group} and is denoted by $G^J$. The standard Borel subgroup of $\GSp(4,F)$ consists of all upper triangular matrices in $\GSp(4,F)$. We let
$$
 U=\GSp(4,F)\cap\begin{bmatrix}1&*&*&*\\&1&*&*\\&&1&*\\&&&1\end{bmatrix}
$$
be its unipotent radical.

The following elements of $\GSp(4,F)$ represent generators for the eight-element Weyl group, \begin{equation}\label{s1s2defeq}
 s_1=\begin{bmatrix}&1\\1\\&&&1\\&&1\end{bmatrix}\qquad\text{and}\qquad s_2=\begin{bmatrix}1\\&&1\\&-1\\&&&1\end{bmatrix}.
\end{equation}
\subsection{Representations}\label{representationssec}
For a smooth representation $\pi$ of $\GSp(4,F)$ or $\GL(2,F)$, we denote by $\pi^\vee$ its smooth contragredient.

For $c_1,c_2$ in $F^\times$, let $\psi_{c_1,c_2}$ be the character of $U$ defined by
\begin{equation}\label{psic1c2eq}
 \psi_{c_1,c_2}(\begin{bmatrix}1&x&*&*\\&1&y&*\\&&1&-x\\&&&1\end{bmatrix})=\psi(c_1x+c_2y).
\end{equation}
An irreducible, admissible representation $(\pi,V)$ of $\GSp(4,F)$ is called \emph{generic} if the space $\Hom_{U}(V,\psi_{c_1,c_2})$ is non-zero. This definition is independent of the choice of $c_1,c_2$. It is known by \cite{Rod1973} that, if non-zero, the space $\Hom_{U}(V,\psi_{c_1,c_2})$ is one-dimensional. Hence, $\pi$ can be realized in a unique way as a space of functions $W:\:\GSp(4,F)\rightarrow\C$ with the transformation property
$$
 W(ug)=\psi_{c_1,c_2}(u)W(g),\qquad u\in U,\;g\in \GSp(4,F),
$$
on which $\pi$ acts by right translations. We denote this model of $\pi$ by $\mathcal{W}(\pi,\psi_{c_1,c_2})$, and call it the \emph{Whittaker model} of $\pi$ with respect to $c_1,c_2$.

We will employ the notation of \cite{SaTa1993} for parabolically induced representations of $\GSp(4,F)$ (all parabolic induction is normalized). For details we refer to the summary given in Sect.\ 2.2 of \cite{NF}. Let $\chi_1$, $\chi_2$ and $\sigma$ be characters of $F^\times$. Then $\chi_1\times\chi_2\rtimes\sigma$ denotes the representation of $\GSp(4,F)$ parabolically induced from the character of the Borel subgroup which is trivial on $U$ and is given by
$$
 {\rm diag}(a,b,cb^{-1},ca^{-1})\longmapsto\chi_1(a)\chi_2(b)\sigma(c),\qquad a,b,c\in F^\times,
$$
on diagonal elements. Let $\sigma$ be a character of $F^\times$ and $\pi$ be an admissible representation of $\GL(2,F)$. Then $\pi\rtimes\sigma$ denotes the representation of $\GSp(4,F)$ parabolically induced from the representation
\begin{equation}\label{Prepeq}
 \mat{A}{*}{}{cA'}\longmapsto\sigma(c)\pi(A),\qquad A\in\GL(2,F),\:c\in F^\times,
\end{equation}
of the Siegel parabolic subgroup $P$. Let $\chi$ be a character of $F^\times$ and $\pi$ an admissible representation of $\GSp(2,F)\cong\GL(2,F)$. Then $\chi\rtimes\pi$ denotes the representation of $\GSp(4,F)$ parabolically induced from the representation
\begin{equation}\label{Qrepeq}
 \begin{bmatrix}t&*&*\\&g&*\\&&\det(g)t^{-1}\end{bmatrix}
 \longmapsto\chi(t)\pi(g),\qquad t\in F^\times,\:g\in\GL(2,F),
\end{equation}
of the Klingen parabolic subgroup $Q$.

For a character $\xi$ of $F^\times$ and a representation $(\pi,V)$ of $\GSp(4,F)$, the \emph{twist} $\xi\pi$ is the representation of $\GSp(4,F)$ on the same space $V$ given by $(\xi\pi)(g)=\xi(\lambda(g))\pi(g)$ for $g$ in $\GSp(4,F)$, where $\lambda$ is the multiplier homomorphism defined above. A similar definition applies to representations $\pi$ of $\GL(2,F)$; in this case, the multiplier is replaced by the determinant. The behavior of parabolically induced representations under twisting is as follows,
\begin{align*}
 \xi(\chi_1\times\chi_2\rtimes\sigma)&=\chi_1\times\chi_2\rtimes\xi\sigma,\\
 \xi(\pi\rtimes\sigma)&=\pi\rtimes\xi\sigma,\\
 \xi(\chi\rtimes\pi)&=\chi\rtimes\xi\pi.
\end{align*}

The irreducible constituents of \emph{all} parabolically induced representations of $\GSp(4,F)$ have been determined in \cite{SaTa1993}. The following table, which is essentially a reproduction of Table A.1 of \cite{NF}, provides a summary of these irreducible constituents. In the table, $\chi,\chi_1,\chi_2,\xi$ and $\sigma$ stand for characters of $F^\times$; the symbol $\nu$ denotes the normalized absolute value; $\pi$ stands for an irreducible, admissible, supercuspidal representation of $\GL(2,F)$, and $\omega_\pi$ denotes the central character of $\pi$. The trivial character of $F^\times$ is denoted by $1_{F^\times}$, the trivial representation of $\GL(2,F)$ by $\triv_{\GL(2)}$ or $\triv_{\GSp(2)}$, depending on the context, the trivial representation of $\GSp(4,F)$ by $\triv_{\GSp(4)}$, the Steinberg representation of $\GL(2,F)$ by $\St_{\GL(2)}$ or $\St_{\GSp(2)}$, depending on the context, and the Steinberg representation of $\GSp(4,F)$ by $\St_{\GSp(4)}$. The names of the representations given in 
the ``representation'' column are taken from \cite{SaTa1993}. The ``tempered'' column indicates the condition on the inducing data under which a representation is tempered. The ``$L^2$'' column indicates which representations are square integrable after an appropriate twist. Finally, the ``g'' column indicates which representations are generic.

In addition to all irreducible, admissible, non-supercuspidal representations, the table also includes two classes of supercuspidal representations denoted by Va$^*$ and XIa$^*$. The reason that these supercuspidal representations are included in the table is that they are in $L$-packets with some non-supercuspidal representations. Namely, the Va representation $\delta([\xi,\nu\xi],\nu^{-1/2}\sigma)$ and the Va$^*$ representation $\delta^*([\xi,\nu\xi],\nu^{-1/2}\sigma)$ form an $L$-packet, and the XIa representation $\delta(\nu^{1/2}\pi,\nu^{-1/2}\sigma)$ and the XIa$^*$ representation $\delta^*(\nu^{1/2}\pi,\nu^{-1/2}\sigma)$ form an $L$-packet; see the paper \cite{GaTa2011}. Incidentally, the other non-singleton $L$-packets involving non-supercuspidal representations are the two-element packets $\{\tau(S,\nu^{-1/2}\sigma),\tau(T,\nu^{-1/2}\sigma)\}$ (type VIa and VIb), as well as $\{\tau(S,\pi),\tau(T,\pi)\}$ (type VIIIa and VIIIb).

$$
\renewcommand{\arraystretch}{1.19}
\setlength{\arraycolsep}{0.3cm}
 \begin{array}{ccccccccc}
  \toprule
   &\mbox{constituents of}&&\mbox{representation}
   &{\rm tempered}&L^2\!&\,{\rm g}\\
  \toprule
  {\rm I}&\multicolumn{3}{c}{\chi_1\times\chi_2\rtimes\sigma\quad
   \mbox{(irreducible)}}&\mbox{$\chi_i,\sigma$ unitary}
   &&\bullet\\
  \midrule
  {\rm II}&\nu^{1/2}\chi\times\nu^{-1/2}\chi\rtimes\sigma&\mbox{a}
   &\chi\St_{\GL(2)}\rtimes\sigma&\mbox{$\chi,\sigma$ unitary}&&\bullet\\
  \cmidrule{3-7}
   &(\chi^2\neq\nu^{\pm1},\chi\neq\nu^{\pm 3/2})&\mbox{b}
   &\chi\triv_{\GL(2)}\rtimes\sigma
   &&&\\
  \midrule
  {\rm III}&\chi\times\nu\rtimes\nu^{-1/2}\sigma&\mbox{a}
   &\chi\rtimes\sigma\St_{\GSp(2)}&\mbox{$\chi,\sigma$ unitary}&
   &\bullet\\
  \cmidrule{3-7}
  &(\chi\notin\{1,\nu^{\pm2}\})&\mbox{b}
   &\chi\rtimes\sigma\triv_{\GSp(2)}
   &&&\\\midrule
  {\rm IV}&\nu^2\times\nu\rtimes\nu^{-3/2}\sigma&\mbox{a}&\sigma\St_{\GSp(4)}&\mbox{$\sigma$ unitary}&\bullet&\bullet\\
  \cmidrule{3-7}
  &&\mbox{b}&L(\nu^2,\nu^{-1}\sigma\St_{\GSp(2)})&&&\\
  \cmidrule{3-7}
  &&\mbox{c}&L(\nu^{3/2}\St_{\GL(2)},\nu^{-3/2}\sigma)&&&\\
  \cmidrule{3-7}
  &&\mbox{d}&\sigma\triv_{\GSp(4)}&&&\\
  \midrule
  {\rm V}&\nu\xi\times\xi\rtimes\nu^{-1/2}\sigma&\mbox{a}
   &\delta([\xi,\nu\xi],\nu^{-1/2}\sigma)&\mbox{$\sigma$ unitary}&\bullet&\bullet\\
  \cmidrule{3-7}
  &(\xi^2=1,\:\xi\neq1)&\mbox{b}&L(\nu^{1/2}\xi\St_{\GL(2)},\nu^{-1/2}\sigma)&&&\\
  \cmidrule{3-7}
  &&\mbox{c}&L(\nu^{1/2}\xi\St_{\GL(2)},\xi\nu^{-1/2}\sigma)&&&\\
  \cmidrule{3-7}
  &&\mbox{d}&L(\nu\xi,\xi\rtimes\nu^{-1/2}\sigma)&&&\\
  \midrule
  {\rm VI}&\nu\times1_{F^\times}\rtimes\nu^{-1/2}\sigma&\mbox{a}
   &\tau(S,\nu^{-1/2}\sigma)&\mbox{$\sigma$ unitary}&&\bullet\\
  \cmidrule{3-7}
  &&\mbox{b}&\tau(T,\nu^{-1/2}\sigma)&\mbox{$\sigma$ unitary}&&\\
  \cmidrule{3-7}
  &&\mbox{c}&L(\nu^{1/2}\St_{\GL(2)},\nu^{-1/2}\sigma)&&&\\
  \cmidrule{3-7}
  &&\mbox{d}&L(\nu,1_{F^\times}\rtimes\nu^{-1/2}\sigma)&&&\\
  \toprule
  {\rm VII}&\multicolumn{3}{c}{\chi\rtimes\pi\quad
   \mbox{(irreducible)}}&\mbox{$\chi,\pi$ unitary}&&\bullet\\
  \midrule
  {\rm VIII}&1_{F^\times}\rtimes\pi&\mbox{a}&\tau(S,\pi)&\mbox{$\pi$ unitary}&&\bullet\\
  \cmidrule{3-7}
  &&\mbox{b}&\tau(T,\pi)&\mbox{$\pi$ unitary}&&\\
  \midrule
  {\rm IX}&\nu\xi\rtimes\nu^{-1/2}\pi&\mbox{a}
   &\delta(\nu\xi,\nu^{-1/2}\pi)&\mbox{$\pi$ unitary}&\bullet&\bullet\\
  \cmidrule{3-7}
  &(\xi\neq1,\:\xi\pi=\pi)&\mbox{b}&L(\nu\xi,\nu^{-1/2}\pi)&&&\\
  \toprule
  {\rm X}&\multicolumn{3}{c}{\pi\rtimes\sigma\quad
   \mbox{(irreducible)}}&\mbox{$\pi,\sigma$ unitary}
   &&\bullet\\
  \midrule
  {\rm XI}&\nu^{1/2}\pi\rtimes\nu^{-1/2}\sigma&\mbox{a}
   &\delta(\nu^{1/2}\pi,\nu^{-1/2}\sigma)
   &\mbox{$\pi,\sigma$ unitary}&\bullet&\bullet\\
  \cmidrule{3-7}
   &(\omega_\pi=1)&\mbox{b}&L(\nu^{1/2}\pi,\nu^{-1/2}\sigma)&&&\\
  \toprule
  {\rm Va^*}&\mbox{(supercuspidal)}&&\delta^*([\xi,\nu\xi],\nu^{-1/2}\sigma)
   &\mbox{$\sigma$ unitary}&\bullet&\\
  \midrule
  {\rm XIa^*}&\mbox{(supercuspidal)}&
   &\delta^*(\nu^{1/2}\pi,\nu^{-1/2}\sigma)
   &\mbox{$\pi,\sigma$ unitary}&\bullet&\\
  \toprule
 \end{array}
$$
\section{Generalities on Bessel functionals}
In this section we gather some definitions, notation, and basic results about Bessel functionals. 
\subsection{Quadratic extensions}
\label{quadextsubsec}
Let $D \in F^\times$. If $D \notin F^{\times 2}$, then let $\sroot = \sqrt{D}$ be a square root of $D$ in $\bar F$, and  $L=F(\sroot)$. If $D \in F^{\times 2}$, then let $\sqrt{D}$ be a square root of $D$ in $F^\times$,  $L = F\times F$, and  $\sroot = (-\sqrt{D},\sqrt{D}) \in L$. In both cases $L$ is a two-dimensional $F$-algebra containing $F$, $L=F+F \sroot$, and $\sroot^2=D$. We will abuse terminology slightly, and refer to $L$ as the \emph{quadratic extension associated to $D$}. We define a map $\gamma :L \to L$ called \emph{Galois conjugation} by $\gamma(x+y\sroot) = x-y\sroot$. Then $\gamma(xy)=\gamma(x)\gamma(y)$ and $\gamma(x+y)=\gamma(x)+\gamma(y)$ for $x,y \in L$, and the fixed points of $\gamma$ are the elements of $F$. The group $\Gal(L/F)$ of $F$-automorphisms $\alpha:L \to L$ is $\{1,\gamma\}$. We define norm and trace functions
$\Norm_{L/F}:L \to F$ and $\Trace_{L/F}:L \to F$ by $\Norm_{L/F}(x) = x \gamma (x)$ and $\Trace_{L/F}(x) = x + \gamma (x)$ for $x \in L$. We let $\chi_{L/F}$ be the quadratic character associated to $L/F$, so that $\chi_{L/F}(x)=(x,D)_F$ for $x \in F^\times$. 
\subsection{\texorpdfstring{$2\times 2$}{} symmetric matrices}
\label{twobytwosubsec}
Let $a,b,c \in F$ and set
\begin{equation}
\label{Sdefeq}
S = \mat{a}{b/2}{b/2}{c}.
\end{equation}
Let $D=b^2/4-ac= -\det(S)$. Assume that $D \neq 0$. The \emph{discriminant} $\disc(S)$ of $S$ is the class in $F^\times/F^{\times2}$ determined by $D$. 
It is known that there exists $g \in \GL(2,F)$ such that ${}^t g S g $ is of the form  $\left[\begin{smallmatrix} a_1 & \\ & a_2 \end{smallmatrix} \right]$ and that $(a_1,a_2)_F$ is independent of the choice of $g$ such that ${}^t g S g $ is diagonal; we
define the \emph{Hasse invariant} $\varepsilon(S) \in \{ \pm 1\}$ by $\varepsilon(S) = (a_1,a_2)_F$. In fact, one has:
$$
\renewcommand{\arraystretch}{1.2}
\begin{array}{ccccc}
\toprule
S & g & {}^tgSg & \disc(S) & \varepsilon(S) \\
\midrule
a \neq 0, c \neq 0  & \left[ \begin{smallmatrix} 1 & \frac{-b}{2a} \\ & 1 \end{smallmatrix} \right] & \left[ \begin{smallmatrix} a & \\ & c-\frac{b^2}{4a} \end{smallmatrix} \right]  & (b^2/4-ac)F^{\times 2} & (a, b^2/4-ac)_F =(c,b^2/4-ac)_F\\
\cmidrule{1-5}
a \neq 0, c = 0  & \left[ \begin{smallmatrix} 1 & \frac{-b}{2a} \\ & 1 \end{smallmatrix} \right] & \left[ \begin{smallmatrix} a & \\ & -\frac{b^2}{4a} \end{smallmatrix} \right] & F^{\times 2}& 1 \\
\cmidrule{1-5}
a=0, c \neq 0 &  \left[ \begin{smallmatrix} &1\\ 1& -\frac{b}{2c} \end{smallmatrix} \right] & \left[\begin{smallmatrix} c & \\ &-\frac{b^2}{4c}\end{smallmatrix}\right]&F^{\times 2} & 1 \\
\cmidrule{1-5}
a=0, c=0 & \left[ \begin{smallmatrix} 1& 1 \\ 1&-1 \end{smallmatrix} \right] & \left[ \begin{smallmatrix} b & \\ & -b \end{smallmatrix} \right] &F^{\times 2}& 1 \\
\bottomrule
\end{array}
$$
If $\disc(S)=F^{\times 2}$, then we say that $S$ is \emph{split}. If $S$ is split, then for any $\lambda \in F^\times$ there exists $g \in \GL(2,F)$ such that ${}^tgSg=\left[\begin{smallmatrix} & \lambda \\ \lambda & \end{smallmatrix} \right]$.
\subsection{Another \texorpdfstring{$F$}{}-algebra}
\label{anotheralgebrasubsec}
Let $S$ be as in \eqref{Sdefeq} with $\disc(S) \neq 0$. Set $D=b^2/4-ac$. We define
\begin{equation}\label{AFdefeq}
 A= A_S= \{ \mat{x-yb/2}{-ya}{yc}{x+yb/2} : x,y \in F \}.
\end{equation}
Then, with respect to matrix addition and multiplication, $A$ is a two-dimensional $F$-algebra naturally containing $F$. One can verify that
\begin{equation}\label{AFdefeq2}
 A=\mat{}{1}{1}{}\{g\in \Mat_2(F):\:^tgSg=\det(g)S\}\mat{}{1}{1}{}.
\end{equation}
We define $T=T_S = A^\times$. Let $L$ be the quadratic extension associated to $D$; we also say that
$L$ is the \emph{quadratic extension associated to $S$}.
We define an isomorphism of $F$-algebras,
\begin{equation}
\label{ALisoeq}
 A \stackrel{\sim}{\longrightarrow} L, \qquad \mat{x-yb/2}{-ya}{yc}{x+yb/2} \longmapsto x + y \sroot. 
\end{equation}
The restriction of this isomorphism to $T$ is an isomorphism $T \stackrel{\sim}{\longrightarrow} L^\times$, and  we identify characters of $T$ and characters of $L^\times$ via this isomorphism. The automorphism of $A$ corresponding to the automorphism $\gamma$ of $L$ will also be denoted by $\gamma$. It has the effect of replacing $y$ by $-y$ in the matrix \eqref{AFdefeq}. We have $\det (t) = \Norm_{L/F}(t)$ for $t \in A$, where we identify elements of $A$ and $L$ via \eqref{ALisoeq}.
\begin{lemma}\label{TFGL2lemma}
 Let $T$ be as above, and assume that $L$ is a field. Let $B_2$ be the group of upper triangular matrices in $\GL(2,F)$. Then $TB_2=\GL(2,F)$.
\end{lemma}
\begin{proof}
This can easily be verified using the explicit form of the matrices in $T$ and the assumption $D \notin F^{\times 2}$.
\end{proof}
\subsection{Bessel functionals}\label{besselsec}
Let $a,b$ and $c$ be in $F$. Define $S$ as in \eqref{Sdefeq}, and define a character $\theta=\theta_{a,b,c}=\theta_S$ of $N$ by
\begin{equation}\label{thetaSsetupeq}
 \theta(\begin{bmatrix} 1 &&y&z \\ &1&x&y \\ &&1& \\ &&&1 \end{bmatrix}) = \psi(ax+by+cz) = \psi (\trace(S\mat{}{1}{1}{}\mat{y}{z}{x}{y}))
\end{equation}
for $x,y,z \in F$. Every character of $N$ is of this form for uniquely determined $a,b,c$ in $F$, or, alternatively, for a uniquely determined symmetric $2\times2$ matrix $S$. We say that $\theta$ is \emph{non-degenerate} if $\det(S)\neq0$. Given $S$ with $\det(S)\neq0$, let $A$ be as in \eqref{AFdefeq}, and let $T=A^\times$. We embed $T$ into $\GSp(4,F)$ via the map defined by
\begin{equation}\label{Tembeddingeq}
 t\longmapsto \mat{t}{}{}{\det(t)t'},\qquad t\in T.
\end{equation}
The image of $T$ in $\GSp(4,F)$ will also be denoted by $T$; the usage should be clear from the context. For $t\in T$ we have $\lambda(t)=\det(t)=\Norm_{L/F}(t)$. It is easily verified that
$$
 \theta(tnt^{-1})=\theta(n)\qquad\text{for $n\in N$ and $t\in T$}.
$$
We refer to the semidirect product
\begin{equation}\label{Ddefeq}
 D=TN
\end{equation}
as the \emph{Bessel subgroup} defined by character $\theta$ (or, the matrix $S$). Given a character $\Lambda$ of $T$ (identified with a character of $L^\times$ as explained above), we can define a character $\Lambda\otimes\theta$ of $D$ by
$$
 (\Lambda\otimes\theta)(tn)=\Lambda(t)\theta(n)\qquad\text{for $n\in N$ and $t\in T$}.
$$
Every character of $D$ whose restriction to $N$ coincides with $\theta$ is of this form for an appropriate $\Lambda$.

Now let $(\pi,V)$ be an admissible representation of $\GSp(4,F)$. Let $\theta$ be a non-degenerate character of $N$, and let $\Lambda$ be a character of the associated group $T$. We say that $\pi$ admits a \emph{$(\Lambda,\theta)$-Bessel functional} if $\Hom_{D}(V,\C_{\Lambda\otimes\theta})\neq0$. A non-zero element $\beta$ of $\Hom_{D}(V,\C_{\Lambda\otimes\theta})$ is called a \emph{$(\Lambda,\theta)$-Bessel functional} for $\pi$. If such a $\beta$ exists, then $\pi$ admits a model consisting of functions $B:\:\GSp(4,F)\rightarrow\C$ with the Bessel transformation property
$$
 B(tng)=\Lambda(t)\theta(n)B(g)\qquad\text{for $t\in T$, $n\in N$ and $g\in\GSp(4,F)$},
$$
by associating to each $v$ in $V$ the function $B_v$ that is defined by $B_v(g)=\beta(\pi(g)v)$ for $g \in \GSp(4,F)$. We note that if $\pi$ admits a central character $\omega_\pi$ and a $(\Lambda,\theta)$-Bessel functional, then $\Lambda|_{F^\times}=\omega_\pi$. For a character $\sigma$ of $F^\times$, it is easy to verify that
\begin{equation}\label{besseltwistformula}
 \Hom_D(\pi,\C_{\Lambda\otimes\theta})=\Hom_D(\sigma\pi,\C_{(\sigma\circ\Norm_{L/F})\Lambda\otimes\theta}).
\end{equation}
If $\pi$ is irreducible, then, using that $\pi^\vee\cong\omega_\pi^{-1}\pi$ (Proposition 2.3 of \cite{Takloo-Bighash2000}), one can also verify that
\begin{equation}\label{besselcontragredientformula}
 \Hom_D(\pi,\C_{\Lambda\otimes\theta})\cong\Hom_D(\pi^\vee,\C_{(\Lambda\circ\gamma)^{-1}\otimes\theta}).
\end{equation}

The twisted Jacquet module of $V$ with respect to $N$ and $\theta$ is the quotient $V_{N,\theta}=V/V(N,\theta)$, where $V(N,\theta)$ is the subspace spanned by all vectors $\pi(n)v-\theta(n)v$ for $v$ in $V$ and $n$ in $N$. This Jacquet module carries an action of $T$ induced by the representation $\pi$. Evidently, there is a natural isomorphism
\begin{equation}\label{DTJacqueteq}
 \Hom_{D}(V,\C_{\Lambda\otimes\theta})\cong\Hom_{T}(V_{N,\theta},\C_\Lambda).
\end{equation}
Hence, when calculating the possible Bessel functionals on a representation $(\pi,V)$, a first step often consists in calculating the Jacquet modules $V_{N,\theta}$. We will use this method to calculate the possible Bessel functionals for most of the non-supercuspidal, irreducible, admissible representations of $\GSp(4,F)$. The few representations that are inaccessible with this method will be treated using the theta correspondence.

In this paper we do not assume that $(\Lambda,\theta)$-Bessel functionals are unique up to scalars. See Sect.\ \ref{uniquenesssec} for some remarks on uniqueness.
\subsection{Action on Bessel functionals}
\label{actionsubsec}
There is an action of $M$, defined in \eqref{Mdefeq}, on the set of Bessel functionals. Let $(\pi,V)$ be an irreducible, admissible representation of $\GSp(4,F)$, and let $\beta: V \to \C$ be a $(\Lambda, \theta)$-Bessel functional for $\pi$. Let $a,b,c\in F$ be
such that \eqref{thetaSsetupeq} holds. Let $m \in M$, with
$$
m = \begin{bmatrix} g & \\ & \lambda g' \end{bmatrix},
$$
where $\lambda \in F^\times$ and $g \in \GL(2,F)$. Define $m \cdot \beta : V \to \C$ by $(m\cdot \beta)(v) = \beta (\pi (m^{-1})v)$ for $v\in V$. Calculations show that $m\cdot \beta$ is a $(\Lambda',\theta')$-Bessel functional with $\theta'$ defined by 
$$
\theta'(\begin{bmatrix} 1 &&y&z \\ &1&x&y \\ &&1& \\ &&&1 \end{bmatrix}) = \psi(a'x+b'y+c'z) = \psi (\trace(S'\mat{}{1}{1}{}\mat{y}{z}{x}{y})), \quad x,y,z \in F, 
$$
where
$$
S' = \mat{a'}{b'/2}{b'/2}{c'} = \lambda\,^t h S h\quad \text{with}\quad  h=\mat{}{1}{1}{} g^{-1} \mat{}{1}{1}{}.
$$
Since $\disc(S') = \disc (S)$, the  quadratic extension $L'$ associated to $S'$ is the same as the quadratic extension $L$ associated to  $S$. 
There is an isomorphism of $F$-algebras 
$$
A'=A_{S'} \stackrel{\sim}{\longrightarrow} A= A_S, \qquad a \mapsto  g^{-1} a  g.
$$
Let $T'  = A'{}^\times$. Finally, $\Lambda':T' \to \C^\times$ is given by $\Lambda'(t') = \Lambda(g^{-1}t'g)$ for $t' \in T'$. 

For example, assume that $\beta'$ is a \emph{split} Bessel functional, i.e., a Bessel functional for which the discriminant of the associated symmetric matrix $S'$ is the class $F^{\times2}$. By Sect.~\ref{twobytwosubsec} there exists $m$ as above such that $\beta'=m\cdot \beta$, where the symmetric matrix $S$ associated to the $(\Lambda,\theta)$-Bessel functional $\beta$ is
\begin{equation}\label{splitSeq}
 S = \mat{}{1/2}{1/2}{},
\end{equation}
and
\begin{equation}\label{splitthetaeq}
 \theta(\begin{bmatrix} 1 &&y&z \\ &1&x&y \\ &&1& \\ &&&1 \end{bmatrix}) = \psi(y).
\end{equation}
In this case
\begin{equation}\label{splitthetaTeq}
 T=T_S=\{\begin{bmatrix}a\\&b\\&&a\\&&&b\end{bmatrix}:\:a,b\in F^\times\}.
\end{equation}
Sometimes when working with split Bessel functionals it is more convenient to work with the conjugate group
\begin{equation}\label{splitthetaconjNeq}
 N_{\mathrm{alt}}=s_2^{-1}Ns_2=\begin{bmatrix}1&*&&*\\&1\\&*&1&*\\&&&1\end{bmatrix}
\end{equation}
and the conjugate character
\begin{equation}\label{splitthetaconjeq}
 \theta_{\mathrm{alt}}(\begin{bmatrix}1&-y&&z\\&1&&\\&x&1&y\\&&&1\end{bmatrix})=\psi(y).
\end{equation}
In this case the stabilizer of $\theta_{\mathrm{alt}}$ is
\begin{equation}\label{splitthetaconjTeq}
 T_{\mathrm{alt}}=\{\begin{bmatrix}a\\&a\\&&b\\&&&b\end{bmatrix}:\:a,b\in F^\times\}.
\end{equation}
\subsection{Galois conjugation of Bessel functionals} \label{galoissubsec}
The action of $M$  can be used to define the Galois conjugate of a Bessel functional. Let $S$ be as in \eqref{Sdefeq}, and let $A=A_S$ and $T=T_S$. Define
\begin{equation}
\label{hgammaeq}
h_\gamma=
\left\{ 
\begin{array}{ll}
\left[\begin{matrix} 1& b/a\\ & -1 \end{matrix} \right] & \text{if $a \neq 0$}, \\
\left[\begin{matrix} 1& \\ -b/c& -1 \end{matrix} \right] & \text{if $a = 0$ and $c \neq 0$}, \\
\left[\begin{matrix} &1 \\ 1 &  \end{matrix} \right] & \text{if $a=c=0$}.
\end{array}
\right.
\end{equation}
Then $h_\gamma \in \GL(2,F)$, $h_\gamma^2=1$, $S={}^t h_\gamma S h_\gamma $ and $\det (h_\gamma) = -1$.
Set 
$$
g_\gamma = \mat{}{1}{1}{} h_\gamma^{-1} \mat{}{1}{1}{}= \mat{}{1}{1}{} h_\gamma \mat{}{1}{1}{} \in \GL(2,F), \quad m_\gamma = \begin{bmatrix} g_\gamma & \\ & g_\gamma' \end{bmatrix} \in M.
$$
We have $g_\gamma Tg_\gamma^{-1} = T$, and the diagrams
$$
\begin{CD}
A @>\sim>> L \\
@V\text{conjugation by $g_\gamma$}VV @VV\gamma V\\
A @>\sim>> L
\end{CD}
\qquad\qquad\qquad\qquad\qquad\qquad
\begin{CD}
T @>\sim>> L^\times \\
@V\text{conjugation by $g_\gamma$}VV @VV\gamma V\\
T @>\sim>> L^\times
\end{CD}
$$
commute. Let $(\pi,V)$ be an irreducible, admissible representation of $\GSp(4,F)$, and let $\beta$ be a $(\Lambda,\theta)$-Bessel functional for $\pi$. We refer to $m_\gamma \cdot\beta$ as the \emph{Galois conjugate} of $\beta$. We note that $m_\gamma\cdot \beta$
is a $(\Lambda\circ \gamma, \theta)$-Bessel functional for $\pi$. Hence,
\begin{equation}\label{besselgaloiseq}
 \Hom_D(\pi,\C_{\Lambda\otimes\theta})\cong\Hom_D(\pi,\C_{(\Lambda\circ\gamma)\otimes\theta}).
\end{equation}
In combination with \eqref{besselcontragredientformula}, we get
\begin{equation}\label{besselgaloiseq2}
 \Hom_D(\pi,\C_{\Lambda\otimes\theta})\cong\Hom_D(\pi^\vee,\C_{\Lambda^{-1}\otimes\theta}).
\end{equation}
\subsection{Waldspurger functionals}
\label{waldfuncsubsec}
Our analysis of Bessel functionals will often involve a similar type of functional on representations of $\GL(2,F)$. Let $\theta$ and $S$ be as in \eqref{thetaSsetupeq}, and let $T\cong L^\times$ be the associated subgroup of $\GL(2,F)$. Let $\Lambda$ be a character of $T$. Let $(\pi,V)$ be an irreducible, admissible representation of $\GL(2,F)$. A \emph{$(\Lambda,\theta)$-Waldspurger functional} on $\pi$ is a non-zero linear map $\beta:\:V\rightarrow\C$ such that
$$
 \beta(\pi(g)v)=\Lambda(g)\beta(v)\qquad\text{for all }v\in V\text{ and }g\in T.
$$
For trivial $\Lambda$, such functionals were the subject of Proposition 9 of \cite{Wald1980} and Proposition 8 of \cite{Wald1985}. For general $\Lambda$ see \cite{Tu1983}, \cite{Saito1993} and Lemme 8 of \cite{Wald1985}. The $(\Lambda,\theta)$-Waldspurger functionals are the non-zero elements of the space ${\rm Hom}_T(\pi,\C_\Lambda)$, and it is known that this space is at most one-dimensional. An obvious necessary condition for ${\rm Hom}_T(\pi,\C_\Lambda)\neq0$ is that $\Lambda\big|_{F^\times}$ equals $\omega_\pi$, the central character of $\pi$. By Sect.\ \ref{galoissubsec}, Galois conjugation on $T$ is given by conjugation by an element of $\GL(2,F)$. Hence,
\begin{equation}\label{waldspurgerpropeq1}
 \Hom_T(\pi,\C_\Lambda)\cong\Hom_T(\pi,\C_{\Lambda\circ\gamma}).
\end{equation}
Using $\pi^\vee\cong\omega_\pi^{-1}\pi$, one verifies that
\begin{equation}\label{waldspurgerpropeq3}
 \Hom_T(\pi,\C_\Lambda)\cong\Hom_T(\pi^\vee,\C_{(\Lambda\circ\gamma)^{-1}}).
\end{equation}
In combination with \eqref{waldspurgerpropeq1}, we also have
\begin{equation}\label{waldspurgerpropeq4}
 \Hom_T(\pi,\C_\Lambda)\cong\Hom_T(\pi^\vee,\C_{\Lambda^{-1}}).
\end{equation}
Let $\pi^{\mathrm{JL}}$ denote the Jacquet-Langlands lifting of $\pi$ in the case that $\pi$ is a discrete series representation, and $0$ otherwise. Then, by the discussion on p.~1297 of \cite{Tu1983},
\begin{equation}\label{waldspurgerpropeq2}
 \dim\Hom_T(\pi,\C_\Lambda)+\dim\Hom_T(\pi^{\mathrm{JL}},\C_\Lambda)=1.
\end{equation}
It is easy to see that, for any character $\sigma$ of $F^\times$,
\begin{equation}\label{Waldspurgertwistingeq}
 {\rm Hom}_T(\pi,\C_\Lambda)={\rm Hom}_T(\sigma\pi,\C_{(\sigma\circ \Norm_{L/F})\Lambda}).
\end{equation}
For $\Lambda$ such that $\Lambda\big|_{F^\times}=\sigma^2$, it is known that
\begin{equation}\label{StGL2Waldspurgereq2}
 \dim({\rm Hom}_T(\sigma\St_{\GL(2)},\C_\Lambda))=\left\{\begin{array}{l@{\qquad}l}
 0&\text{ if $L$ is a field and $\Lambda=\sigma\circ \Norm_{L/F}$},\\
 1&\text{ otherwise};\end{array}\right.
\end{equation}
see Proposition 1.7 and Theorem 2.4 of \cite{Tu1983}. As in the case of Bessel functionals, we call a Waldspurger functional \emph{split} if the discriminant of the associated matrix $S$ lies in $F^{\times2}$. By Lemme 8 of \cite{Wald1985}, an irreducible, admissible, infinite-dimensional representation of $\GL(2,F)$ admits a split $(\Lambda,\theta)$-Waldspurger functional with respect to any character $\Lambda$ of $T$ that satisfies $\Lambda\big|_{F^\times}=\omega_\pi$ (this can also be proved in a way analogous to the proof of Proposition \ref{GSp4genericprop} below, utilizing the standard zeta integrals for $\GL(2)$).
\section{Split Bessel functionals}
Irreducible, admissible, generic representations of $\GSp(4,F)$ admit a theory of zeta integrals, and every zeta integral gives rise to a split Bessel functional. As a consequence, generic representations admit \emph{all} possible split Bessel functionals; see Proposition \ref{GSp4genericprop} below for a precise formulation.

To put the theory of zeta integrals on a solid foundation, we will use $P_3$-theory. The group $P_3$, defined below, plays a role in the representation theory of $\GSp(4)$ similar to the ``mirabolic'' subgroup in the theory for $\GL(n)$. Some of what follows is a generalization of Sects.\ 2.5 and 2.6 of \cite{NF}, where $P_3$-theory was developed under the assumption of trivial central character. The general case requires only minimal modifications.

While every generic representation admits split Bessel functionals, we will see that the converse is not true. $P_3$-theory can also be used to identify the non-generic representations that admit a split Bessel functional. This is explained in Sect.\ \ref{splitbesselnongenericsec} below.
\subsection{The group \texorpdfstring{$P_3$}{} and its representations}
Let $P_3$ be the subgroup of $\GL(3,F)$ defined as the intersection
$$
 P_3=\GL(3,F)\cap\begin{bmatrix} *&*&*\\ *&*&*\\&&1\end{bmatrix}.
$$
We recall some facts about this group, following \cite{BeZe1976}. Let
$$
 U_3=P_3\cap\begin{bmatrix}1&*&*\\&1&*\\&&1\end{bmatrix},\qquad N_3=P_3\cap\begin{bmatrix}1&&*\\&1&*\\&&1\end{bmatrix}.
$$
We define characters $\Theta$ and $\Theta'$ of $U_3$ by
$$
 \Theta(\begin{bmatrix}1&u_{12}&*\\&1&u_{23}\\&&1\end{bmatrix})=\psi(u_{12}+u_{23}),\qquad
 \Theta'(\begin{bmatrix}1&u_{12}&*\\&1&u_{23}\\&&1\end{bmatrix})=\psi(u_{23}).
$$
If $(\pi,V)$ is a smooth representation of $P_3$, we may consider the twisted Jacquet modules
$$
 V_{U_3,\Theta}=V/V(U_3,\Theta),\qquad V_{U_3,\Theta'}=V/V(U_3,\Theta')
$$
where $V(U_3,\Theta)$ (resp.\ $V(U_3,\Theta')$) is spanned by all elements of the form $\pi(u)v-\Theta(u)v$ (resp.\ $\pi(u)v-\Theta'(u)v$) for $v$ in $V$ and $u$ in $U_3$. Note that $V_{U_3,\Theta'}$ carries an action of the subgroup
$$
 \begin{bmatrix} *&&\\&1\\&&1\end{bmatrix}\cong F^\times
$$
of $P_3$. We may also consider the Jacquet module $V_{N_3}=V/V(N_3)$, where $V(N_3)$ is the space spanned by all vectors of the form $\pi(u)v-v$ for $v$ in $V$ and $u$ in $N_3$. Note that $V_{N_3}$ carries an action of the subgroup
$$
 \begin{bmatrix} *&*&\\ *&*\\&&1\end{bmatrix}\cong\GL(2,F)
$$
of $P_3$.

Next we define three classes of smooth representations of $P_3$, associated with the groups $\GL(0)$, $\GL(1)$ and $\GL(2)$. Let
\begin{equation}\label{tauP3GL0eq}
 \tau^{P_3}_{\GL(0)}(1):=\cInd^{P_3}_{U_3}(\Theta),
\end{equation}
where $\cInd$ denotes compact induction. Then $\tau^{P_3}_{\GL(0)}(1)$ is a smooth, irreducible representation of $P_3$. Next, let $\chi$ be a smooth representation of $\GL(1,F)\cong F^\times$. Define a representation $\chi\otimes\Theta'$ of the subgroup
$$
 \begin{bmatrix}*&*&*\\&1&*\\&&1\end{bmatrix}
$$
of $P_3$ by
$$
 (\chi\otimes\Theta')(\begin{bmatrix}a&*&*\\&1&y\\&&1\end{bmatrix})=\chi(a)\psi(y).
$$
Then
$$
 \tau^{P_3}_{\GL(1)}(\chi):=\cInd^{P_3}_{\left[\begin{smallmatrix} *&*&*\\&1&*\\&&1\end{smallmatrix}\right]}(\chi\otimes\Theta')
$$
is a smooth representation of $P_3$. It is irreducible if and only if $\chi$ is one-dimensional. Finally, let $\rho$ be a smooth representation of $\GL(2,F)$. We define the representation $\tau_{\GL(2)}^{P_3}(\rho)$ of $P_3$ to have the same space as $\rho$, and action given by
\begin{equation}\label{tauP3GL2eq}
 \tau_{\GL(2)}^{P_3}(\rho)(\begin{bmatrix}a&b&*\\c&d&*\\&&1\end{bmatrix})=\rho(\mat{a}{b}{c}{d}).
\end{equation}
Evidently, $\tau_{\GL(2)}^{P_3}(\rho)$ is irreducible if and only if $\rho$ is irreducible.

\begin{proposition}\label{P3representationsprop}
 Let notations be as above.
 \begin{enumerate}
  \item Every irreducible, smooth representation of $P_3$ is isomorphic to exactly one of
   $$
    \tau^{P_3}_{\GL(0)}(1),\qquad
    \tau^{P_3}_{\GL(1)}(\chi),\qquad
    \tau_{\GL(2)}^{P_3}(\rho),
   $$
   where $\chi$ is a character of $F^\times$ and $\rho$ is an irreducible, admissible representation of $\GL(2,F)$. Moreover, the equivalence classes of $\chi$ and $\rho$ are uniquely determined.
 \item Let $(\pi,V)$ be a smooth representation of $P_3$ of finite length. Then there exists a chain of $P_3$ subspaces
  $$
   0\subset V_2\subset V_1\subset V_0=V
  $$
  with the following properties,
  \begin{align*}
   V_2&\cong \dim(V_{U_3,\Theta})\cdot\tau^{P_3}_{\GL(0)}(1),\\
   V_1/V_2&\cong\tau^{P_3}_{\GL(1)}(V_{U_3,\Theta'}),\\
   V_0/V_1&\cong\tau^{P_3}_{\GL(2)}(V_{N_3}).
  \end{align*}
 \end{enumerate}
\end{proposition}
\begin{proof}
See 5.1 -- 5.15 of \cite{BeZe1976}.
\end{proof}
\subsection{\texorpdfstring{$P_3$}{}-theory for arbitrary central character}
It is easy to verify that any element of the Klingen parabolic subgroup $Q$ can be written in a unique way as
\begin{equation}\label{Qelementeq}
 \begin{bmatrix}ad-bc\\&a&b\\&c&d\\&&&1\end{bmatrix}\begin{bmatrix}1&\,-y&x&z\\&1&&x\\&&1&y\\&&&1\end{bmatrix}\begin{bmatrix}u\\&u\\&&u\\&&&u\end{bmatrix}
\end{equation}
with $\mat{a}{b}{c}{d}\in\GL(2,F)$, $x,y,z\in F$, and $u\in F^\times$. Let $Z^J$ be the center of the Jacobi group, consisting of all elements of $\GSp(4)$ of the form
\begin{equation}\label{ZJdefeq}
 \begin{bmatrix}1&&&*\\&1\\&&1\\&&&1\end{bmatrix}.
\end{equation}
Evidently, $Z^J$ is a normal subgroup of $Q$ with $Z^J\cong F$. 
Let $(\pi,V)$ be a smooth representation of $\GSp(4,F)$. Let $V(Z^J)$ be the span of all vectors $v-\pi(z)v$, where $v$ runs through $V$ and $z$ runs through $Z^J$. Then $V(Z^J)$ is preserved by the action of $Q$. Hence $Q$ acts on the quotient
$$
 V_{Z^J}:=V/V(Z^J).
$$
Let $\bar Q$ be the subgroup of $Q$ consisting of all elements of the form \eqref{Qelementeq} with $u=1$, i.e.,
$$
 \bar Q=\GSp(4)\cap\begin{bmatrix} *&*&*&*\\&*&*&*\\&*&*&*\\&&&1\end{bmatrix}.
$$
The map
\begin{equation}\label{QP3mapeq}
 i(\begin{bmatrix}ad-bc\\&a&b\\&c&d\\&&&1\end{bmatrix}\begin{bmatrix}1&\,-y&x&z\\&1&&x\\&&1&y\\&&&1\end{bmatrix})=\begin{bmatrix}a&b\\c&d\\&&1\end{bmatrix}\begin{bmatrix}1&&x\\&1&y\\&&1\end{bmatrix}
\end{equation}
establishes an isomorphism $\bar Q/Z^J\cong P_3$.

Recall the character $\psi_{c_1,c_2}$ of $U$ defined in \eqref{psic1c2eq}. Note that $U$ maps onto $U_3$ under the map \eqref{QP3mapeq}, and that the diagrams
$$
 \xymatrix{U\ar[r]^i\ar[dr]_{\psi_{-1,1}}&U_3\ar[d]^\Theta\\
           &\C^\times}
 \qquad\qquad\qquad
 \xymatrix{U\ar[r]^i\ar[dr]_{\psi_{-1,0}}&U_3\ar[d]^{\Theta'}\\
           &\C^\times}
$$
are commutative. The radical $N_Q$ (see \eqref{NQdefeq}) maps onto $N_3$ under the map \eqref{QP3mapeq}. The following theorem is exactly like Theorem 2.5.3 of \cite{NF}, except that the hypothesis of trivial central character is removed.

\begin{theorem} \label{finitelength}
 Let $(\pi,V)$ be an irreducible, admissible representation of $\GSp(4,F)$. The quotient $V_{Z^J} =V/V(Z^J)$ is a smooth representation of $\bar Q/Z^J$, and hence, via the map \eqref{QP3mapeq}, defines a smooth representation of $P_3$. As a representation of $P_3$, $V_{Z^J}$ has finite length. Hence, $V_{Z^J}$ has a finite filtration by $P_3$ subspaces such that the successive quotients are irreducible and of the form $\tau_{\GL(0)}^{P_3} (1)$, $\tau_{\GL(1)}^{P_3} (\chi)$ or $\tau_{\GL(2)}^{P_3} ( \rho)$ for some character $\chi$ of $F^\times$, or  some irreducible, admissible representation $\rho$ of $\GL(2,F)$. Moreover, the following statements hold:
 \begin{enumerate}
  \item There exists a chain of $P_3$ subspaces
   $$\label{V0V1V2def}
    0 \subset V_2 \subset V_1 \subset V_0 = V_{Z^J}
   $$
   such that
   \begin{align*}
    V_2 & \cong \dim \Hom_{U} (V, \psi_{-1,1} ) \cdot \tau_{\GL(0)}^{P_3} (1), \\
    V_1 / V_2 & \cong \tau_{\GL(1)}^{P_3} ( V_{U, \psi_{-1,0}} ), \\
    V_0/V_1  & \cong \tau_{\GL (2)}^{P_3} (V_{N_Q}).
   \end{align*}
   Here, the vector space $V_{U,\psi_{-1,0}}$ admits a smooth action of $\GL(1,F) \cong F^\times$ induced by the operators
   $$
    \pi ( \begin{bmatrix} a &&& \\ &a&& \\ &&1& \\ &&& 1 \end{bmatrix} ),\quad a\in F^\times,
   $$
   and $V_{N_Q}$ admits a smooth action of $\GL(2,F)$ induced by the operators
   $$
    \pi ( \begin{bmatrix} \det g && \\ & g& \\ && 1 \end{bmatrix} ), \quad g \in \GL(2,F).
   $$
  \item The representation  $\pi$ is generic if and only if $V_2 \neq 0$, and if $\pi$ is generic, then $V_2\cong \tau_{\GL(0)}^{P_3}(1)$.
  \item We have $V_2= V_{Z^J}$ if and only if $\pi$ is supercuspidal. If $\pi$ is supercuspidal and generic, then  $V_{Z^J}= V_2 \cong \tau_{\GL(0)}^{P_3}(1)$ is non-zero and irreducible. If $\pi$ is supercuspidal and non-generic, then  $V_{Z^J}=V_2 =0$.
\end{enumerate}
\end{theorem}

\begin{proof}
This is an application of Proposition \ref{P3representationsprop}. See Theorem 2.5.3 of \cite{NF} for the details of the proof.
\end{proof}

Given an irreducible, admissible representation $(\pi,V)$ of $\GSp(4,F)$, one can calculate the semisimplifications of the quotients $V_0/V_1$ and $V_1/V_2$ in the $P_3$-filtration from the Jacquet modules of $\pi$ with respect to the Siegel and Klingen parabolic subgroups. The results are exactly the same as in Appendix A.4 of \cite{NF} (where it was assumed that $\pi$ has trivial central character).

Note that there is a typo in Table A.5 of \cite{NF}: The entry for Vd in the ``$\text{s.s.}(V_0/V_1)$'' column should be $\tau_{\GL(2)}^{P_3}(\nu(\nu^{-1/2}\sigma \times \nu^{-1/2}\xi\sigma))$.
\subsection{Generic representations and zeta integrals}
Let $\pi$ be an irreducible, admissible, generic representation of $\GSp(4,F)$. Recall from Sect.~\ref{representationssec} that $\mathcal{W}(\pi,\psi_{c_1,c_2})$ denotes the Whittaker model of $\pi$ with respect to the character $\psi_{c_1,c_2}$ of $U$. For $W$ in $\mathcal{W}(\pi,\psi_{c_1,c_2})$ and $s \in \C$, we define the \emph{zeta integral} $Z(s,W)$ by
\begin{equation}\label{localzetaintdefeq}
 Z(s,W)=\int\limits_{F^\times}\int\limits_FW(\left[\begin{matrix}a\\&a\\
 &x&1\\&&&1\end{matrix}\right])|a|^{s-3/2}\,dx\,d^\times a.
\end{equation}
It was proved in Proposition 2.6.3 of \cite{NF} that there exists a real number $s_0$, independent of $W$, such that $Z(s,W)$ converges for $\Re(s)>s_0$ to an element of $\mathbb C(q^{-s})$. In particular, all zeta integrals have meromorphic continuation to all of $\C$.
Let $I(\pi)$ be the $\C$-vector subspace  of $\C(q^{-s})$ spanned by all $Z(s,W)$ for $W$ in $\mathcal{W}(\pi,\psi_{c_1,c_2})$. It is easy to see that $I(\pi)$ is independent of the choice of $\psi$ and $c_1,c_2$ in $F^\times$.

\begin{proposition}\label{basicpropertieszetaintegrals}
 Let $\pi$ be a generic, irreducible, admissible representation of $\GSp(4,F)$. Then $I(\pi)$ is a non-zero $\mathbb C[q^{-s},q^s]$-module containing $\mathbb C$,  and there exists $R(X) \in \mathbb C[X]$ such that $R(q^{-s}) I(\pi) \subset \mathbb C[q^{-s},q^s]$, so that $I(\pi)$ is a fractional ideal of the principal ideal domain $\mathbb C[q^{-s},q^s]$ whose quotient field is $\mathbb C(q^{-s})$. The fractional ideal $I(\pi)$ admits a generator of the form $1/Q(q^{-s})$ with $Q(0)=1$, where $Q(X) \in \mathbb C[X]$.
\end{proposition}
\begin{proof}
The proof is almost word for word the same as that of Proposition 2.6.4 of \cite{NF}. The only difference is that, in the calculation starting at the bottom of p.~79 of \cite{NF}, the element $q$ is taken from $\bar Q$ instead of $Q$.
\end{proof}

The quotient $1/Q(q^{-s})$ in this proposition is called the \emph{$L$-factor} of $\pi$, and denoted by $L(s,\pi)$. If $\pi$ is supercuspidal, then $L(s,\pi)=1$. The $L$-factors for all irreducible, admissible, generic, non-supercuspidal representations are listed in Table A.8 of \cite{NF}. By definition,
\begin{equation}\label{zetaLquotienteq}
 \frac{Z(s,W)}{L(s,\pi)}\in\C[q^s,q^{-s}]
\end{equation}
for all $W$ in $\mathcal{W}(\pi,\psi_{c_1,c_2})$.
\subsection{Generic representations admit split Bessel functionals}
In this section we will prove that an irreducible, admissible, generic representation of $\GSp(4,F)$ admits split Bessel functionals with respect to \emph{all} characters $\Lambda$ of $T$. This is a characteristic feature of generic representations, which will follow from Proposition \ref{nongenericsplitproposition} in the next section.

\begin{lemma}\label{GSp4genericlemma}
 Let $(\pi,V)$ be an irreducible, admissible, generic representation of $\GSp(4,F)$.
 Let $\sigma$ be a unitary character of $F^\times$, and let $s\in\C$ be arbitrary.
 Then there exists a non-zero functional $f_{s,\sigma}:\:V\rightarrow\C$
 with the following properties.
 \begin{enumerate}
  \item For all $x,y,z\in F$ and $v\in V$,
   \begin{equation}\label{GSp4genericlemmaeq1}
    f_{s,\sigma}(\pi(\begin{bmatrix}1&&y&z\\&1&x&y\\&&1\\&&&1\end{bmatrix})v)
    =\psi(y)f_{s,\sigma}(v).
   \end{equation}
  \item For all $a\in F^\times$ and $v\in V$,
   \begin{equation}\label{GSp4genericlemmaeq2}
    f_{s,\sigma}(\pi(\begin{bmatrix}a\\&1\\&&a\\&&&1\end{bmatrix})v)
    =\sigma(a)^{-1}|a|^{-s+1/2}f_{s,\sigma}(v).
   \end{equation}
 \end{enumerate}
\end{lemma}
\begin{proof}
We may assume that $V=\mathcal{W}(\pi,\psi_{c_1,c_2})$ with $c_1=1$. Let $s_0\in\R$
be such that $Z(s,W)$ is absolutely convergent for $\Re (s)>s_0$.
Then the integral
\begin{equation}\label{ZsWsigmadefeq}
 Z_\sigma(s,W)=\int\limits_{F^\times}\int\limits_F
 W(\begin{bmatrix}a\\&a\\&x&1\\&&&1\end{bmatrix})|a|^{s-3/2}\sigma(a)\,dx\,d^\times a
\end{equation}
is also absolutely convergent for $\Re(s)>s_0$, since $\sigma$ is unitary.
Note that these are the zeta integrals for the twisted representation $\sigma\pi$. Therefore, by \eqref{zetaLquotienteq}, the quotient $Z_\sigma(s,W)/L(s,\sigma\pi)$
is in $\C[q^{-s},q^s]$ for all $W\in\mathcal{W}(\pi,\psi_{c_1,c_2})$. We may therefore define, for any complex $s$,
\begin{equation}\label{GSp4genericlemmaeq3}
 f_{s,\sigma}(W)=\frac{Z_\sigma(s,\pi(s_2)W)}{L(s,\sigma\pi)},
\end{equation}
where $s_2$ is as in (\ref{s1s2defeq}). Straightforward calculations using the definition \eqref{ZsWsigmadefeq} show that \eqref{GSp4genericlemmaeq1} and \eqref{GSp4genericlemmaeq2} are satisfied for $\Re (s)>s_0$. Since both sides depend holomorphically on $s$, these identities hold on all of $\C$.
\end{proof}

\begin{proposition}\label{GSp4genericprop}
 Let $(\pi,V)$ be an irreducible, admissible and generic representation of $\GSp(4,F)$. Let $\omega_\pi$ be the central character of $\pi$. Then $\pi$ admits a split $(\Lambda,\theta)$-Bessel functional with respect to any character
 $\Lambda$ of $T$ that satisfies $\Lambda\big|_{F^\times}=\omega_\pi$.
\end{proposition}
\begin{proof}
Let $\theta$ be as in \eqref{splitthetaeq} with $T$ as in \eqref{splitthetaTeq}. Let $s\in\C$ and $\sigma$ be a unitary character of $F^\times$ such that
$$
 \Lambda(\begin{bmatrix}a\\&1\\&&a\\&&&1\end{bmatrix})=\sigma(a)^{-1}|a|^{-s+1/2}
 \qquad\text{for all }a\in F^\times.
$$
Let $f_{s,\sigma}$ be as in Lemma \ref{GSp4genericlemma}. By \eqref{GSp4genericlemmaeq2},
\begin{equation}\label{GSp4genericpropeq1}
 f_{s,\sigma}(\pi(\begin{bmatrix}a\\&1\\&&a\\&&&1\end{bmatrix})v)
 =\Lambda(a)f_{s,\sigma}(v)\qquad\text{for all }a\in F^\times.
\end{equation}
Since $\Lambda\big|_{F^\times}=\omega_\pi$ we
have in fact $f_{s,\sigma}(\pi(t)v)=\Lambda(t)f_{s,\sigma}(v)$ for all $t\in T$. Hence
$f_{s,\sigma}$ is a Bessel functional as desired.
\end{proof}
\subsection{Split Bessel functionals for non-generic representations}\label{splitbesselnongenericsec}
The converse of Proposition \ref{GSp4genericprop} is not true: There exist irreducible, admissible, non-generic representations of $\GSp(4,F)$ which admit split Bessel functionals. This follows from the following proposition. In fact, using this result and the $P_3$-filtrations listed in Table A.6 of \cite{NF}, one can precisely identify which non-generic representations admit split Bessel functionals. Other than in the generic case, the possible characters $\Lambda$ of $T$ are restricted to a finite number.

\begin{proposition}\label{nongenericsplitproposition}
 Let $(\pi,V)$ be an irreducible, admissible and non-generic representation of $\GSp(4,F)$. Let the semisimplification of the quotient $V_1=V_1/V_2$ in the $P_3$-filtration of $\pi$ be given by $\sum_{i=1}^n\tau^{P_3}_{\GL(1)}(\chi_i)$ with characters $\chi_i$ of $F^\times$.
 \begin{enumerate}
  \item $\pi$ admits a split Bessel functional if and only if the quotient $V_1$ in the $P_3$-filtration of $\pi$ is non-zero.
  \item  Let $\beta$ be a non-zero $(\Lambda,\theta)$-Bessel functional, with $\theta$ as in \eqref{splitthetaeq}, and a character $\Lambda$ of the group $T$ explicitly given in \eqref{splitthetaTeq}. Then there exists an $i$ for which
   \begin{equation}\label{nongenericsplitpropositioneq1}
    \Lambda(\begin{bmatrix}a\\&1\\&&a\\&&&1\end{bmatrix})=|a|^{-1}\chi_i(a)\qquad\text{for all}\quad a\in F^\times.
   \end{equation}
  \item If $V_1$ is non-zero, then there exists an $i$ such that $\pi$ admits a split $(\Lambda,\theta)$-Bessel functional with respect to a character $\Lambda$ of $T$ satisfying \eqref{nongenericsplitpropositioneq1}.
  \item The space of split $(\Lambda,\theta)$-Bessel functionals is zero or one-dimensional.
  \item The representation $\pi$ does not admit any split Bessel functionals if and only if $\pi$ is of type IVd, Vd, VIb, VIIIb, IXb, or is supercuspidal.
 \end{enumerate}
\end{proposition}
\begin{proof}
Let $N_{\mathrm{alt}}$ be as in \eqref{splitthetaconjNeq} and $\theta_{\mathrm{alt}}$ be as in \eqref{splitthetaconjeq}. We use the fact that any $(\Lambda,\theta_{\mathrm{alt}})$-Bessel functional factors through the twisted Jacquet module $V_{ N_{\mathrm{alt}},\theta_{\mathrm{alt}} }$. To calculate this Jacquet module, we use the $P_3$-filtration of Theorem \ref{finitelength}. Since $\pi$ is non-generic, the $P_3$-filtration simplifies to
$$
 0\subset V_1\subset V_0=V_{Z^J},
$$
with $V_1$ of type $\tau^{P_3}_{\GL(1)}$ and $V_0/V_1$ of type $\tau^{P_3}_{\GL(2)}$. Taking further twisted Jacquet modules and observing Lemma 2.5.6 of \cite{NF}, it follows that
$$
 V_{N_{\mathrm{alt}},\theta_{\mathrm{alt}}}=(V_1)_{\left[\begin{smallmatrix}1\\ *&1&*\\&&1\end{smallmatrix}\right],\psi},\qquad\text{where}\quad\psi(\begin{bmatrix}1\\x&1&y\\&&1\end{bmatrix})=\psi(y).
$$
By Lemma 2.5.5 of \cite{NF}, after suitable renaming,
$$
 0=J_n\subset\ldots\subset J_1\subset J_0=(V_1)_{\left[\begin{smallmatrix}1\\ *&1&*\\&&1\end{smallmatrix}\right],\psi},
$$
where $J_i/J_{i+1}$ is one-dimensional, and ${\rm diag}(a,1,1)$ acts on $J_i/J_{i+1}$ by $|a|^{-1}\chi_i(a)$. Table A.6 of \cite{NF} shows that all the $\chi_i$ are pairwise distinct. This proves i), ii), iii) and iv).

v) If $\pi$ is one of the representations mentioned in v), then $V_1/V_2=0$ by Theorem \ref{finitelength} (in the supercuspidal case), or by Table A.6 in \cite{NF} (in the non-supercuspidal case). By part i), $\pi$ does not admit a split Bessel functional. For any representation not mentioned in v), the quotient $V_1/V_2$ is non-zero, so that a split Bessel functional exists by iii).
\end{proof}
\section{Theta correspondences}
Let $S$ be as in \eqref{Sdefeq}, and let $\theta=\theta_S$ be as in \eqref{thetaSsetupeq}. Let $(\pi,V)$ be an irreducible, admissible representation of $\GSp(4,F)$, and let $(\sigma,W)$ be an irreducible, admissible representation of $\GO(X)$, where $X$ is an even-dimensional, symmetric, bilinear space. Let $\omega$ be the Weil representation of the group $R$, consisting of the pairs $(g,h)\in\GSp(4,F)\times\GO(X)$ with the same similitude factors, on the Schwartz space $\mathcal{S}(X^2)$. Assume that the pair $(\pi,\sigma)$ occurs in the theta correspondence defined by $\omega$, i.e., $\Hom_R(\omega,\pi\otimes\sigma)\neq0$. It is a theme in the theory of the theta correspondence to relate the twisted Jacquet module $V_{N,\theta}$ of $\pi$ to invariant functionals on $\sigma$; a necessary condition for the non-vanishing of $V_{N,\theta}$ is that $X$ represents $S$. See for example the remarks in Sect.~6 of \cite{Roberts1999}.

Applications to $(\Lambda,\theta_S)$-Bessel functionals also require the involvement of $T$. The idea is roughly as follows. The group $T$ is contained in $M$. Moreover, $\omega(m,h)$ for $(m,h)$ in $R\cap(M\times\GO(X))$ is given by an action of such pairs on $X^2$. The study of this action leads to the definition of certain compatible embeddings of $T$ into $\GO(X)$. Using these embeddings allows us to show that if $\pi$ has a $(\Lambda,\theta)$-Bessel functional, then $\sigma$ admits a non-zero functional transforming according to $\Lambda^{-1}$.

After setting up notations and studying the embeddings of $T$ mentioned above, we obtain the main result of this section, Theorem \ref{fourdimthetatheorem}. Section \ref{thetaapplicationssec} contains the applications to Bessel functionals.

\subsection{The spaces}
\label{thespacessubsec}
In this section we will consider non-degenerate symmetric bilinear spaces $(X,\langle\cdot, \cdot \rangle)$ over $F$ such that 
\begin{equation}
\label{Xtypeseq}
\text{$\dim X =2$,  or $\dim X =4$ and $\disc(X)=1$}. 
\end{equation} 
We begin by recalling the constructions of the isomorphism classes of these spaces, and the characterization of their similitude groups. 
Let $m \in F^\times$, $A=A_{\left[\begin{smallmatrix} 1 & \\ & -m \end{smallmatrix} \right]}$ and $T=A^\times$ be as in Sect. \ref{anotheralgebrasubsec}, so that
\begin{equation}
\label{specialATeq}
A =\{\begin{bmatrix}x&-y\\-ym&x\end{bmatrix}:x,y\in F\}, \qquad T=A^\times =\{\begin{bmatrix}x&-y\\-ym&x\end{bmatrix}:x,y\in F,\:x^2 -y^2m \neq 0\}.
\end{equation}
Let $\lambda \in F^\times$. Define a non-degenerate two-dimensional symmetric bilinear space $(X_{m,\lambda},\langle\cdot,\cdot\rangle_{m,\lambda})$ by
\begin{equation}
\label{twodimexeq}
X_{m,\lambda}=A_{\left[\begin{smallmatrix} 1 & \\ & -m \end{smallmatrix} \right]}, \qquad \langle x_1,x_2 \rangle_{m,\lambda} =  \lambda\,\mathrm{tr}(x_1x_2^*)/2, \quad x_1,x_2 \in X_{m,\lambda}.
\end{equation}
Here, $*$ is the canonical involution of $2 \times 2$ matrices, given by $\left[\begin{smallmatrix} a&b \\ c& d \end{smallmatrix} \right]^* =
\left[\begin{smallmatrix} d&-b \\ -c& a \end{smallmatrix} \right]$.
Define a homomorphism $\rho: T \to\GSO(X_{m,\lambda})$  by $\rho(t)x= tx$ for $x \in X_{m,\lambda}$. 
We also recall the Galois conjugation map $\gamma:A \to A$ from Sect.~\ref{anotheralgebrasubsec}; it is given by $\gamma(x) =x^*$ for $x \in A$. The map $\gamma$ can be regarded as an $F$ linear endomorphism
\begin{equation}
\label{gammaendoeq}
\gamma: X_{m,\lambda} \longrightarrow X_{m,\lambda},
\end{equation}
and as such is contained in $\OO(X_{m,\lambda})$ but not in $\SO(X_{m,\lambda})$. 
\begin{lemma}
\label{twodimclasslemma}
If  $(X_{m,\lambda},\langle\cdot,\cdot\rangle_{m,\lambda})$ is as in \eqref{twodimexeq}, then $\disc(X_{m,\lambda}) = mF^{\times 2}$,   $\varepsilon(X_{m,\lambda}) = (\lambda, m)$, and the homomorphism $\rho$ is an isomorphism, so that
\begin{equation}
\label{rhoiso2eq}
\rho: T \stackrel{\sim}{\longrightarrow} \GSO(X_{m,\lambda}).
\end{equation}
The image $\rho(T)$ and the map $\gamma$ generate $\GO(X_{m,\lambda})$. 
If $(X_{m,\lambda},\langle\cdot,\cdot\rangle_{m,\lambda})$ and $(X_{m',\lambda'},\langle\cdot,\cdot\rangle_{m',\lambda'})$ are as in \eqref{twodimexeq}, then
$(X_{m,\lambda},\langle\cdot,\cdot\rangle_{m,\lambda}) \cong (X_{m',\lambda'},\langle\cdot,\cdot\rangle_{m',\lambda'})$ if and only if $mF^{\times 2} = m'F^{\times 2}$ and $(\lambda ,m) = (\lambda', m')$. 
Every two-dimensional, non-degenerate symmetric bilinear space over $F$ is isomorphic $(X_{m,\lambda},\langle\cdot,\cdot\rangle_{m,\lambda})$ for some $m$ and $\lambda$.
\end{lemma}
\begin{proof}
Let $m, \lambda \in F^\times$. In $X_{m,\lambda}$ let $x_1 =\left[ \begin{smallmatrix} 1 & \\ & 1 \end{smallmatrix}\right]$ and
$x_2 =\left[ \begin{smallmatrix}  &1 \\ m &  \end{smallmatrix}\right]$. Then $x_1,x_2$ is a basis for $X_{m,\lambda}$, and in this basis the matrix for $X_{m,\lambda}$ is $\lambda \left[\begin{smallmatrix} 1 & \\ & -m \end{smallmatrix} \right]$. Calculations using this matrix show that $\disc(X_{m,\lambda}) = mF^{\times 2}$ and  $\varepsilon(X_{m,\lambda}) = (\lambda, m)$.
The map $\rho$ is clearly injective. To see that $\rho$ is surjective, let $h \in \GSO(X_{m,\lambda})$. 
Write $h$ in the ordered basis $x_1,x_2$ so that $h = \left[ \begin{smallmatrix} h_1 & h_2 \\ h_3 & h_4 \end{smallmatrix} \right]$. By the definition of $\GSO(X_{m,\lambda})$, we have ${}^t h \lambda  \left[\begin{smallmatrix} 1 & \\ & -m \end{smallmatrix} \right] h = \det (h) \lambda   \left[\begin{smallmatrix} 1 & \\ & -m \end{smallmatrix} \right]$. By the definition of $T$, this implies that $t=\left[ \begin{smallmatrix} & 1 \\ 1 & \end{smallmatrix} \right] h \left[ \begin{smallmatrix} & 1 \\ 1 & \end{smallmatrix} \right] \in T$. Hence, $h = \left[ \begin{smallmatrix} h_1 & h_3m \\ h_3 & h_1 \end{smallmatrix} \right]$ for some $h_1,h_3 \in F$.  Calculations now show that $\rho(t) x_1 = h (x_1)$ and $\rho(t) x_2 = h (x_2)$, so that $\rho(t) = h$. This proves the first assertion. The second assertion follows from the fact that two non-degenerate symmetric bilinear spaces over $F$ with the same finite dimension are isomorphic if and only if they have the same discriminant and Hasse invariant. For the final 
assertion, let $(X, \langle\cdot,\cdot\rangle)$ be a two-dimensional, non-degenerate symmetric bilinear space over $F$. There exists a basis for $X$ with respect to which the matrix for $X$ is of the form $\left[\begin{smallmatrix} \alpha_1 & \\ & \alpha_2 \end{smallmatrix}\right]$ for some $\alpha_1, \alpha_2  \in F^\times$. Then $\disc(X) =- \alpha_1 \alpha_2F^{\times 2}$ and $\varepsilon(X) = (\alpha_1,\alpha_2)_F$. An argument shows that there exists $\lambda \in F^\times$ such that $(\lambda, \disc(X))_F = \varepsilon(X)$. 
We now have $(X,\langle\cdot,\cdot\rangle) \cong (X_{m,\lambda},\langle\cdot,\cdot\rangle_{m,\lambda})$ with $m=\disc(X)$ because both spaces have the same discriminant and Hasse invariant. 
\end{proof}

Next, define a four-dimensional non-degenerate symmetric bilinear space over $F$ by setting 
\begin{equation}
\label{Xmateq}
X_{\Mat_2}=\Mat_2(F), \qquad \langle x_1,x_2 \rangle_{\Mat_2} = \mathrm{tr} (x_1x_2^*)/2, \quad x_1,x_2 \in X_{\Mat_2}. 
\end{equation}
Here, $*$ is the canonical involution of $2 \times 2$ matrices, given by $\left[\begin{smallmatrix} a&b \\ c& d \end{smallmatrix} \right]^* =
\left[\begin{smallmatrix} d&-b \\ -c& a \end{smallmatrix} \right]$.  Define $\rho: \GL(2,F) \times \GL(2,F) \to \GSO(X_{\Mat_2})$ by 
$\rho(g_1,g_2)x = g_1 x g_2^*$ for $g_1,g_2 \in \GL(2,F)$ and $x \in X_{\Mat_2}$. The map $*:X_{\Mat_2} \to X_{\Mat_2}$ is contained in 
$\OO(X_{\Mat_2})$ but not in $\SO(X_{\Mat_2})$. 

Finally, let $H$ be the division quaternion algebra over $F$. Let $1,i,j,k$ be a quaternion algebra basis for $H$, i.e., 
\begin{equation}
\label{Hdefeq}
H = F + F i + F j + Fk, \quad i^2 \in F^\times,\  j^2  \in F^\times,\ k =ij,\ ij = -ji. 
\end{equation}
Let $*$ be the canonical involution on $H$ so that $(a+b\cdot i + c\cdot j +d \cdot k)^* = a-b\cdot i -c \cdot j -d \cdot k$, and define the norm and trace functions
$\Norm,\Trace: H \to F$ by $\Norm(x) =xx^*$ and $\Trace(x) = x+x^*$ for $x \in H$. 
Define another four-dimensional non-degenerate symmetric bilinear space over $F$ by setting
\begin{equation}
\label{XHeq}
X_{H}=H, \qquad \langle x_1,x_2 \rangle_H = \Trace (x_1x_2^*)/2, \quad x_1,x_2 \in X_H. 
\end{equation}
Define $\rho: H^\times \times H^\times \to \GSO(X_H)$ by  $\rho(h_1,h_2)x = h_1 x h_2^*$ for $h_1, h_2 \in H^\times$ and $x \in X_{H}$. The map $*:X_{H} \to X_{H}$ is contained in 
$\OO(X_{H})$ but not in $\SO(X_{H})$. 

\begin{lemma}
\label{fourdisconelemma}
The  symmetric bilinear space $(X_{\Mat_2},\langle\cdot,\cdot\rangle_{\Mat_2})$ is non-degenerate, has dimension four, discriminant $\disc(X_{\Mat_2})=1$, and Hasse invariant $\varepsilon(X_{\Mat_2}) = (-1,-1)$. 
The $(X_H,\langle\cdot,\cdot\rangle_H)$ symmetric bilinear space is non-degenerate, has dimension four, discriminant $\disc(X_H)=1$, and Hasse invariant $\varepsilon(X_H) = -(-1,-1)$. 
The sequences
\begin{gather}
1 \longrightarrow F^\times \longrightarrow \GL(2,F) \times \GL(2,F) \stackrel{\rho}{\longrightarrow} \GSO(X_{\Mat_2}) \longrightarrow 1, \label{GSOexacteq1}\\
1 \longrightarrow F^\times \longrightarrow H^\times \times H^\times \stackrel{\rho}{\longrightarrow}  \GSO(X_H) \longrightarrow 1 \label{GSOexacteq2}
\end{gather}
are exact; here, the second maps send $a$ to $(a,a^{-1})$ for $a \in F^\times$. The image $\rho(\GL(2,F) \times \GL(2,F))$ and the map $*$ generate $\GO(X_{\Mat_2})$, and the image $\rho(H^\times \times H^\times)$ and the map $*$ generate $\GO(X_{H})$. 
Every four-dimensional, non-degenerate symmetric linear space over $F$ of discriminant $1$ is isomorphic to  $(X_{\Mat_2},\langle\cdot,\cdot\rangle_{\Mat_2})$ or 
$(X_{H},\langle\cdot,\cdot\rangle_{H})$.
\end{lemma}

\begin{proof}
See, for example, Sect.~2 of \cite{Roberts2001}. 
\end{proof}

\subsection{Embeddings}
Suppose that $(X,\langle\cdot, \cdot \rangle)$ satisfies \eqref{Xtypeseq}. 
We define an action of the group $\GL(2,F) \times \GO(X)$ on the set $X^2$ by 
\begin{equation}
\label{gl2goeq}
 (g,h)\cdot (x_1,x_2) = (hx_1,hx_2)g^{-1} = (g'_1hx_1+g'_3hx_2,g'_2hx_1+g'_4hx_2)
\end{equation}
for $(x_1,x_2) \in X^2$,  $h \in \GO(X)$ and $g \in \GL(2,F)$ with $g^{-1}=\left[ \begin{smallmatrix} g'_1& g'_2 \\ g'_3 & g'_4 \end{smallmatrix} \right]$.
For $S$ as in \eqref{Sdefeq} with $\det(S) \neq 0$, we define
\begin{equation}
\label{omegaSdefeq}
\Omega=\Omega_S =\Omega_{S,(X,\langle\cdot, \cdot\rangle)}= \{(x_1,x_2) \in X^2: \begin{bmatrix} \langle x_1,x_1 \rangle & \langle x_1, x_2\rangle \\ \langle x_1 , x_2 \rangle & \langle x_2 , x_2\rangle \end{bmatrix} = S \}. 
\end{equation}
We  say that $(X,\langle\cdot, \cdot\rangle)$ \emph{represents} $S$ if the set $\Omega$ is non-empty. 
\begin{lemma}
\label{ghlemma}
Let $(X,\langle\cdot,\cdot\rangle)$ be a non-degenerate symmetric bilinear space over $F$ satisfying \eqref{Xtypeseq},
and let $S$ be as in \eqref{Sdefeq} with $\det(S) \neq 0$. 
The subgroup 
\begin{equation}
\label{Bdefeq}
B=B_S= \{ (g,h) \in \GL(2,F) \times \GO(X): \text{${}^t g S g = \det(g) S$ and $\det(g) = \lambda (h)$} \}
\end{equation}
maps $\Omega=\Omega_S$ to itself under the action of $\GL(2,F) \times \GO(X)$ on $X^2$.
\end{lemma}
\begin{proof} 
Let $(g,h) \in B$, and 
let $g =\left[ \begin{smallmatrix} g_1&g_2\\g_3&g_4 \end{smallmatrix} \right]$. To start, we note that the assumption ${}^tgSg=\det(g)S$ is equivalent to ${}^tg^{-1}Sg^{-1}=\det(g)^{-1}S$, which is in turn equivalent to
\begin{align*}
 ag_4^2-bg_3g_4+cg_3^2 & = \det (g) a,\\
-ag_4g_2+b(g_1g_4+g_2g_3)/2-cg_3g_1 &= \det(g) b/2,\\
ag_2^2 - b g_2g_1 +cg_1^2 & = \det (g) c. 
\end{align*}
Let $(x_1,x_2) \in \Omega$ and set $(y_1,y_2) = (g,h_1(t))\cdot (x_1,x_2)$. By the definition of the action and $\Omega$, and using $\det(g) = \lambda (h)$, we have
\begin{align*}
\langle y_1, y_1 \rangle
&= \det(g)^{-1} \big( g_4^2 \langle x_1,x_1 \rangle -2 g_3 g_4 \langle x_1,x_2 \rangle + g_3^2 \langle x_2,x_2 \rangle\big) \\
& = \det(g)^{-1} \big( g_4^2 a- g_3 g_4 b + g_3^2 c\big) \\
&=a. 
\end{align*}
Similarly, $\langle y_1,y_2 \rangle = b/2$ and $\langle y_2,y_2 \rangle =c$. It follows that $(y_1,y_2) = (g,h_1(t))(x_1,x_2) \in \Omega$.
\end{proof}

Let $(X,\langle\cdot,\cdot\rangle)$ be a non-degenerate symmetric bilinear space over $F$ satisfying \eqref{Xtypeseq}, and let $S$ be as in \eqref{Sdefeq} with $\det(S) \neq 0$. Assume that $\Omega$ is non-empty, and let $T=T_S$, as in Sect.~\ref{anotheralgebrasubsec}. The goal of this section is to define, for each $z \in \Omega$, a set 
\begin{equation}
\label{Ecaleq}
\mathcal{E}(z) = \mathcal{E}_{(X,\langle\cdot,\cdot\rangle), S}(z)
\end{equation}
of embeddings  $\tau:T \to \GSO(X)$ such that:
\begin{align}
&\text{$\tau(t)=t$ for $t \in F^\times \subset T$;}\label{tauoneeq} \\
&\text{$\lambda(\tau(t))=\det(t)$ for $t \in T$, so that $(t,\tau(t)) \in B$ for $t \in T$;}\label{tautwoeq} \\
&\text{$(\begin{bmatrix}&1\\1&\end{bmatrix}t\begin{bmatrix}&1\\1&\end{bmatrix},\tau(t))\cdot z = z$ for $t\in T$.}\label{tauthreeeq}
\end{align}

We begin by noting some properties of $\Omega$. The set $\Omega$ is closed in $X^2$. The subgroup $\OO(X) \cong 1\times \OO(X) \subset B \subset \GL(2,F) \times \GO(X)$ preserves $\Omega$, i.e., if $ h \in \OO(X)$ and $(x_1,x_2) \in \Omega$, then $(hx_1,hx_2) \in \Omega$. Since  $\det(S) \neq 0$, the group $\OO(X)$ acts transitively  on $\Omega$. If $\dim X =4$, then $\SO(X)$ acts transitively on $\Omega$. If $\dim X =2$, then the action of 
$\SO(X)$ on $\Omega$ has two orbits.

\begin{lemma}
\label{simlemma}
Let $(X,\langle\cdot,\cdot\rangle)$ be a non-degenerate symmetric bilinear space over $F$ satisfying \eqref{Xtypeseq},
and let $S$ be as in \eqref{Sdefeq} with $\det(S) \neq 0$. Assume that $\dim X =2$ and $\Omega$ is non-empty. 
Let  $z=(z_1,z_2) \in \Omega$. For
$$
t =\mat{}{1}{1}{} g \mat{}{1}{1}{}= \mat{}{1}{1}{} \mat{g_1}{g_2}{g_3}{g_4} \mat{}{1}{1}{} \in T
$$
let $\tau_z(t):X \to X$ be the linear map 
 that has $g $ as matrix in the ordered
basis $z_1,z_2$ for $X$, so that
\begin{align*}
\tau_z(t)  (z_1) &= g_1 z_1+g_3z_2,\\
\tau_z(t)  (z_2) &= g_2z_1+g_4 z_2.
\end{align*}
\begin{enumerate}
\item For $t \in T$, the map $\tau_z(t)$  is contained in $\GSO(X)$ and $\lambda(\tau_z(t)) =\det(t)$.
\item If $z'$ lies in the  $\SO(X)$ orbit of $z$, and $t \in T$, then $\tau_z(t) = \tau_{z'}(t)$. 
\item The map sending $t$ to $\tau_z(t)$ defines an isomorphism
$
\tau_z:T\stackrel{\sim}{\longrightarrow} \GSO(X). 
$
\item Let $h_0 \in \OO(X)$ with $\det(h_0)=-1$. Let $z' \in \Omega$ not be in the $\SO(X)$ orbit of $z$. Then $\tau_{z'}(t) = h_0 \tau_z(t) h_0^{-1}$ for $t \in T$. 
\item Let $t \in T$. The element $(\left[\begin{smallmatrix}&1\\1&\end{smallmatrix}\right] t \left[\begin{smallmatrix}&1\\1&\end{smallmatrix}\right],\tau_z(t)) \in  B$ acts by the identity on the $\SO(X)$ orbit of $z$, and maps the other $\SO(X)$ orbit of $\Omega$ to itself.  
\end{enumerate}
\end{lemma}
\begin{proof}
i) A computation verifies that $\tau_z(t) \in \GO(X)$, with similitude factor $\lambda(\tau_z(t))=\det(g)=\det(t)$, and the equality $\det(\tau_z(t))=\lambda(\tau_z(t))$ implies that $\tau_z(t) \in \GSO(X)$ by the definition of $\GSO(X)$.  

ii) Suppose that  $z'=(z_1',z_2')$  lies in the  $\SO(X)$ orbit of $z$, and let $c \in \SO(X)$ be such that $c(z_1) = z_1'$ and $c(z_2) = z_2'$. Then
$\tau_{z'}(t)=c\tau_z(t)c^{-1}$. But the group $\GSO(X)$ is abelian, so that $\tau_{z'}(t)=c\tau_z(t)c^{-1} = \tau_z(t)$. 

iii) Calculations prove that $\tau_z:T \to \GSO(X)$ is an isomorphism. 

iv) Let $z''=h_0(z)$.  A calculation shows that $\tau_{z''}(t) = h_0 \tau_z(t) h_0^{-1}$ for $t \in T$. By, ii), $\tau_{z''}(t)=\tau_{z'}(t)$ for $t \in T$. 

v) Write $g =\left[\begin{smallmatrix}&1\\1&\end{smallmatrix}\right] t \left[\begin{smallmatrix}&1\\1&\end{smallmatrix}\right]$, so that ${}^tgSg=\det(g)S$. 
Let $g =\left[ \begin{smallmatrix} g_1&g_2\\g_3&g_4 \end{smallmatrix} \right]$.  By the definition of $\tau_z(t)$, we have
\begin{align*}
(g,\tau_z(t))\cdot z
&=(\det(g)^{-1} g_4(g_1z_1+g_3z_2)-\det(g)^{-1}  g_3(g_2z_1+g_4z_2), \\
&\qquad \det(g)^{-1}  (-g_2)(g_1z_1+g_3z_2) + \det(g)^{-1}  g_1(g_2z_1+g_4z_2))\\
&=z.
\end{align*}
By ii),  it follows that $(g,\tau_z(t))$ acts by the identity on all of the $\SO(X)$ orbit of $z$. 
Next, let $z' \in \Omega$ with $z' \notin \SO(X)  z$. Assume that 
$(g,\tau_z(t))\cdot z' \in \SO(X) z$; we will obtain a contradiction. Since $(g,\tau_z(t))\cdot z' \in \SO(X)z$ and since we have already proved that $(g,\tau_z(t))$ acts
by the identity on $\SO(X)z$, we have:
\begin{align*}
(g,\tau_z(t)) \cdot \big((g,\tau_z(t)) \cdot z'\big) &= (g,\tau_z(t)) \cdot z' \\
(g,\tau_z(t))\cdot  z' & =  z'.
\end{align*}
This is a contradiction since $z' \notin \SO(X)z$ and $(g,\tau_z(t))\cdot z' \in \SO(X) z$. 
\end{proof}

Let $(X,\langle\cdot,\cdot\rangle)$ be a non-degenerate symmetric bilinear space over $F$ satisfying \eqref{Xtypeseq},
and let $S$ be as in \eqref{Sdefeq} with $\det(S) \neq 0$. Assume that $\dim X =2$ and $\Omega$ is non-empty. For
$z \in \Omega$, we define
\begin{equation}
\label{twodimcalEeq}
\mathcal{E}(z) = \mathcal{E}_{(X,\langle\cdot,\cdot\rangle), S}(z) = \{ \tau_z \},
\end{equation}
with $\tau_z$ as defined in Lemma \ref{simlemma}. It is evident from Lemma \ref{simlemma} that the element of $\mathcal{E}(z)$
has the properties \eqref{tauoneeq}, \eqref{tautwoeq}, and \eqref{tauthreeeq}.

\begin{lemma}
\label{twoscalarslemma}
Let $(X,\langle\cdot,\cdot\rangle)$ be a non-degenerate symmetric bilinear space over $F$ satisfying \eqref{Xtypeseq},
and let $S$ be as in \eqref{Sdefeq} with $\det(S) \neq 0$. Assume that $\dim X =2$. Let $\lambda, \lambda'\in F^\times$, and set
$\Omega=\Omega_{\lambda S}$ and $\Omega' = \Omega_{\lambda' S}$. Assume that $\Omega$ and $\Omega'$ are
non-empty. Then
\begin{equation}
\label{Eunioneq}
\bigcup\limits_{z \in \Omega} \mathcal{E}(z) = \bigcup\limits_{z' \in \Omega'} \mathcal{E}(z').
\end{equation}
\end{lemma}
\begin{proof}
Let $\Omega_1$ and $\Omega_2$ be the two $\SO(X)$ orbits of the action of $\SO(X)$ on $\Omega$ so that $\Omega=\Omega_1 \sqcup \Omega_2$, and analogously define and write
$\Omega' = \Omega_1' \sqcup \Omega_2'$. Let $z=(z_1,z_2) \in \Omega_1$ and $z'=(z_1',z_2') \in \Omega'_1$. Define a linear map $h:X \to X$ by setting $h(z_1) = z_1'$ and $h(z_2)=z_2'$. We have $\langle h(x), h(y) \rangle = (\lambda' /\lambda) \langle x,y \rangle $ for $x,y\in X$, 
so that $h \in \GO(X)$. Assume that $h \notin \GSO(X)$. Let $z''=(z_1'',z_2'') \in \Omega_2'$,
and let $h':X\to X$ be the linear map defined by $h'(z_1')=z_1''$ and $h'(z_2')=z_2''$. Then $h' \in \OO(X)$ with $\det (h') =-1$, so that $h'h \in \GSO(X)$ and $(h'h)(z_1)=z_1''$ and $(h'h)(z_2)=z_2''$. Therefore, by renumbering if necessary, we may assume that $h \in \GSO(X)$. Next, a calculation shows that $h\tau_z(t)h^{-1} = \tau_{z'}(t)$ for $t \in T$. Since $\GSO(X)$ is abelian, this means that $\tau_z =\tau_{z'}$. The claim \eqref{Eunioneq} follows now from ii) and iv) of Lemma \ref{simlemma}.
\end{proof}

\begin{lemma}
\label{Zdecomplemma}
Let $(X,\langle\cdot,\cdot\rangle)$ be a non-degenerate symmetric bilinear space over $F$ satisfying \eqref{Xtypeseq},
and let $S$ be as in \eqref{Sdefeq} with $\det(S) \neq 0$. Assume that $\dim X =4$ and $\Omega$ is non-empty. Let $z=(z_1,z_2) \in \Omega$,
and set $U = Fz_1+Fz_2$, so that $X = U \oplus U^\perp$ with $\dim U =\dim U^\perp =2$. 
There exists $\lambda \in F^\times$ such that $(U^\perp, \langle \cdot, \cdot \rangle)$ represents $\lambda S$. 
\end{lemma}
\begin{proof} Let $\Mat_{4,1}(F)$ be the $F$ vector space of $4 \times 1$ matrices with entries from $F$. 
Let $D=-\det(S)$. 
Let $\lambda \in F^\times$, and define a four-dimensional symmetric bilinear space $X_\lambda$ by letting $X_\lambda=\Mat_{4,1}(F)$ 
with symmetric bilinear form $b$ given by $b(x,y)={}^txMy$, where 
$$
M = \begin{bmatrix} S & \\ & \lambda S \end{bmatrix}.
$$
Evidently, $\disc(X_\lambda) =1$, and the Hasse invariant of $X_\lambda$ is 
$\varepsilon(X_\lambda) = (-1,-1)_F(-\lambda, D)_F$.
Now assume that $X$ is isotropic. Then the Hasse invariant of $X$ is $(-1,-1)_F$. It follows that if $\lambda=-1$, then  $\varepsilon(X_\lambda) = \varepsilon(X)$, so that $X_\lambda \cong X$. By the Witt cancellation theorem, $(U^\perp, \langle \cdot, \cdot \rangle)$ represents $\lambda S$.
Next, assume that $X$ is anisotropic, so that $\varepsilon(X) = -(-1,-1)_F$. By hypothesis, $(X,\langle \cdot,\cdot \rangle)$
represents $S$; since $X$ is anisotropic, this implies that $D \notin F^{\times 2}$. Since, $D \notin F^{\times 2}$, 
there exists $\lambda \in F^\times$ such that $-1 = (-\lambda,D)_F$. It follows that $\varepsilon(X_\lambda) = \varepsilon(X)$, so that $X_\lambda \cong X$; again the Witt cancellation theorem implies that $(U^\perp, \langle \cdot, \cdot \rangle)$ represents $\lambda S$.
\end{proof}

Let $(X,\langle\cdot,\cdot\rangle)$ be a non-degenerate symmetric bilinear space over $F$ satisfying \eqref{Xtypeseq},
and let $S$ be as in \eqref{Sdefeq} with $\det(S) \neq 0$. Assume that $\dim X =4$, $\Omega=\Omega_S$ is non-empty, and 
let $T=T_S$, as in Sect.~\ref{anotheralgebrasubsec}. Let $z=(z_1,z_2) \in \Omega$, and as in Lemma \ref{Zdecomplemma}, let
$U = Fz_1+Fz_2$, so that $X = U \oplus U^\perp$ with $\dim U =\dim U^\perp =2$. By Lemma \ref{Zdecomplemma} there exists
$\lambda \in F^\times$ such that $(U^\perp, \langle \cdot, \cdot \rangle)$ represents $\lambda S$. 
Let $\tau_z: T \to \GSO(U)$ be the isomorphism
from Lemma \ref{simlemma} that is associated to $z$. Also,  let $\tau_{z'}, \tau_{z''}: T \to \GSO(U^\perp)$ be the isomorphisms from Lemma \ref{simlemma}, where $z'$ and $z''$ are representatives for the two $\SO(U^\perp)$ orbits of $\SO(U^\perp)$ acting on $\Omega_{\lambda S, (U^\perp,\langle\cdot,\cdot\rangle)}$;
by Lemma \ref{twoscalarslemma}, $\{\tau_{z'}, \tau_{z''}\}$ does not depend on the choice of $\lambda$.  We now define 
\begin{equation}
\label{fourEzeq}
\mathcal{E}(z) = \mathcal{E}_{(X,\langle\cdot,\cdot\rangle), S}(z) = \{ \tau_1, \tau_2 \}, 
\end{equation}
where $\tau_1,\tau_2:T \to \GSO(X)$ are defined by 
$$
\tau_1(t) = \begin{bmatrix} \tau_z(t) & \\ & \tau_{z'}(t) \end{bmatrix},\qquad
\tau_2(t) = \begin{bmatrix} \tau_z(t) & \\ & \tau_{z''}(t) \end{bmatrix}
$$
with respect to the decomposition $Z = U \oplus U^\perp$, for $t \in T$. 
For $t \in T$, the similitude factor of $\tau_i(t)$ is $\det(t)$. It is evident that the elements of
$\mathcal{E}(z)$ satisfy \eqref{tauoneeq}, \eqref{tautwoeq}, and \eqref{tauthreeeq}.

\begin{lemma}
\label{zstablemma}
Let $(X,\langle\cdot,\cdot\rangle)$ be a non-degenerate symmetric bilinear space over $F$ satisfying \eqref{Xtypeseq},
and let $S$ be as in \eqref{Sdefeq} with $\det(S) \neq 0$. 
Assume that  $\Omega=\Omega_S$ is non-empty, and 
let $A=A_S$ and $T=T_S$, as in Sect.~\ref{anotheralgebrasubsec}. If $\dim X =4$,
assume that $A$ is a field. Let $z \in \Omega$
and $\tau \in \mathcal{E}(z)$. Let $C$ be a compact, open subset
of $\Omega$ containing $z$.  There exists a compact, open
subset $C_0$ of $\Omega$ such that $z \in C_0 \subset C$ and  
$(\left[ \begin{smallmatrix} &1\\1& \end{smallmatrix} \right] t \left[\begin{smallmatrix}&1\\1&\end{smallmatrix}\right],\tau(t)) \cdot C_0 = C_0$
for $t \in T$. 
\end{lemma}
\begin{proof}
Assume $\dim X =2$. Let $C_0$ be the intersection of $C$
with the $\SO(X)$ orbit of $z$ in $\Omega$. Then $C_0$ is a compact,
open subset of $\Omega$ because the $\SO(X)$ orbit of $z$ in $\Omega$
is closed and open in $\Omega$, and $C$ is compact and open. We
have $(\left[ \begin{smallmatrix} &1\\1& \end{smallmatrix} \right] t \left[\begin{smallmatrix}&1\\1&\end{smallmatrix}\right],\tau(t)) \cdot C_0 = C_0$
for $t \in T$ by v) of Lemma \ref{simlemma}. 
Assume $\dim X =4$. 
The group of pairs  $(\left[ \begin{smallmatrix} &1\\1& \end{smallmatrix} \right] t \left[\begin{smallmatrix}&1\\1&\end{smallmatrix}\right],\tau(t))$  for $t \in T$ acts on $X^2$ and can be regarded as a subgroup of $\GL(X^2)$. The group $T$ contains $F^\times$, and the pairs with $t \in F^\times$
 act by the identity on $X^2$. The assumption that $A$ is a field
implies that $T/F^\times$ is compact, and hence the image $\mathcal{K}$ in $\GL(X^2)$ of this group of pairs
is compact.  There exists a lattice
$\mathcal{L}$ of $X^2$ such that $k\cdot\mathcal{L} = \mathcal{L}$ for $k \in \mathcal{K}$.
Also, by \eqref{tauthreeeq} we have that $k\cdot z=z$ for $k \in \mathcal{K}$. Let $n$ be sufficiently large so 
that $(z+\varpi^n \mathcal{L})\cap \Omega \subset C$.
Then  $C_0 = (z+\varpi^n\mathcal{L})\cap \Omega$ is the desired set. 
\end{proof}
\subsection{Example embeddings}
\label{exampleembeddingssubsec}
In this section we provide explicit formulas for the embeddings of the previous section.
\begin{lemma}\label{lambdaSlemma}
Let $S$ be as in  \eqref{Sdefeq}. Let $m, \lambda \in F^\times$, and define  $(X_{m,\lambda}, \langle \cdot, \cdot \rangle_{m,\lambda})$ as in \eqref{twodimexeq}. 
The set   $\Omega=\Omega_S$ is non-empty if and only if $\disc(S)=mF^{\times 2}$ and $\varepsilon(S) = (\lambda,m)_F$.  Assume that the set $\Omega$ is non-empty. 
Set $D=b^2/4-ac$ so that $\disc(S) = DF^{\times 2}$, and define $\Delta$ and the quadratic extension $L=F+F\Delta$ of $F$ (which need not be a field) associated to $D$ as in Sect.~\ref{quadextsubsec}. Similarly, define $\Delta_m$ with respect to $m$; the quadratic extension associated $m$ is also $L$ and $L=F+F\Delta_m$. 
The set of compositions
$$
L^\times \stackrel{\sim}{\longrightarrow} T_S \stackrel{\tau}{\longrightarrow} \GSO(X_{m,\lambda})
$$
for $z \in \Omega$ and $\tau \in \mathcal{E}(z)$ is the same as the set consisting of the two compositions
$$
L^\times \stackrel{\sim}{\longrightarrow} T_{\left[\begin{smallmatrix} 1 & \\ & -m \end{smallmatrix}\right] } \stackrel{\sim}{\longrightarrow} \GSO(X_{m,\lambda}), \quad L^\times \stackrel{\gamma}{\longrightarrow} L^\times \stackrel{\sim}{\longrightarrow} T_{\left[\begin{smallmatrix} 1 & \\ & -m \end{smallmatrix}\right] } \stackrel{\sim}{\longrightarrow} \GSO(X_{m,\lambda}).
$$
Here, 
the maps $L^\times \to T_{\left[\begin{smallmatrix} 1 & \\ & -m \end{smallmatrix}\right] }$
are as in \eqref{ALisoeq}, and the isomorphism $\rho$ of $T_{\left[\begin{smallmatrix} 1 & \\ & -m \end{smallmatrix}\right] }$ with 
$\GSO(X_{m,\lambda})$ is as in \eqref{rhoiso2eq}. 
\end{lemma}
\begin{proof} By definition, $\Omega$ is non-empty if and only if there exist $x_1,x_2 \in X_{m,\lambda}$ such that $S = \left[\begin{smallmatrix} \langle x_1,x_1 \rangle & \langle x_1,x_2 \rangle \\ \langle x_1,x_2 \rangle & \langle x_2, x_2 \rangle \end{smallmatrix} \right]$. Since $X_{m,\lambda}$ is two-dimensional, this means that $\Omega$ is non-empty if and only if $(X_{m,\lambda}, \langle \cdot, \cdot \rangle_{m,\lambda})$ is equivalent to the symmetric bilinear space over $F$ defined by $S$. From Lemma \ref{twodimclasslemma}, we have $\disc(X_{m,\lambda}) = m F^{\times 2}$ and  $\varepsilon(X_{m,\lambda}) = (\lambda,m)_F$. Since a finite-dimensional non-degenerate symmetric bilinear space over $F$ is determined by its dimension, discriminant and Hasse invariant, it follows that $\Omega$ is non-empty if and only if $\disc(S) =m F^{\times 2}$ and $\varepsilon(S)=(\lambda,m)_F$.

Assume that $\Omega$ is non-empty, so that  $\disc(S)=m F^{\times 2}$ and $\varepsilon(S) = (\lambda,m)_F$. Let $e \in F^\times$ be such that $\Delta=e \Delta_m$; then $b^2/4-ac=D = e^2 m$. Assume first that $a \neq 0$.  By Sect.~\ref{twobytwosubsec},  $\varepsilon(S) =(a,m)_F$. Therefore, $(a,m)_F=(\lambda,m)_F$. It follows that there exist $g,h \in F^\times$ such that $g^2 - m h^2 = \lambda^{-1} a$. Set 
$$
z_1=\begin{bmatrix} g & h \\ -h(-m) & g \end{bmatrix},
\quad
z_2=a^{-1} \begin{bmatrix} ehm +gb/2 & eg+hb/2 \\ -(eg+hb/2)(-m) & ehm+gb/2 \end{bmatrix}.
$$
Then $z_1,z_2 \in X_{m,\lambda}$, and a calculation shows that 
$$
\begin{bmatrix}
 \langle z_1,z_1 \rangle_{m,\lambda} &  \langle z_1,z_2 \rangle_{m,\lambda}  \\  \langle z_1,z_2 \rangle_{m,\lambda} &  \langle z_2,z_2 \rangle_{m,\lambda}
\end{bmatrix}
= S.
$$
It follows that $z=(z_1,z_2) \in \Omega$. Let $u \in L^\times$. Write $u =x + y\Delta$ for some $x,y \in F^\times$. By \eqref{ALisoeq}, $u$ corresponds to $t = \left[\begin{smallmatrix} x - yb/2 & -ya \\ yc & x+yb/2 \end{smallmatrix} \right] \in T_S$. Using the definition of $\tau_z(t)$, we find that
\begin{align}
\tau_z(t) (z_1) & = \begin{bmatrix} gx -ehym & hx-egy \\ -(hx-egy)(-m) & gx-ehym \end{bmatrix},\label{z1teq}\\
\tau_z(t)(z_2)&= \frac{1}{2a}\begin{bmatrix}(2ehm +bg)x- (behm + 2e^2mg)  y  &(2 e g+ b h ) x- ( b e g+ 2e^2m h) y  \\ -((2 e g+ b h ) x- ( b e g+ 2e^2m h) y )(-m) & (2ehm +bg)x-(behm + 2e^2mg)  y  \end{bmatrix}.\label{z2teq}
\end{align}
On the other hand, we also have that $u = x + y e \Delta_m$, and $u$ corresponds to the element 
$
t'=\left[\begin{smallmatrix} x & -ye \\ (ye) (-m) & x \end{smallmatrix}\right]
$
in $T_{\left[\begin{smallmatrix} 1 & \\ & -m \end{smallmatrix}\right] } $. Moreover, calculations show that $\rho(t')(z_1)=t'\cdot z_1$ and $\rho(t')(z_2)=t' \cdot z_2$ are as in \eqref{z1teq} and \eqref{z2teq}, respectively, proving that the two compositions 
$$
L^\times \stackrel{\sim}{\longrightarrow} T_S \stackrel{\tau_z}{\longrightarrow} \GSO(X_{m,\lambda}),\qquad
L^\times \stackrel{\sim}{\longrightarrow} T_{\left[\begin{smallmatrix} 1 & \\ & -m \end{smallmatrix}\right] } \stackrel{\rho}{\longrightarrow} \GSO(X_{m,\lambda})
$$
are the same map. Next, let $z'=(\gamma(z_1),\gamma(z_2))$. Then $z' \in \Omega$, and calculations as above show that the two compositions
$$
L^\times \stackrel{\sim}{\longrightarrow} T_S \stackrel{\tau_{z'}}{\longrightarrow} \GSO(X_{m,\lambda}),\qquad
L^\times \stackrel{\gamma}{\longrightarrow}  L^\times \stackrel{\sim}{\longrightarrow} T_{\left[\begin{smallmatrix} 1 & \\ & -m \end{smallmatrix}\right] } \stackrel{\rho}{\longrightarrow} \GSO(X_{m,\lambda})
$$
are the same. This completes the proof in this case since $z$ and $z'$ are representatives for the two $\SO(X_{m,\lambda})$ orbits of $\Omega$, and by ii) of Lemma \ref{simlemma}, $\cup_{w \in \Omega}\mathcal{E}(w) = \{\tau_z,\tau_{z'}\}$. Now assume that $a=0$. Set
$$
z_1=\lambda^{-1} \begin{bmatrix} b/2&-e \\ e(-m) & b/2 \end{bmatrix}, \qquad
z_2=\begin{bmatrix} (c\lambda^{-1}+1)/2 & -eb^{-1}(c\lambda^{-1}-1) \\ eb^{-1} (c\lambda^{-1}-1) (-m) & (c\lambda^{-1}+1)/2 \end{bmatrix}.
$$
Again, a calculation shows that $z=(z_1,z_2) \in \Omega$. Let $u \in L^\times$ with $u =x + y\Delta$ for some $x,y \in F^\times$. Then $u$ corresponds to $ t=\left[\begin{smallmatrix} x - yb/2 &  \\ yc & x+yb/2 \end{smallmatrix} \right] \in T_S$, and $u$ corresponds to $t'=\left[\begin{smallmatrix} x &-y e \\ ye(-m) & x\end{smallmatrix} \right] \in T_{\left[\begin{smallmatrix} 1 & \\ & -m \end{smallmatrix} \right]}$. Computations show that $\tau_z(t)(z_1)=\rho(t')(z_1)$ and $\tau_z(t)(z_2)=\rho(t')(z_2)$, proving that the compositions
$$
L^\times \stackrel{\sim}{\longrightarrow} T_S \stackrel{\tau_{z'}}{\longrightarrow} \GSO(X_{m,\lambda}),\qquad
 L^\times \stackrel{\sim}{\longrightarrow} T_{\left[\begin{smallmatrix} 1 & \\ & -m \end{smallmatrix}\right] } \stackrel{\rho}{\longrightarrow} \GSO(X_{m,\lambda})
$$
are the same. As in the previous case, if $z'=(\gamma(z_1),\gamma(z_2))$, then $z' \in \Omega$, and the two compositions
$$
L^\times \stackrel{\sim}{\longrightarrow} T_S \stackrel{\tau_{z'}}{\longrightarrow} \GSO(X_{m,\lambda}),\qquad
L^\times \stackrel{\gamma}{\longrightarrow}  L^\times \stackrel{\sim}{\longrightarrow} T_{\left[\begin{smallmatrix} 1 & \\ & -m \end{smallmatrix}\right] } \stackrel{\rho}{\longrightarrow} \GSO(X_{m,\lambda})
$$
are the same. As above, this completes the proof. \end{proof}

Let $c \in F^\times$, and set
\begin{equation}
\label{Sceq}
S = \begin{bmatrix} 1 &  \\ & c \end{bmatrix}.
\end{equation}
Let $(X_{\Mat_2}, \langle \cdot,\cdot \rangle_{\Mat_2})$ be as in \eqref{Xmateq}.
Let $A=A_S$ and $T=T_S$  be as in Sect.~\ref{anotheralgebrasubsec}.
We embed $A$ in $\Mat_2(F)$ via the inclusion map. 
Set 
$$
z_1 = \begin{bmatrix} 1 & \\ & 1 \end{bmatrix}, \quad z_2 = \begin{bmatrix} & 1 \\ -c & \end{bmatrix}, \quad
z_1' = \begin{bmatrix} 1 & \\ & -1 \end{bmatrix}, \quad z_2' = \begin{bmatrix} &1 \\ c & \end{bmatrix}. 
$$
The vectors $z_1,z_2,z_1',z_2'$  form an orthogonal ordered basis for $X_{\Mat_2}$,
and in this basis the matrix for $X_{\Mat_2}$ is
$$
\begin{bmatrix} S & \\ & -S \end{bmatrix}.
$$
As in Lemma \ref{Zdecomplemma}, set $U =Fz_1 +Fz_2$. Then $U^\perp=Fz_1'+Fz_2'$, and the $\lambda$ of Lemma \ref{Zdecomplemma} is $-1$. 
Calculations show that the set $\mathcal{E}(z)=\mathcal{E}_{X_{\Mat_2}}(z)$ of \eqref{fourEzeq} is
\begin{equation}
\label{Ezmateq}
\mathcal{E}_{X_{\Mat_2}}(z) = \{ \tau_1,\tau_2 \}, \qquad \tau_1(t) = \rho(t,1), \quad \tau_2(t) =\rho(1,\gamma(t)), \quad t \in T^\times. 
\end{equation}

Finally,  let $S$ be as in \eqref{Sceq} with $-c \notin F^{\times 2}$, and let $(X_{H}, \langle \cdot,\cdot \rangle_{H})$ be as in \eqref{XHeq}.
Let $A=A_S$ and $T=T_S$  be as in Sect.~\ref{anotheralgebrasubsec}.
Let $L$ be the quadratic extension associated to $-c$ as in Sect.~\ref{quadextsubsec}; $L$ is a field. Let $e$ be a representative for the non-trivial
coset of $F^\times /  \Norm_{L/F}(L^\times)$, so that $(e,-c)_F=-1$. We realize the division quaternion algebra $H$ over $F$ as 
\begin{equation}
\label{SpecialHeq}
H = F + F i + F j + Fk, \quad i^2 = -c,\ j^2 = e,\ k =ij,\ ij = -ji. 
\end{equation}
 We embed $A$ into $H$ via the map defined by
$$
\begin{bmatrix} x & -y \\ cy & x \end{bmatrix} \mapsto x -yi
$$
for $x,y \in F$.
Let
\begin{equation}
\label{specialijkeq}
z_1 =1, \quad z_2 = i, \quad z_1' = j, \quad z_2' =k.
\end{equation}
The vectors $z_1,z_2,z_1',z_2'$  form an orthogonal ordered basis for $X_{H}$,
and in this basis the matrix for $X_{H}$ is
$$
\begin{bmatrix} S &  \\ & -e S \end{bmatrix}.
$$
As in Lemma \ref{Zdecomplemma}, set $U =Fz_1 +Fz_2$. Then $U^\perp=Fz_1'+Fz_2'$, and the $\lambda$ of Lemma \ref{Zdecomplemma} is $-e$. 
Calculations again show that the set $\mathcal{E}(z)=\mathcal{E}_{X_H}(z)$ of \eqref{fourEzeq} is
\begin{equation}
\label{EzHeq}
\mathcal{E}_{X_H}(z) = \{ \tau_1,\tau_2 \}, \qquad \tau_1(t) = \rho(t,1), \quad \tau_2(t) =\rho(1,\gamma(t)), \quad t \in T^\times. 
\end{equation}

To close this subsection, we note that $(X_H,\langle\cdot,\cdot\rangle_H)$ does not represent $S$ if $S$ is as in \eqref{Sceq} but $-c \in F^{\times 2}$. 
To see this, assume that $-c \in F^{\times 2}$ and $(X_H,\langle\cdot,\cdot\rangle_H)$ represents $S$; we will obtain a contradiction. Write $-c=t^2$ for some $t \in F^\times$. Since $X_H$ represents $S$, there exist $x_1,x_2 \in H$ such that $\langle x_1,x_1 \rangle_H = \Norm(x_1)=1$, $\langle x_2,x_2 \rangle_H = \Norm(x_2)=c=-t^2$ and $\langle x_1,x_2 \rangle_H = \Trace(x_1x_2^*)/2 =0$. A calculation shows that $\Norm(tx_1+x_2)=0$. Since $H$ is a division algebra, this means that $tx_1=-x_2$. Hence, $t^2 = N(tx_1) = \langle tx_1,tx_1 \rangle_H = \langle tx_1 ,-x_2 \rangle_H = -t \langle x_1,x_2 \rangle_H =0$, a contradiction.
\subsection{Theta correspondences and Bessel functionals} 
\label{thetabesselsubsec} 
In this section we make the connection between Bessel functionals for $\GSp(4,F)$ and equivariant functionals on representations of $\GO(X)$. The main result is Theorem \ref{fourdimthetatheorem} below.

Let $(X,\langle\cdot,\cdot\rangle)$ be a non-degenerate symmetric bilinear space over $F$ satisfying \eqref{Xtypeseq}.
We define the subgroup $\GSp(4,F)^+$ of $\GSp(4,F)$ by
\begin{equation}
\label{gsp4plusdef}
\GSp(4,F)^+ = \{ g \in \GSp(4,F): \lambda (g) \in \lambda (\GO(X)) \}.
\end{equation}
The following lemma follows from \eqref{rhoiso2eq} and the exact sequences \eqref{GSOexacteq1} and  \eqref{GSOexacteq2}, which  facilitate the computation of $\lambda(\GSO(X))$. Note that $\Norm(H^\times) = F^\times$. 
\begin{lemma}
\label{gsp4calclemma}
Let $(X,\langle\cdot,\cdot\rangle)$ be a non-degenerate symmetric bilinear space over $F$ satisfying \eqref{Xtypeseq}. Then
\begin{equation}
\label{gsp4pluscomp}
[\GSp(4,F): \GSp(4,F)^+] = 
\begin{cases}
1 & \text{if $\dim X =4$, or $\dim X=2$ and $\disc(X) =1$},\\
2 & \text{if $\dim X =2$ and $\disc(X) \neq 1$}.
\end{cases}
\end{equation}
\end{lemma}

\begin{lemma}
\label{tingsp4lemma}
Let $(X,\langle\cdot,\cdot\rangle)$ be a non-degenerate symmetric bilinear space over $F$ satisfying \eqref{Xtypeseq}, and let $S$ be as in \eqref{Sdefeq} with $\det(S) \neq 0$. Let $\Omega=\Omega_S$ be as in \eqref{omegaSdefeq}, and assume that $\Omega$ is non-empty. Let $T=T_S$ be as in Sect.~\ref{anotheralgebrasubsec}. Embed $T$ as a subgroup of $\GSp(4,F)$, as in \eqref{Tembeddingeq}. Then $T$ is contained in $\GSp(4,F)^+$. 
\end{lemma}
\begin{proof}
By \eqref{gsp4pluscomp} we may assume that $\dim X=2$ and $\disc(X) \neq 1$. Since $\Omega$ is non-empty and $\dim X=2$, we make take $S$ to be the matrix of the symmetric bilinear form $\langle\cdot,\cdot\rangle$ on $X$. By definition, $\GO(X)$ is then the set of $h \in \GL(2,F)$ such that ${}^t h S h = \lambda (h) S$ for some $\lambda (h) \in F^\times$. From \eqref{AFdefeq2}, we have that ${}^t h S h = \det(h) S$ for $h = \left[ \begin{smallmatrix} &1 \\ 1& \end{smallmatrix} \right] t\left[ \begin{smallmatrix} &1 \\ 1& \end{smallmatrix} \right]$ with $t \in T$. It follows that $\det(T)$ is contained in $\lambda (\GO(X))$. This implies that $T$ is contained in $\GSp(4,F)^+$. 
\end{proof}

Let $(X,\langle\cdot,\cdot\rangle)$ be a non-degenerate symmetric bilinear space over $F$ satisfying \eqref{Xtypeseq}.
Define
$$
R=\{ (g,h) \in \GSp(4,F) \times \GO(X):\lambda (g) = \lambda (h) \}.
$$
We consider the Weil representation $\omega$ of $R$ on the space $\mathcal{S}(X^2)$ defined with respect to
$\psi^2$, where $\psi^2(x) = \psi(2x)$ for $x \in F$. If $\varphi \in \mathcal{S}(X^2)$, $g \in \GL(2,F)$ and 
$h \in \GO(X)$ with $\det (g) = \lambda (h)$, and $x_1,x_2 \in X$, then
\begin{gather}
\big( \omega(\begin{bmatrix} 1&& y & z \\ &1&x&y \\ &&1& \\ &&& 1 \end{bmatrix},1) \varphi \big)(x_1, x_2)  = 
\psi(  \langle x_1, x_1 \rangle x +  2\langle x_1,x_2 \rangle y +\langle x_2, x_2 \rangle z)\varphi(x_1,x_2), \label{Znformulaeq}\\
\big( \omega(\begin{bmatrix} g & \\ & \det(g) g' \end{bmatrix},h) \varphi \big) (x_1, x_2)
 =  ( \det(g), \disc(X) )_F \varphi (( \left[\begin{smallmatrix} & 1 \\ 1& \end{smallmatrix}\right] g \left[\begin{smallmatrix} & 1 \\ 1& \end{smallmatrix}\right],h)^{-1}\cdot (x_1,x_2) ). \label{Zaformulaeq}
\end{gather}
For these formulas, see Sect.~1 of \cite{Roberts2001}; note that the additive character we are using is $\psi^2$. Also, in \eqref{Zaformulaeq}
we are using the action of $\GL(2,F) \times \GO(X)$ defined in \eqref{gl2goeq}.  

We will also use the Weil representation $\omega_1$ of
$$
R_1=\{ (g,h) \in \GL(2,F) \times \GO(X):\det(g) = \lambda (h) \}
$$
on $\mathcal{S}(X)$ defined with respect to $\psi^2$. For formulas, again see Sect.~1 of \cite{Roberts2001}. The two Weil representations $\omega$ and $\omega_1$ are related as follows.
\begin{lemma}\label{restrictionlemma}
		The map 
		\begin{equation}
		\label{gsp4gl2eq}
		T:\mathcal{S}(X) \otimes \mathcal{S}(X) \longrightarrow \mathcal{S}(X^2), 
		\end{equation}
		determined by the formula
		\begin{equation}
		\label{gsp4gl2eq2}
		T(\varphi_1 \otimes \varphi_2) (x_1,x_2) = \varphi_1(x_1)\varphi_2(x_2)
		\end{equation}
		for $\varphi_1$ and $\varphi_2$ in $\mathcal{S}(X)$ and $x_1$ and $x_2$ in $X$, is a well-defined complex linear isomorphism such that 
		\begin{align}
		&T\circ (\omega_1\big(\begin{bmatrix} a_2 &b_2 \\ c_2 & d_2 \end{bmatrix},h) \otimes \omega_1(\begin{bmatrix} a_1 &b_1 \\ c_1 & d_1 \end{bmatrix},h)\big)\nonumber\\
		&\qquad =\omega(\begin{bmatrix} a_1&&&b_1\\&a_2&b_2&\\&c_2&d_2&\\c_1&&&d_1\end{bmatrix},h)\circ T \label{gsp4gl2eq3}
		\end{align}
		for $g_1=\left[\begin{smallmatrix} a_1&b_1\\c_1&d_1\end{smallmatrix} \right]$ and $g_2=\left[\begin{smallmatrix}a_2&b_2\\c_2&d_2\end{smallmatrix}\right]$ in $\GL(2,F)$ and $h$ in $\GO(X)$ such that 
		$$
		\det(g_1)=\det(g_2)=\lambda(h).
		$$
\end{lemma}
This lemma can be verified by a direct calculation using standard generators for $\SL(2,F)$.

Let $\theta=\theta_S$ be the character of $N$ defined in \eqref{thetaSsetupeq} with respect to a matrix $S$ as in \eqref{Sdefeq}.
Let $\mathcal{S}(X^2)(N,\theta)$ be the subspace of $\mathcal{S}(X^2)$ spanned by all vectors $\omega(n)\varphi-\theta(n)\varphi$, where $n$ runs through $N$ and $\varphi$ runs through $\mathcal{S}(X^2)$, and set $\mathcal{S}(X^2)_{N,\theta}=\mathcal{S}(X^2)/\mathcal{S}(X^2)(N,\theta)$.
\begin{lemma}
\label{Nthetaomegalemma} (Rallis)
Let $(X,\langle\cdot,\cdot\rangle)$ be a non-degenerate symmetric bilinear space over $F$ satisfying \eqref{Xtypeseq},
and let $S$ be as in \eqref{Sdefeq} with $\det(S) \neq 0$.
 If $(X,\langle \cdot, \cdot \rangle)$ does not represent $S$, then the twisted Jacquet module $\mathcal{S}(X^2)_{N,\theta}$ is zero. Assume that $(X,\langle \cdot, \cdot \rangle)$ represents $S$. The map $\mathcal{S}(X^2) \to \mathcal{S}(\Omega)$ defined by $\varphi\mapsto \varphi|_{\Omega}$ induces an isomorphism
 $$
  \mathcal{S}(X^2)_{N,\theta} \stackrel{\sim}{\longrightarrow} \mathcal{S}(\Omega). 
 $$
 Equivalently, $\mathcal{S}(X^2)(N,\theta)$ is the space of $\varphi \in \mathcal{S}(X^2)$ such that $\varphi|_{\Omega} =0$. 
\end{lemma}
\begin{proof}
See Lemma 2.3 of \cite{KuRa1994}. 
\end{proof}

Let $(X,\langle\cdot,\cdot\rangle)$ be a non-degenerate symmetric bilinear space over $F$ satisfying \eqref{Xtypeseq},
and let $S$ be as in \eqref{Sdefeq} with $\det(S) \neq 0$. Let $\Omega=\Omega_S$ be as in \eqref{omegaSdefeq}, and assume that $\Omega$ is non-empty. In Lemma \ref{ghlemma}
we noted that the subgroup $B$ of $\GL(2,F) \times \GO(X)$ acts on $\Omega$. By identifying $\OO(X)$ with $1\times \OO(X) \subset \GL(2,F) \times \GO(X)$,
we obtain an action of $\OO(X)$ on $\Omega$: this is given by $h\cdot (x_1,x_2) = (hx_1,hx_2)$, where $h \in \OO(X)$ and $(x_1,x_2) \in
\Omega$. This action is transitive. 
We obtain an action of $\OO(X)$ on $\mathcal{S}(\Omega)$ by defining $(h \cdot \varphi) (x) = \varphi(h^{-1} \cdot x)$
for $h \in \OO(X)$, $\varphi \in \mathcal{S}(\Omega)$ and $x \in \Omega$. This action is used in the next lemma. 

\begin{lemma}
\label{Mnonvanishlemma}
Let $(X,\langle\cdot,\cdot\rangle)$ be a non-degenerate symmetric bilinear space over $F$ satisfying \eqref{Xtypeseq},
and let $S$ be as in \eqref{Sdefeq} with $\det(S) \neq 0$.
Let $\Omega=\Omega_S$ be as in \eqref{omegaSdefeq}, and assume that $\Omega$ is non-empty.
Let $(\sigma_0,W_0)$ be an admissible representation of $\OO(X)$, and 
let $M':\mathcal{S}(\Omega) \to W_0$ be a non-zero $\OO(X)$ map.
Let $z \in \Omega$. There exists a compact, open subset $C$ of $\Omega$
containing $z$ such that if $C_0$ is a compact, open subset of $\Omega$
such that $z \in C_0 \subset C$, then $M'(f_{C_0}) \neq 0$. Here, $f_{C_0}$
is the characteristic function of $C_0$. 
\end{lemma}
\begin{proof}
Let $H$ be the subgroup of $h \in \OO(X)$ such that $hz=z$. 
By 1.6 of \cite{BeZe1976}, the map $H \backslash \OO(X) \stackrel{\sim}{\longrightarrow} \Omega$
defined by $Hh \mapsto h^{-1} z$ is a homeomorphism, so that  the 
map $\mathcal{S}(\Omega) \stackrel{\sim}{\longrightarrow} \cInd_H^{\OO(X)} \triv_H$
that sends $\varphi$ to the function $f$ such that $f(h) = \varphi(h^{-1}z)$ for $h \in \OO(X)$
is an $\OO(X)$ isomorphism. Via this isomorphism, we may regard $M'$ as defined
on $\cInd_H^{\OO(X)} \triv_H$, and it will suffice to prove that that there exists
a compact, open neighborhood $C$ of the identity in $\OO(X)$ such that if 
$C_0$ is a compact, open neighborhood of the identity in $\OO(X)$ with $C_0 \subset C$, then
$M'(f_{HC_0}) \neq 0$, where $f_{HC_0}$ is the characteristic function of $HC_0$.
Since $\sigma_0$ is admissible, by 2.15 of \cite{BeZe1976} we have $(\sigma_1)^\vee \cong \sigma_0$  where $\sigma_1=\sigma_0^\vee$.
Let $W_1$ be the space of $\sigma_1$. 
We may regard $M'$ as a non-zero element of $\Hom_{\OO(X)}(\cInd_H^{\OO(X)} \triv_H, \sigma_1^\vee)$. 
Now $H$ and $\OO(X)$ are unimodular since both are orthogonal groups ($H$ is isomorphic to $\OO(U^\perp)$, where $U=Fz_1+Fz_2$). 
By 2.29 of \cite{BeZe1976}, there exists an element $\lambda$ of $\Hom_H(\sigma_1, \triv_H)$ such that $M'$
is given by
$$
M'(f)(v) = \int\limits_{H \backslash \OO(X)} f(h) \lambda (\sigma_1(h) v)\, dh
$$
for $f \in \cInd_H^{\OO(X)} \triv_H$ and $v \in W_1$. Since $M'$
is non-zero, there exists $v \in W_1$ such that $\lambda(v) \neq 0$. Let $C$
be a compact, open neighborhood of $1$ in $\OO(X)$ such that $\sigma_1(h) v =v$ for $h \in C$. Let
$C_0$ be a compact, open neighborhood of $1$ in $\OO(X)$ such that $C_0 \subset C$. Then
\begin{align*}
 M'(f_{HC_0})(v)&= \int\limits_{H \backslash \OO(X)} f_{HC_0}(h) \lambda(\sigma_1(h) v)\, dh\\
&= \int\limits_{H \backslash H C_0}  \lambda(\sigma_1(h) v)\, dh\\
&=\mathrm{vol}(H \backslash H C_0) \lambda (v),
\end{align*}
which is non-zero. 
\end{proof}

In the following theorem we mention the set $\mathcal{E}(z)$ of embeddings of $T$ into $\GSO(X)$; see \eqref{Ecaleq}, \eqref{twodimcalEeq} and \eqref{fourEzeq}. 

\begin{theorem}\label{fourdimthetatheorem}
Let $(X,\langle\cdot,\cdot\rangle)$ be a non-degenerate symmetric bilinear space over $F$ satisfying \eqref{Xtypeseq},
and let $S$ be as in \eqref{Sdefeq} with $\det(S) \neq 0$.  Let $A=A_S$, $T=T_S$, and $L=L_S$ be as in Sect.~\ref{anotheralgebrasubsec}. If $\dim X =4$,
assume that $A$ is a field.  Let $(\pi,V)$ be an irreducible, admissible representation of $\GSp(4,F)^+$, and
let $(\sigma,W)$ be an irreducible, admissible representation of $\GO(X)$. Assume that there is a non-zero $R$ map $M:\mathcal{S}(X^2) \to \pi \otimes \sigma$. Let $\theta=\theta_S$ and let $\Lambda$ be a character of $T$. 
\begin{enumerate}
 \item If $\Hom_N(\pi, \C_\theta) \neq 0$, then  $\Omega = \Omega_S$ is non-empty and $D=TN$ is contained in $\GSp(4,F)^+$.  
 \item Assume that $\Hom_N(\pi,\C_\theta) \neq 0$ so that $\Omega=\Omega_S$ is non-empty, and $D=TN \subset \GSp(4,F)^+$ by i). Assume further that $\Hom_D(\pi,\C_{\Lambda \otimes \theta})\neq 0$. 
Let $z \in \Omega$, and $\tau \in \mathcal{E}(z)$. There exists a non-zero vector $w \in W$ such that 
$$
\sigma(\tau(t)) w =  \Lambda^{-1}(t) w
$$
for $t \in T$. 
\end{enumerate}
\end{theorem}
\begin{proof}
i) The assumptions $\Hom_R(\mathcal{S}(X^2), V \otimes W )\neq 0$ and $\Hom_N(V,\C_\theta)\neq 0$ imply that $\Hom_N(\mathcal{S}(X^2), \C_\theta) \neq 0$. This means that $\mathcal{S}(X^2)_{N,\theta} \neq 0$; by Lemma \ref{Nthetaomegalemma}, we obtain $\Omega \neq \emptyset$. Lemma \ref{tingsp4lemma} now also yields that $D \subset \GSp(4,F)^+$. 

ii) Let $\beta$ be a non-zero element of $\Hom_D(\pi,\C_{\Lambda \otimes \theta})$. 
We first claim that the composition $M'$
$$
\mathcal{S}(X^2) \stackrel{M}{\longrightarrow} V \otimes W \stackrel{\beta \otimes \mathrm{id}}{\longrightarrow} \C_{\Lambda \otimes \theta} \otimes W
$$
is non-zero. 
Let $\varphi \in \mathcal{S}(X^2)$ be such that $M(\varphi) \neq 0$, and write
$$
M(\varphi) = \sum_{\ell =1}^t v_\ell \otimes w_\ell
$$
where $v_1,\dots,v_t \in V$  and $w_1,\dots, w_t \in W$. We may assume that the vectors $w_1,\dots, w_t$ are linearly independent and that $v_1 \neq 0$. Since $\beta$ is non-zero and $V$ is an irreducible representation of $\GSp(4,F)^+$, it follows that there exists $g \in \GSp(4,F)^+$ such that $\beta(\pi(g) v_1)\neq 0$.
Let $h \in \GO(X)$ be such that $\lambda(h)=\lambda(g)$. Then $(g,h) \in R$. Since $M$ is an $R$-map, we have 
$$
M(\omega(g,h) \varphi) = \sum_{\ell =1}^t \pi(g) v_\ell \otimes \sigma(h) w_\ell.
$$
Applying $\beta \otimes \mathrm{id}$ to this equation, we get
$$
M'(\omega(g,h)\varphi) = \sum_{\ell=1}^t \beta (\pi(g) v_\ell) \otimes \sigma(h) w_\ell
$$
in $\C_{\Lambda \otimes \theta} \otimes W$. Since the vectors $\sigma(h) w_1,\dots, \sigma(h) w_t$ are also linearly independent, and since $\beta (\pi(g) v_1)$ is non-zero, it follows that the vector $M'(\omega(g,h)\varphi)$ is non-zero; this proves $M' \neq 0$. 

Next, the map $M'$ induces a non-zero map $\mathcal{S}(X^2)_{N,\theta} \to \C_{\Lambda \otimes \theta}  \otimes W$, which we also denote by $M'$.
Lemma \ref{Nthetaomegalemma}  implies that the restriction map yields an isomorphism $\mathcal{S}(X^2)_{N,\theta} \stackrel{\sim}{\longrightarrow}
\mathcal{S}(\Omega)$. Composing, we thus obtain a non-zero map $\mathcal{S}(\Omega) \to \C_{\Lambda \otimes \theta} \otimes W$, which
we again denote by $M'$. Let $z \in \Omega$ and $\tau \in \mathcal{E}(z)$. By Lemma \ref{ghlemma}, the elements $(\left[ \begin{smallmatrix} &1\\1& \end{smallmatrix} \right] t \left[\begin{smallmatrix}&1\\1&\end{smallmatrix}\right],\tau(t))$ for $t \in T$ act on $\Omega$. We can regard these elements as
acting on $\mathcal{S}(\Omega)$ via the definition $\big( (\left[ \begin{smallmatrix} &1\\1& \end{smallmatrix} \right] t \left[\begin{smallmatrix}&1\\1&\end{smallmatrix}\right],\tau(t)) \cdot \varphi \big)(x) = \varphi ( (\left[ \begin{smallmatrix} &1\\1& \end{smallmatrix} \right] t \left[\begin{smallmatrix}&1\\1&\end{smallmatrix}\right],\tau(t))^{-1} \cdot x)$ for $\varphi \in \mathcal{S}(\Omega)$ and $x \in \Omega$. 
Moreover, by the definition of $M'$  and \eqref{Zaformulaeq}, we have
\begin{equation}
\label{M1transeq}
 M'(  (\left[ \begin{smallmatrix} &1\\1& \end{smallmatrix} \right] t \left[\begin{smallmatrix}&1\\1&\end{smallmatrix}\right],\tau(t))\cdot  \varphi ) = (\det(t), \disc(X))_F \Lambda (t) \sigma(\tau(t)) M'(\varphi)
\end{equation}
for $t \in T$ and $\varphi \in \mathcal{S}(\Omega)$. Let $C$ be the compact, open subset from Lemma \ref{Mnonvanishlemma} with respect to $M'$ and $z$; note that the restriction of $\sigma$ to $\OO(X)$ is admissible. By  Lemma \ref{zstablemma} there exists a compact, open subset $C_0$ of $C$ containing $z$ such that $(\left[ \begin{smallmatrix} &1\\1& \end{smallmatrix} \right] t \left[\begin{smallmatrix}&1\\1&\end{smallmatrix}\right],\tau(t)) \cdot C_0 = C_0$ for $t \in T$. 
Let $\varphi = f_{C_0}$. Then $(\left[ \begin{smallmatrix} &1\\1& \end{smallmatrix} \right] t \left[\begin{smallmatrix}&1\\1&\end{smallmatrix}\right],\tau(t)) \cdot \varphi = \varphi$ for $t \in T$, and by Lemma \ref{Mnonvanishlemma}, we have $M'(\varphi) \neq 0$.  From \eqref{M1transeq} we have
$\sigma (\tau(t))M'(\varphi) = (\det(t), \disc(X))_F \Lambda(t)^{-1} M'(\varphi)=\chi_{L/F}(\Norm_{L/F}(t)) \Lambda(t)^{-1} M'(\varphi)=\Lambda(t)^{-1} M'(\varphi)$ for $t \in T$. Since $M'(\varphi) \neq 0$, this proves ii). 
\end{proof}

Let $(X,\langle\cdot,\cdot\rangle)$ be a non-degenerate symmetric bilinear space over $F$ satisfying \eqref{Xtypeseq}, and let $S$ be as  in \eqref{Sdefeq} with $\det(S) \neq 0$. If $\Omega_S$ is non-empty and $z=(z_1,z_2) \in \Omega_S$, then we let $\OO(X)_z$ be the subgroup of $h \in \OO(X)$ such that $h(z_1)=z_1$ and $h(z_2)=z_2$. 
\begin{proposition}
\label{scdimprop}
Let $(X,\langle\cdot,\cdot\rangle)$ be a non-degenerate symmetric bilinear space over $F$ satisfying \eqref{Xtypeseq}, and assume that $\dim X=4$.  Let $S$ be as 
in \eqref{Sdefeq} with $\det(S) \neq 0$. Assume that $\Omega_S$ is non-empty, and let $z$ be in $\Omega_S$. Let $\varPi$ and $\sigma$ be irreducible, admissible, supercuspidal representations of $\GSp(4,F)$ and $\GO(X)$, respectively. If $\Hom_R(\omega,\varPi \otimes \sigma) \neq 0$, then 
\begin{equation}
\label{Nthetadimeq}
\dim \varPi_{N,\theta_S} = \dim \Hom_{\OO(X)_z} (\sigma, \C_1).
\end{equation}
\end{proposition}
\begin{proof}
Assume that $\Hom_R(\omega,\varPi \otimes \sigma) \neq 0$. By Proposition 3.3 of \cite{Roberts2001} the restriction of $\sigma$ to $\OO(X)$ is multiplicity-free. By Lemma 4.2 of \cite{Roberts1996} we have $\varPi|_{\SSp(4,F)} = \varPi_1 \oplus \dots \oplus \varPi_t$, where $\varPi_1,\dots,\varPi_t$ are mutually non-isomorphic, irreducible, admissible representations of $\SSp(4,F)$, $\sigma|_{\OO(X)} = \sigma_1\oplus \dots \oplus \sigma_t$, where $\sigma_1,\dots, \sigma_t$ are mutually non-isomorphic, irreducible, admissible representations of $\OO(X)$, with $\Hom_{\SSp(4,F) \times \OO(X)} (\omega, \varPi_i \otimes \sigma_i) \neq 0$ for $i \in \{1,\dots,t\}$. Let $i \in \{1,\dots,t\}$; to prove the proposition, it will suffice to prove that $ (\varPi_i)_{N,\theta_S} \cong \Hom_{\OO(X)_z}(\sigma_i,\C_1)$ as complex vector spaces. By Lemma 6.1 of \cite{Roberts1999}, we have $\Theta(\sigma_i)_{N,\theta_S} \cong \Hom_{\OO(X)_z}(\sigma_i^\vee,\C_1)$ as complex vector spaces. 
By 1) a) of the theorem on p.~69 of \cite{MoeglinVignerasWaldspurger1987}, the representation $\Theta(\sigma_i)$ of $\SSp(4,F)$ is irreducible. By Theorem 2.1 of \cite{Kudla1986} we have $\varPi_i \cong \Theta(\sigma_i)$. Therefore, $(\varPi_i)_{N,\theta_S} \cong \Hom_{\OO(X)_z}(\sigma_i^\vee,\C_1)$. By the first theorem on p.~91 of \cite{MoeglinVignerasWaldspurger1987}, $\sigma_i^\vee \cong \sigma_i$. The proposition follows. 
\end{proof}

\subsection{Representations of \texorpdfstring{$\GO(X)$}{}}
\label{goxsubsec}
Let $m,\lambda \in F^\times$. By Lemma  \ref{twodimclasslemma}, the group $\GSO(X_{m,\lambda})$ is abelian. It follows that the irreducible, admissible representations of $\GSO(X_{m,\lambda})$ are characters. To describe the representations of $\GO(X_{m,\lambda})$, let $\mu: \GSO(X_{m,\lambda}) \to \C^\times$ be a character. We recall that the map $\gamma$ from \eqref{gammaendoeq} is a  representative for the non-trivial
coset of $\GSO(X_{m,\lambda})$ in $\GO(X_{m,\lambda})$. Define $\mu^\gamma: \GSO(X_{m,\lambda}) \to \C^\times$ by $\mu^\gamma(x) = \mu (\gamma x \gamma^{-1})$. If $\mu^\gamma \neq \mu$, then the representation $\ind^{\GO(X_{m,\lambda})}_{\GSO(X_{m,\lambda})} \mu$ is irreducible, and we define
$$
\mu^+ = \ind^{\GO(X_{m,\lambda})}_{\GSO(X_{m,\lambda})} \mu. 
$$
Assume that $\mu= \mu^\gamma$. Then the induced representation $\ind^{\GO(X_{m,\lambda})}_{\GSO(X_{m,\lambda})} \mu$ is reducible, and is the direct
sum of the two extensions of $\mu$ to $\GO(X_{m,\lambda})$. We let $\mu^+$ be the extension of $\mu$ to $\GO(X_{m,\lambda})$ such that $\mu^+(\gamma)=1$
and let $\mu^-$ be the extension of $\mu$ to $\GO(X_{m,\lambda})$ such that $\mu^-(\gamma)=-1$. Every irreducible, admissible
representation of $\GO(X_{m,\lambda})$ is of the form $\mu^+$ or $\mu^-$ for some character $\mu$ of $\GSO(X_{m,\lambda})$. We will sometimes identify characters of $\GSO(X_{m,\lambda})$ with characters of $T_{\left[ \begin{smallmatrix} 1 & \\ & -m \end{smallmatrix} \right]}$, via \eqref{rhoiso2eq}, and in turn identify characters of $T_{\left[ \begin{smallmatrix} 1 & \\ & -m \end{smallmatrix} \right]}$ with characters of $L^\times$, via \eqref{ALisoeq}. Here $L$ is associated to $m$, as in Sect.~\ref{quadextsubsec}, so that $L=F(\sqrt{m})$ if $m \notin F^{\times 2}$, and $L = F\times F$ if $m \in F^{\times 2}$. 

Next, let $(X,\langle\cdot,\cdot\rangle)$ be either $(X_{\Mat_2},\langle\cdot,\cdot\rangle_{\Mat_2})$ or $(X_H,\langle\cdot,\cdot\rangle_H)$, as in \eqref{Xmateq} or \eqref{XHeq}. 
If $X=X_{\Mat_2}$, set $G=\GL(2,F)$, and if $X=X_{H}$, set $G=H^\times$. Let $h_0$ be the element of $\GO(X)$ that maps $x$ to $x^*$; then $h_0$ represents the non-trivial coset of $\GSO(X)$ in $\GO(X)$. Let $\pi_1$ and $\pi_2$ be irreducible, admissible representations of $G$ with the same central character. Via the exact sequences \eqref{GSOexacteq1} and \eqref{GSOexacteq2}, the representations $(\pi_1,V_1)$ and $(\pi_2,V_2)$ define an irreducible, admissible representation $\pi_1\otimes\pi_2$ of $\GSO(X)$ which has space $V_1 \otimes V_2$ and action given by the formula $(\pi_1\otimes\pi_2)(\rho(g_1,g_2)) =\pi_1(g_1) \otimes \pi_2(g_2)$ for $g_1,g_2 \in G$. If $\pi_1$ and $\pi_2$ are not isomorphic, then $\pi_1\otimes\pi_2$ induces irreducibly to $\GO(X)$; we denote this induced representation by $(\pi_1\otimes\pi_2)^+$. Assume that $\pi_1$ and $\pi_2$ are isomorphic. In this case the representation $\pi_1\otimes\pi_2$ does not induce irreducibly to $\GO(X)$, but instead has two extensions $\sigma_1$ and 
$\sigma_2$ to representations of $\GO(X)$. Moreover, the space of linear forms on $\pi_1\otimes\pi_2$ that are invariant under the subgroup of $\GSO(X)$ of elements $\rho(g,g^{*-1})$ for $g \in G$ is one-dimensional. Let $\lambda$ be a non-zero functional in this space. Then $\lambda\circ \sigma_i(h_0)$ is another such functional, so that $\lambda\circ \sigma_i(h_0)=\varepsilon_i\lambda$ with $\{\varepsilon_1,\varepsilon_2\}=\{1,-1\}$. The representation $\sigma_i$ for which $\varepsilon_i=1$ is denoted by $(\pi_1\otimes\pi_2)^+$, and the representation $\sigma_j$ for which $\varepsilon_j=-1$ is denoted by $(\pi_1\otimes\pi_2)^-$. See \cite{Roberts1999} for details. 

\begin{proposition}
\label{Ozdimeq}
Let $H$ be as in \eqref{Hdefeq} and let $X_H$ be as in \eqref{XHeq}. Let $S$ be as in \eqref{Sceq} with $-c \notin F^{\times 2}$; we may assume that $i^2 = -c$, as in \eqref{SpecialHeq}. Let $z=(z_1,z_2)$ be as in \eqref{specialijkeq}, so that $z \in \Omega_S$. Set $L=F(\sqrt{-c})$. We have
$$
\dim \Hom_{\OO(X_H)_z} (\sigma_0,\C_1) = 1
$$
for the following families of irreducible, admissible representations $\sigma_0$ of $\GO(X_H)$:
\begin{enumerate}
\item $\sigma_0 = (\sigma 1_{H^\times} \otimes \sigma \chi_{L/F})^+$;
\item $\sigma_0 = (\sigma1_{H^\times} \otimes \sigma \pi^{\mathrm{JL}})^+$.
\end{enumerate}
Here, $\sigma$ is a character of $F^\times$, and $\pi$ is a supercuspidal, irreducible, admissible representation of $\GL(2,F)$ with trivial central character with $\Hom_{L^\times}(\pi^{\mathrm{JL}},\C_1) \neq 0$. 
\end{proposition}
\begin{proof}
We begin by describing $\OO(X_H)_z$. Define $g_1:X_H \to X_H$ by
$$
g_1(1) =1, \quad g_1(i) = i, \quad g_1(j) = j, \quad g_1(k) = -k.
$$
Evidently, $g_1 \in \OO(X_H)_z$, moreover, $\det(g_1)=-1$. It follows that $\OO(X_H)_z = (\SO(X_H) \cap \OO(X_H)_z) \sqcup (\SO(X_H) \cap \OO(X_H)_z) g_1$. Using that $z_1=1$, $z_2=i$, and the fact that every element of $\SO(X_H)$ is of the form $\rho(h_1,h_2)$ for some $h_1,h_2 \in H^\times$, a calculation shows that $\SO(X_H) \cap \OO(X_H)_z$ is $\{\rho(h^*{}^{-1},h): h \in (F+Fi)^\times = L^\times \}$.

i) Since $\sigma_0|_{\OO(X_H)} = (1_{H^\times} \otimes \chi_{L/F})^+$, we may assume that $\sigma=1$. A model for $\sigma_0$ is $\C \oplus \C$, with action defined by 
\begin{align*}
\sigma_0(\rho(h_1,h_2))(w_1 \oplus w_2) & = \chi_{L/F}(\Norm(h_2))w_1 \oplus \chi_{L/F}(\Norm(h_1)) w_2, \\
\sigma_0(*)(w_1 \oplus w_2 ) & = w_2 \oplus w_1
\end{align*}
for $w_1,w_2 \in \C$ and $h_1,h_2 \in H^\times$; here, $*$ is the canonical involution of $H$, regarded as an element of $\OO(X_H)$ with determinant $-1$. Using that $g_1 = * \circ \rho(k^{*-1},k)$, we find that the restriction of $\sigma_0$ to $\OO(X_H)_z$ is given by 
\begin{align*}
\sigma_0(\rho(h^{*-1},h)) (w_1 \oplus w_2)& = w_1 \oplus w_2, \\
\sigma_0(g_1)(w_1 \oplus w_2) & = \chi_{L/F}(\Norm(k))(w_2 \oplus w_1)
\end{align*}
for $w_1,w_2 \in \C$ and $h \in (F+Fi)^\times = L^\times$. Therefore, $\sigma_0|_{\OO(X_H)_z}$ is the direct sum of the trivial character $\OO(X_H)_z$, and the non-trivial character of $\OO(X_H)_z$ that is trivial on $\SO(X_H) \cap \OO(X_H)_z$ and sends $g_1 $ to $-1$. This implies that  $\Hom_{\OO(X_H)_z}(\sigma_0,\C_1)$ is one-dimensional.

ii) Again, we may assume that $\sigma=1$. Let $V$ be the space of $\pi^{\mathrm{JL}}$. As a model for $\sigma_0$ we take $V\oplus V$ with action of $\GO(X_H)$
defined by
\begin{align*}
\sigma_0(\rho(h_1,h_2))(v_1 \oplus v_2)& = \pi^{\mathrm{JL}}(h_2)v_1 \otimes \pi^{\mathrm{JL}}(h_1)v_2, \\
\sigma_0(*)(v_1 \oplus v_2) & = v_2 \oplus v_1
\end{align*}
for $h_1,h_2 \in H^\times$ and $v_1,v_2 \in V$. 
By hypothesis, $\Hom_{L^\times}(\pi^{\mathrm{JL}},\C_1)\neq 0$. This space is one-dimensional; see Sect.~\ref{waldfuncsubsec}. We have $kLk^{-1}=L$; in fact, conjugation by $k$ on $L$ is the non-trivial  element of $\Gal(L/F)$. Since $\Hom_{L^\times}(\pi^{\mathrm{JL}},\C_1)$ is one-dimensional, there exists $\varepsilon \in \{ \pm 1\}$ such that $\lambda \circ \pi^{\mathrm{JL}}(k) = \varepsilon \lambda$ for $\lambda \in \Hom_{L^\times}(\pi^{\mathrm{JL}},\C_1)$. Define a map
$$
\Hom_{L^\times}(\pi^{\mathrm{JL}},\C_1) \longrightarrow \Hom_{\OO(X_H)_z}(\sigma_0,\C_1)
$$
by sending $\lambda$ to $\Lambda$, where $\Lambda$ is defined by $\Lambda(v_1\oplus v_2)=\lambda(v_1)+\varepsilon\lambda(v_2)$ for $v_1,v_2 \in V$. A computation using the fact that $g_1 = * \circ \rho(k^{*-1},k)$ shows that this map is well defined. It is straightforward to verify that this map is injective and surjective, so that
$$
 \Hom_{L^\times}(\pi^{\mathrm{JL}},\C_1)\cong \Hom_{\OO(X_H)_z}(\sigma_0,\C_1).
$$
Hence, $\Hom_{\OO(X_H)_z}(\sigma_0,\C_1)$ is one-dimensional. 
\end{proof}

\subsection{\texorpdfstring{$\GO(X)$}{} and \texorpdfstring{$\GSp(4,F)$}{}}\label{thetasubsec}
In this section we will gather together some information about the theta correspondence between $\GO(X)$ and $\GSp(4)$ when $X$ is as in \eqref{Xtypeseq}. When $\dim(X)=4$, we recall in Theorem \ref{Ganthetatheorem} some results from \cite{GaAt2011} and \cite{GaTa2011}. When $\dim(X)=2$, we calculate two theta lifts, producing representations of type Vd and IXb, in Proposition \ref{thetaliftprop}. This calculation uses $P_3$-theory. We include this material because, to the best of our knowledge, such a computation is absent from the literature.

We let $R_Q$ be the group of elements of $R$ of the form 
$$
		(\begin{bmatrix} *&*&*&*\\&*&*&*\\&*&*&*\\&&&*\end{bmatrix},*). 
$$
Let $Z^J$ be the group defined in \eqref{ZJdefeq}.
\begin{lemma}\label{jacanisosteplemma}
 Let $(X,\langle\cdot,\cdot\rangle)$ be an even-dimensional symmetric bilinear space satisfying \eqref{Xtypeseq}; assume additionally that $X$ is anisotropic.
 There is an isomorphism of complex vector spaces
 \begin{equation}\label{firstisobackeq}
  T_1: \mathcal{S}(X^2)_{Z^J }\;\stackrel{\sim}{\longrightarrow}\;\mathcal{S}(X) 
 \end{equation}
 that is given by 
 $$
  T_1 \big( \varphi +  \mathcal{S}(X^2) (Z^J) \big) (x) = \varphi(x,0)
 $$
 for $\varphi$ in $\mathcal{S}(X^2)$ and $x$ in $X$. The subgroup $R_Q$ of $R$ acts on the quotient $\mathcal{S}(X^2)_{Z^J } $. Transferring this action to $\mathcal{S}(X)$ via $T_1$, the formulas for the resulting action are
 \begin{align}
  ( \begin{bmatrix} t &&& \\ &a&b& \\ &c&d& \\ &&& \lambda(h) t^{-1} \end{bmatrix}, h) \cdot \varphi & = |\lambda(h)|^{-\dim(X)/4}\,(t,\disc(X))_F\,|t|^{\dim(X)/2}\,\omega_1(\begin{bmatrix} a&b \\ c&d \end{bmatrix},h) \varphi, \label{transeq1} \\ 
  (\begin{bmatrix} 1 & x & y & z \\ &1&&y \\ &&1&-x \\ &&&1 \end{bmatrix} ,1) \cdot \varphi & = \varphi \label{transeq2}
 \end{align}
 for $\varphi$ in $\mathcal{S}(X)$, $x,y$ and $z$ in $F$, $t$ in $F^\times$, and $g=\left[\begin{smallmatrix} a&b \\ c&d \end{smallmatrix} \right]$ in $\GL(2,F)$ and $h$ in $\GO(X)$ with $\lambda (h) = \det(g)$. 
\end{lemma}
\begin{proof}
We first claim that
\begin{equation}\label{jacanisoeq1}
		\mathcal{S}(X^2)(Z^J)
		=
		\{ \varphi \in \mathcal{S}(X^2): \varphi(X \times 0) = 0 \}.
\end{equation}
Let $\varphi$ be in $\mathcal{S}(X^2)(Z^J)$. By the lemma in 2.33 of \cite{BeZe1976} there exists a positive integer $n$ so that
\begin{equation}
			\label{jacanisoeq3}
			\int\limits_{\p^{-n}} \omega(\begin{bmatrix} 1&&&b \\ &1&& \\ &&1& \\ &&&1 \end{bmatrix},1)\varphi\, db =0.
\end{equation}
Evaluating at $(x,0)$ and using \eqref{Znformulaeq} shows that $\varphi(X\times0)=0$. Conversely, assume that $\varphi$ is contained in the right hand side of \eqref{jacanisoeq1}. For any integer $k$ let
\begin{equation}
		\label{Lkdefeq}
		L_k=\{x \in X: \langle x, x \rangle \in \p^k \}.
\end{equation}
It is known that $L_k$ is a lattice, i.e., it is a compact and open $\OF$ submodule of $X$; see the proof of Theorem 91:1 of \cite{OMeara1973}. Any lattice is free of rank $\dim X$ as a $\OF$ module. Since $\varphi(X\times0)=0$, there exists a positive integer $n$ such that $\varphi(X\times L_n)=0$. We claim that \eqref{jacanisoeq3} holds. Let $x_1$ and $x_2$ be in $X$. Evaluating \eqref{jacanisoeq3} at $(x_1,x_2)$ gives
$$
			\big( \int\limits_{\p^{-n}} \psi(b \langle x_2,x_2 \rangle)\, db \big) \varphi(x_1,x_2).
$$
This is zero if $x_2$ is in $L_n$ because $\varphi(X \times L_n)=0$. Assume that $x_2$ is not in $L_n$. By the definition of $L_n$, we have $\langle x_2, x_2 \rangle \notin \p^n$. This implies that
$$
			\int\limits_{\p^{-n}} \psi(b \langle x_2,x_2 \rangle)\, db =0,
$$
proving our claim. This completes the proof of \eqref{jacanisoeq1}.

Using \eqref{jacanisoeq1}, it is easy to verify that the map $T_1$ is an isomorphism of vector spaces. Equation \eqref{transeq1} follows from Lemma \ref{restrictionlemma}, and equation \eqref{transeq2} follows from \eqref{Znformulaeq} and \eqref{Zaformulaeq}.
\end{proof}

\begin{proposition}\label{thetaliftprop}
Let $m \in F^\times$, and let $(X_{m,1},\langle\cdot,\cdot\rangle_{m,1})$ be as \eqref{twodimexeq}. Assume that $m\notin F^{\times2}$, so that $X_{m,1}$ is anisotropic. Let $E=F(\sqrt{m})$, and identify characters of $\GSO(X_{m,1})$ and characters of $E^\times$ via \eqref{ALisoeq} and \eqref{rhoiso2eq}. Let $\chi_{E/F}$ be the quadratic character associated to $E$. Let $\varPi$ be an irreducible, admissible representation of $\GSp(4,F)$, and 
let $\sigma$ be an irreducible, admissible representation of $\GO(X_{m,1})$.
 \begin{enumerate}
  \item Assume that $\sigma=\mu^+$ with $\mu=\mu\circ\gamma$, so that $\mu=\alpha\circ\Norm_{E/F}$ for a character $\alpha$ of $F^\times$. Then $\Hom_R(\omega, \varPi^\vee \otimes \sigma) \neq 0$ if and only if $\varPi=L(\nu \chi_{E/F}, \chi_{E/F} \rtimes \nu^{-1/2} \alpha)$ (type Vd).
  \item Assume that $\sigma=\mu^+= \ind^{\GO(X_{m,1})}_{\GSO(X_{m,1})}(\mu)$ with $\mu \neq \mu \circ \gamma$. Then $\Hom_R(\omega, \varPi^\vee \otimes \sigma) \neq 0$ if and only if $\varPi=L(\nu \chi_{E/F}, \nu^{-1/2} \pi(\mu))$ (type IXb). Here, $\pi(\mu)$ is the supercuspidal, irreducible, admissible representation of $\GL(2,F)$ associated to $\mu$.
 \end{enumerate}
\end{proposition}
\begin{proof}
Let $(\sigma,W)$ be as in i) or ii). In the case of i), set $\pi(\mu)=\alpha\times\alpha\chi_{E/F}$. Then $\Hom_{R_1}(\omega_1,\pi(\mu)^\vee\otimes\sigma)\neq0$, and $\pi(\mu)$ is the unique irreducible, admissible representation of $\GL(2,F)$ with this property, by Theorem 4.6 of \cite{JacquetLanglands1970}.

Let $(\varPi',V)$ be an irreducible, admissible representation of $\GSp(4,F)$ such that $\Hom_R(\omega, \varPi' \otimes \sigma) \neq 0$. Let $T$ be a non-zero element of this space. The non-vanishing of $T$ implies that the central characters of $\varPi'$ and $\sigma$ satisfy
\begin{equation}\label{thetaliftpropeq2}
 \omega_{\varPi'}=\omega_\sigma^{-1}=(\mu|_{F^\times})^{-1}.
\end{equation}
We first claim that $V$ is non-supercuspidal. By reasoning as in \cite{GK}, there exist $\lambda_1,\dots, \lambda_t$ in $F^\times$ and an irreducible $\SSp(4,F)$ subspace $V_0$ of $V$ such that 
$$
 V=V_1\oplus \dots\oplus V_t,
$$
where
\begin{equation}\label{thetaliftpropeq1}
			V_1 = \pi(\begin{bmatrix} 1&&& \\ &1&& \\ &&\lambda_1& \\ &&&\lambda_1 \end{bmatrix}) V_0\quad ,\dots, \quad
			V_t = \pi(\begin{bmatrix} 1&&& \\ &1&& \\ &&\lambda_t& \\ &&&\lambda_t \end{bmatrix}) V_0.
\end{equation}
Similarly, there exist irreducible $\OO(X)$ subspaces $W_1,\dots, W_r$ of $W$ such that 
$$
			W=W_1 \oplus \dots \oplus W_r.
$$
There exists an $i$ and a $j$ such that $\Hom_{\SSp(4,F)\times\OO(X)}(\omega,V_i\otimes W_j)\neq0$. As in the proof of Lemma 4.2 of \cite{Roberts1996}, there is an irreducible constituent $U_1$ of $\pi(\mu)^\vee$ such that $\Hom_{\OO(X)}(\omega_1,U_1\otimes W_j)\neq0$.
By Theorem 4.4 of \cite{Roberts1998}, the representation $V_i$ is non-supercuspidal, so that $V$ is non-supercuspidal.

Since $V$ is non-supercuspidal, we have $V_{Z^J}\neq0$ by Tables A.5 and A.6 of \cite{NF} (see the comment after Theorem \ref{finitelength}). We claim next that $\Hom_{R_Q}(\mathcal{S}(X^2)_{Z^J},V_{Z^J}\otimes W)\neq0$. It follows from \eqref{thetaliftpropeq1} that $(V_i)_{Z^J}\neq0$. Let $p_i:\:V\to V_i$ and $q_j:\:W\to W_j$ be the projections. These maps are $\SSp(4,F)$ and $\OO(X)$ maps, respectively. The composition
$$
			\mathcal{S}(X^2) \stackrel{T}{\longrightarrow} V \otimes W \stackrel{p_i \otimes q_j}{\longrightarrow} V_i \otimes W_j \longrightarrow (V_i)_{Z^J} \otimes W_j
$$
is non-zero and surjective; note that $V_i\otimes W_j$ is irreducible. The commutativity of the diagram
$$
			\begin{CD}
			\mathcal{S}(X^2) @>T>> V\otimes W @>p_i \otimes q_j >> V_i \otimes W_j \\
			@. @VVV @VVV\\
			@. V_{Z^J } \otimes W @>p_i \otimes q_j>> (V_i)_{Z^J } \otimes W_j
			\end{CD}
$$
implies our claim that $\Hom_{R_Q}(\mathcal{S}(X^2)_{Z^J},V_{Z^J}\otimes W)\neq0$.

Let $R_{\bar Q}$ be the subgroup of $R_Q$ consisting of the elements of the form
$$
		(\begin{bmatrix} *&*&*&* \\ &*&*&* \\ &*&*&* \\ &&&1 \end{bmatrix}, *).
$$
Let $R_{P_3}$ be the subgroup of $P_3 \times \GO(X)$ consisting of the elements of the form
$$
 (\begin{bmatrix} a&b&x \\ c&d&y \\ &&1 \end{bmatrix},h),\qquad ad-bc=\lambda(h).
$$
There is a homomorphism from $R_{\bar Q}$ to $R_{P_3}$ given by 
$$
 (\begin{bmatrix} *&*&*&* \\ &a&b&x \\ &c&d&y \\ &&&1 \end{bmatrix},h) \mapsto (\begin{bmatrix} a&b&x \\ c&d&y \\ &&1 \end{bmatrix},h)
$$
for $\begin{bmatrix} a&b \\ c&d \end{bmatrix}$ in $\GL(2,F)$, $x$ and $y$ in $F$, and $h$ in $\GO(X)$ with $ad-bc=\lambda(h)$. We consider $Z^J$ a subgroup of $R_{\bar Q}$ via $z\mapsto(z,1)$. The above homomorphism then induces an isomorphism $R_{\bar Q}/Z^J\cong R_{P_3}$.

We restrict the $R_Q$ modules $\mathcal{S}(X^2)_{Z^J}$ and $V_{Z^J}\otimes W$ to $R_{\bar Q}$. The subgroup $Z^J$ of $R_{\bar Q}$ acts trivially, so that these spaces may be viewed as $R_{P_3}$ modules.

Let $\chi$ be a character of $F^\times$. We assert that
\begin{align}
 &\mathrm{Hom}_{R_{P_3}}(\mathcal{S}(X^2)_{Z^J}, \tau_{\GL(0)}^{P_3}(1) \otimes \sigma)=0\label{gl0gl1noeq1},\\
 &\mathrm{Hom}_{R_{P_3}}(\mathcal{S}(X^2)_{Z^J}, \tau_{\GL(1)}^{P_3}(\chi) \otimes \sigma) = 0. \label{gl0gl1noeq2}
\end{align}
Let $\tau$ be $\tau_{\GL(0)}^{P_3}(1)$ or $\tau_{\GL(1)}^{P_3}(\chi)$. Assume that \eqref{gl0gl1noeq1} or \eqref{gl0gl1noeq2} is non-zero; we will obtain a contradiction. Let $S$ be a non-zero element of \eqref{gl0gl1noeq1} or \eqref{gl0gl1noeq2}. Since $S$ is non-zero, there exists $\varphi$ in $\mathcal{S}(X^2)_{Z^J}$ such that $S(\varphi)$ is non-zero. Write $S(\varphi)=\sum_{i=1}^t f_i \otimes w_i$ for some $f_1,\dots,f_t$ in the standard space of $\tau$ and $w_1,\dots,w_t$ in $W$. The elements $f_1,\dots,f_t$ are functions from $P_3$ to $\C$ such that 
$$
 f_i(\begin{bmatrix} 1&&x\\&1&y\\ &&1 \end{bmatrix} p) = \psi(y) f_i(p)
$$
for $x$ and $y$ in $F$, $p$ in $P_3$, and $i=1,\dots,t$. We may assume that the vectors $w_1,\dots, w_t$ are linearly independent, and that there exists $p$ in $P_3$ such that $f_1(p)$ is non-zero. Using the transformation properties of $S$ and $f_1$, we may assume that 
$$
 p=\begin{bmatrix}a\\&1\\ &&1 \end{bmatrix}.
$$
Let $\lambda: \sigma \to \C$ be a linear functional such that $\lambda (w_1)=1$ and $\lambda(w_2) = \dots = \lambda(w_t) =0$, and let $e:\tau \to \C$ be the linear functional that sends $f$ to $f(p)$. The composition $(e \otimes \lambda) \circ S$ is non-zero on $\varphi$. On the other hand, using \eqref{transeq2}, for $y$ in $F$ we have
\begin{align*}
			\big( (e \otimes \lambda) \circ S \big) ( ( \begin{bmatrix}1&& \\ &1&y \\ &&1 \end{bmatrix},1) \varphi )
			&= (e \otimes \lambda)  \big( (\begin{bmatrix} 1&& \\ &1&y \\ &&1 \end{bmatrix},1) \cdot S (\varphi)\big),\\
			\big( (e \otimes \lambda) \circ S \big) (  \varphi )&= (e \otimes \lambda)  \big( (\begin{bmatrix} 1&& \\ &1&y \\ &&1 \end{bmatrix},1) \cdot \sum_{i=1}^t f_i \otimes w_i \big)\\
			&= \sum_{i=1}^t f_i(p\begin{bmatrix} 1&& \\ &1&y \\ &&1 \end{bmatrix}) \lambda(w_i) \\
			&=\psi(y) f_1(p),\\
			\big( (e \otimes \lambda) \circ S \big) (  \varphi )&=\psi(y) \big( (e \otimes \lambda) \circ S \big) (  \varphi ).
\end{align*}
This is a contradiction since $\big( (e \otimes \lambda) \circ S \big) (  \varphi )$ is non-zero, and there exist $y$ in $F$ such that $\psi(y) \neq 1$. This concludes the proof of \eqref{gl0gl1noeq1} and \eqref{gl0gl1noeq2}.

It follows from \eqref{gl0gl1noeq1} and \eqref{gl0gl1noeq2} and the non-vanishing of
$\Hom_{R_{P_3}}(\mathcal{S}(X^2)_{Z^J},V_{Z^J}\otimes W)$ that there exists an irreducible, admissible representation $\rho$ of $\GL(2,F)$ that occurs in the $P_3$ filtration of $V_{Z^J}$ (Theorem \ref{finitelength}) such that $\Hom_{R_{P_3}}(\mathcal{S}(X^2)_{Z^J},\tau_{\GL(2)}^{P_3}(\rho)\otimes W)\neq0$. It follows from \eqref{transeq1} that $\Hom_{R_1}(\omega_1,\nu^{-1/2}\chi_{E/F}\rho\otimes\sigma)\neq0$. By the uniqueness stated in the first paragraph of this proof, it follows that
\begin{equation}\label{thetaliftpropeq3}
 \rho=\nu^{1/2}\chi_{E/F}\,\pi(\mu)^\vee.
\end{equation}
As a consequence, $\omega_\rho=\nu(\mu|_{F^\times})^{-1}\chi_{E/F}$. Together with \eqref{thetaliftpropeq2}, it follows that
\begin{equation}\label{thetaliftpropeq4}
 \omega_{\varPi'}=\chi_{E/F}\nu^{-1}\omega_\rho.
\end{equation}
Going through Table A.5 of \cite{NF}, we see that only the $\varPi'=\varPi^\vee$ with $\varPi$ as asserted in i) and ii) satisfy both \eqref{thetaliftpropeq3} and \eqref{thetaliftpropeq4}. (Observe the remark made after Theorem \ref{finitelength}.)

Conversely, assume that $\varPi$ is as in i) or ii). Since $\Hom_{R_1}(\omega_1,\pi(\mu)^\vee\otimes\sigma)\neq0$, we have $\Hom_{\OO(X)}(\mathcal{S}(X^2),\sigma)\neq0$ by, for example, Remarque b) on p.~67 of \cite{MoeglinVignerasWaldspurger1987}. Arguing as in Theorem 4.4 of \cite{Roberts1996}, there exists some irreducible, admissible representation $\varPi'$ of $\GSp(4,F)$ such that $\Hom_R(\omega, \varPi' \otimes \sigma) \neq 0$. By what we proved above, $\varPi'=\varPi^\vee$. This concludes the proof.
\end{proof}

\begin{theorem}[\cite{GaAt2011},\cite{GaTa2011}]\label{Ganthetatheorem}
Let $(X,\langle\cdot,\cdot\rangle)$ be either $(X_{\Mat_2},\langle\cdot,\cdot\rangle_{\Mat_2})$ or $(X_H,\langle\cdot,\cdot\rangle_H)$, as in \eqref{Xmateq} or \eqref{XHeq}. If $X=X_{\Mat_2}$, set $G=\GL(2,F)$, and if $X=X_{H}$, set $G=H^\times$.
Let $\varPi$ be an irreducible, admissible representation of $\GSp(4,F)$, and let $\pi_1$ and $\pi_2$ be irreducible, admissible representations of $G$ with the same central character.  We have
 $$
  \Hom_R(\omega, \varPi^\vee \otimes (\pi_1 \otimes \pi_2)^+) \neq 0
  $$
for  $\varPi$, $\pi_1$ and $\pi_2$ as in the following table:
$$
\renewcommand{\arraystretch}{1.1}
\begin{array}{clccc}
\toprule
\multicolumn{2}{c}{\text{type of $\varPi$}}&\varPi&\pi_1 & \pi_2\\
\toprule
\mathrm{I}&& \chi_1 \times \chi_2 \rtimes \sigma & \sigma \chi_1 \chi_2 \times \sigma & \sigma \chi_1 \times \sigma \chi_2 \\
\midrule
\mathrm{II}&\mathrm{a}&\chi\St_{\GL(2)} \rtimes \sigma &  \sigma \chi^2  \times \sigma & \sigma \chi  \St_{\GL(2)}\\
\cmidrule{2-5}
&\mathrm{b}&\chi 1_{\GL(2)} \rtimes \sigma & \sigma \chi^2  \times \sigma &  \sigma \chi  1_{\GL(2)} \\
\midrule
\mathrm{III}& \mathrm{b} & \chi \rtimes \sigma 1_{\GSp(2)} &  \sigma \chi \nu^{1/2} \times \sigma \nu^{-1/2} & \sigma \chi \nu^{-1/2} \times \sigma \nu^{1/2}\\
\midrule
\mathrm{IV} & \mathrm{c}&  L(\nu^{3/2}\St_{\GL(2)},\nu^{-3/2}\sigma) & \sigma\nu^{3/2} \times \sigma \nu^{-3/2} &  \sigma \St_{\GL(2)} \\
\cmidrule{2-5}
&\mathrm{d}&  \sigma 1_{\GSp(4)} &  \sigma\nu^{3/2} \times \sigma \nu^{-3/2} & \sigma 1_{\GL(2)} \\
\midrule
\mathrm{V}&\mathrm{a} & \delta([\xi, \nu \xi],\nu^{-1/2}\sigma) & \sigma \St_{\GL(2)} & \sigma \xi \St_{\GL(2)} \\
\cmidrule{2-5}
&\mathrm{a^*} & \delta^*([\xi, \nu \xi],\nu^{-1/2}\sigma) &  \sigma 1_{H^\times} & \sigma \xi 1_{H^\times} \\
\cmidrule{2-5}
&\mathrm{b} & L(\nu^{1/2} \xi \St_{\GL(2)}, \nu^{-1/2} \sigma) &  \sigma 1_{\GL(2)} & \sigma \xi \St_{\GL(2)} \\
\cmidrule{2-5}
&\mathrm{d} & L(\nu  \xi \rtimes \nu^{-1/2} \sigma) &  \sigma 1_{\GL(2)} &  \sigma \xi 1_{\GL(2)} \\
\midrule
\mathrm{VI} &\mathrm{a} & \tau(S,\nu^{-1/2}\sigma) &  \sigma \St_{\GL(2)} & \sigma \St_{\GL(2)} \\
\cmidrule{2-5}
&\mathrm{b} & \tau(T,\nu^{-1/2}\sigma) & \sigma 1_{H^\times} & \sigma 1_{H^\times} \\
\cmidrule{2-5}
&\mathrm{c} & L(\nu^{1/2}\St_{\GL(2)},\nu^{-1/2}\sigma) & \sigma 1_{\GL(2)} & \sigma \St_{\GL(2)} \\
\cmidrule{2-5}
&\mathrm{d} & L(\nu,1_{F^\times} \rtimes \nu^{-1/2} \sigma ) & \sigma 1_{\GL(2)} & \sigma 1_{\GL(2)} \\
\midrule
\mathrm{VIII} & \mathrm{a} & \tau(S,\pi) & \pi & \pi \\
\cmidrule{2-5}
 & \mathrm{b} & \tau(T,\pi) & \pi^{\mathrm{JL}} & \pi^{\mathrm{JL}} \\
\midrule
\mathrm{X} && \pi \rtimes \sigma &\sigma \omega_\pi \times \sigma &\pi  \\
\midrule
\mathrm{XI} & \mathrm{a} & \delta(\nu^{1/2} \pi , \nu^{-1/2} \sigma) & \sigma \St_{\GL(2)}& \sigma \pi \\
\cmidrule{2-5}
 & \mathrm{a^*} & \delta^*(\nu^{1/2} \pi , \nu^{-1/2} \sigma) & \sigma 1_{H^\times} & \sigma \pi^{\mathrm{JL}} \\
\cmidrule{2-5}
 & \mathrm{b} & L(\nu^{1/2} \pi, \nu^{-1/2} \sigma) & \sigma 1_{\GL(2)} & \sigma \pi\\
\bottomrule
\end{array}
$$
 The notation $\pi^{\mathrm{JL}}$ in the table denotes the Jacquet-Langlands lifting of the supercuspidal representation $\pi$ of $\GL(2,F)$ to a representation of $H^\times$. See Sect.~\ref{goxsubsec} for the definitions of the $+$ representation. 
\end{theorem}

\subsection{Applications}\label{thetaapplicationssec}
We now apply Theorem \ref{fourdimthetatheorem} along with knowledge of the theta correspondences of the previous section to obtain results about Bessel functionals.

\begin{corollary}\label{fourdimthetatheoremcor1}
Let $(X,\langle\cdot,\cdot\rangle)$ be either $(X_{\Mat_2},\langle\cdot,\cdot\rangle_{\Mat_2})$ or $(X_H,\langle\cdot,\cdot\rangle_H)$, as in \eqref{Xmateq} or \eqref{XHeq}. If $X=X_{\Mat_2}$, set $G=\GL(2,F)$, and if $X=X_{H}$, set $G=H^\times$.
Let $\varPi$ be an irreducible, admissible representation of $\GSp(4,F)$, and let $\pi_1$ and $\pi_2$ be irreducible, admissible representations of $G$ with the same central character. Assume that
$$
\Hom_R(\omega, \varPi^\vee \otimes (\pi_1 \otimes \pi_2)^+) \neq 0
$$
and that $\varPi$ has a non-split $(\Lambda,\theta)$-Bessel functional with $\theta=\theta_S$. Then 
$$
\Hom_{T}(\pi_1,\C_\Lambda)\neq0 \quad \text{and}\quad \Hom_{T}(\pi_2,\C_\Lambda)\neq0
$$
where $T=T_S$.
\end{corollary}
\begin{proof}
The assumption that the Bessel functional is non-split means that $A=A_S$ is a field.
By Sect.~\ref{twobytwosubsec} and Sect.~\ref{actionsubsec} we may assume that $S$ has the diagonal form \eqref{Sceq}.  By \eqref{besselcontragredientformula}, the contragredient $\varPi^\vee$ has a $((\Lambda\circ\gamma)^{-1},\theta)$-Bessel functional. The assertion follows now from  Theorem \ref{fourdimthetatheorem}, the explicit embeddings in \eqref{Ezmateq} and \eqref{EzHeq}, and the relation \eqref{waldspurgerpropeq1}.
\end{proof}

\begin{corollary}\label{bsummary}
 Let $(\varPi,V)$ be an irreducible, admissible representation of $\GSp(4,F)$. If $\varPi$ is one of the representations in the following table, then $\varPi$ admits a non-zero $(\Lambda,\theta)$-Bessel functional $\beta$ if and only if the quadratic extension $L$ associated to $\beta$, and $\Lambda$, regarded as a character of $L^\times$, are as specified in the table.
 $$
\renewcommand{\arraystretch}{1.2}
\setlength{\arraycolsep}{0.3cm}
\begin{array}{cccc}
\toprule
\text{type of $\varPi$} & \varPi & L & \Lambda\\
 \toprule
 {\text{\rm Va$^*$}}&\delta^*([\chi_{E/F},\nu \chi_{E/F}],\nu^{-1/2}\alpha) &E  & \alpha \circ \Norm_{E/F}\\
\cmidrule{1-4}
 {\rm Vd}& L(\nu \chi_{E/F}, \chi_{E/F} \rtimes \nu^{-1/2} \alpha) & E & \alpha \circ \Norm_{E/F} \\
\cmidrule{1-4}
 {\rm IXb} & L(\nu \chi_{E/F}, \nu^{-1/2} \pi(\mu)) & E & \text{$\mu$ and the Galois conjugate of $\mu$}\\
\bottomrule
\end{array}
 $$
\end{corollary}
\begin{proof}
First we consider the Va$^*$ case. Let $\varPi=\delta^*([\chi_{E/F},\nu \chi_{E/F}],\nu^{-1/2}\alpha)$. By Theorem \ref{Ganthetatheorem},
$$
 \Hom_R(\omega,\varPi^\vee\otimes(\alpha\triv_{H^\times}\otimes\alpha\chi_{E/F}\triv_{H^\times})^+)\neq0.
$$
First, assume  that $\varPi$ admits a non-zero $(\Lambda,\theta)$-Bessel functional, and let $L$ be the quadratic extension
associated to this Bessel functional; we will prove that $E=L$ and that $\Lambda=\alpha\circ\Norm_{E/F}$. By v) of Proposition \ref{nongenericsplitproposition}, this Bessel functional is non-split. It follows from Corollary \ref{fourdimthetatheoremcor1} that
$$
 \alpha(\Norm_{L/F}(t))=\Lambda(t)\qquad\text{and}\qquad(\chi_{E/F}\alpha)(\Norm_{L/F}(t))=\Lambda(t)
$$
for $t$ in $T=L^\times$. It follows that $E=L$, and that $\Lambda=\alpha\circ\Norm_{E/F}$.

Finally, we prove that Va$^*$ admits a Bessel functional as specified in the statement of the corollary. By Theorem \ref{existencetheorem} below, Va$^*$ admits some non-zero Bessel functional. The previous paragraph proves that this Bessel functional must be as described in the statement of the corollary.

The arguments for the cases Vd and IXb are similar; we will only consider the case of type IXb.
Let $\varPi=L(\nu\chi_{E/F},\nu^{-1/2}\pi(\mu))$, where $E$ is a quadratic
extension of $F$, $\chi_{E/F}$ is the quadratic 
character associated to $E/F$, $\mu$ is a character of $E^\times$ that is not Galois invariant,
and $\pi(\mu)$ is the supercuspidal, irreducible, admissible representation of $\GL(2,F)$ associated to $\mu$.

First, assume  that $\varPi$ admits a non-zero $(\Lambda,\theta)$-Bessel functional, and let $L$ be the quadratic extension
associated to this Bessel functional; we will prove that $E=L$, and that $\Lambda$ is $\mu$ or the Galois
conjugate of $\mu$. 
Let $S$ define $\theta$, as in \eqref{thetaSsetupeq}. 
By \eqref{besselcontragredientformula}, $\varPi^\vee$ admits a non-zero $((\Lambda \circ \gamma)^{-1},\theta)$ Bessel functional $\beta$. 
Write $E=F(\sqrt{m})$ for some $m \in F^\times$. By Proposition \ref{thetaliftprop} we have $\Hom_R(\omega, \varPi^\vee \otimes \mu^+) \neq 0$
with $\mu^+$ as in this proposition. The involved symmetric bilinear space is $(X_{m,1},\langle\cdot,\cdot \rangle_{m,1})$. Let
$\GSp(4,F)^+$ be defined with respect to $(X_{m,1},\langle\cdot,\cdot \rangle_{m,1})$ as in \eqref{gsp4plusdef}. By
Lemma \ref{gsp4calclemma} the index of $\GSp(4,F)^+$ in $\GSp(4,F)$ is two. By Lemma 2.1 of \cite{GK}, the restriction of
$\varPi^\vee$ to $\GSp(4,F)^+$ is irreducible or the direct sum of two non-isomorphic irreducible, admissible representations
of $\GSp(4,F)^+$; the non-vanishing of $\Hom_R(\omega, \varPi^\vee \otimes \mu^+)$ and Lemma 4.1 of \cite{Roberts2001} (with $m=2$ and $n=2$) imply that $V^\vee = V_1   \oplus V_2$ with $V_1$ and $V_2$ irreducible $\GSp(4,F)^+$ subspaces of $V$. 
Moreover, for each $i \in \{1,2\}$, there exists $\lambda_i \in F^\times$ such that $\varPi(\left[\begin{smallmatrix} 1 & \\ & \lambda_i \end{smallmatrix} \right]) V_1 =V_i$. Since $\Hom_R(\omega, V^\vee \otimes \mu^+)$ is non-zero, we may assume, after possibly renumbering,  that $\Hom_R(\omega, V_1 \otimes \mu^+)$ is non-zero. 
There exists $i \in \{1,2\}$ such that the restriction of $\beta$ to $V_i$ is non-zero. Let $\beta' = \left[\begin{smallmatrix} 1 & \\ & \lambda_i^{-1} \end{smallmatrix} \right]\cdot \beta$. From Sect.~\ref{actionsubsec} it follows that $\beta'$ is a $((\Lambda \circ \gamma)^{-1}, \theta')$ Bessel functional on $\varPi^\vee$ with $\theta'$ defined by $S'= \lambda_i^{-1} S$; also, the restriction of $\beta'$ to $V_1$ is non-zero. We will now apply
Theorem \ref{fourdimthetatheorem}, with $S'$ and $V_1$ playing the roles of $S$ and $\pi$, respectively.  By i) of this theorem we have that $\Omega_{S'}$ is non-empty; since $S$ and $S'$ have the same discriminant, Lemma \ref{lambdaSlemma}  implies that $L = E$.  Let $z \in \Omega_{S'}$ and $\tau \in \mathcal{E}(z)$. By ii) of Theorem \ref{fourdimthetatheorem}, there exists a non-zero vector $w$ in the space of $\mu^+$ such that $\mu^+(\tau(t)) w = (\Lambda \circ \gamma)(t)w$ for $t \in T_{S'}$. By Lemma \ref{lambdaSlemma} again, this implies that $\mu^+(\rho(x)) w = \Lambda(x) w $ for $x \in L^\times$, or $\mu^+(\rho(\gamma(x))) w = \Lambda(x) w $ for $x \in L^\times$. Since $w \neq 0$, the definition of $\mu^+$ now implies that $\Lambda = \mu$ or $\mu \circ \gamma$, as desired. 

Finally, we prove that $\varPi$ admits Bessel functionals as specified in the statement of the corollary. By Theorem \ref{existencetheorem} below, $\varPi$ admits some non-zero Bessel functional. The previous paragraph proves that this Bessel functional must be as described in the statement of the corollary, and \eqref{besselgaloiseq} implies that the $\varPi$ admits both of the asserted Bessel functionals. 
\end{proof}

The following result will imply uniqueness of Bessel functionals for representations of type Va$^*$ and XIa$^*$.
\begin{corollary}
\label{Vastarlemma}
Let $\sigma$ be a character of $F^\times$. Let $c \in F^\times$ with $-c \notin F^{\times 2}$. Let $S$ be as in \eqref{Sceq} and set $L=F(\sqrt{-c})$.
\begin{enumerate}
\item If $\xi=\chi_{L/F}$, then  $\dim \delta^*([\xi,\nu\xi],\nu^{-1/2}\sigma)_{N,\theta_S}=1$.
\item If $\pi$ is an irreducible, admissible, supercuspidal representation of $\GL(2,F)$ with trivial central character such that $\Hom_{L^\times} (\pi^{\mathrm{JL}},\C_1)\neq0$, then $\dim \delta^*(\nu^{1/2}\pi,\nu^{-1/2}\sigma)_{N,\theta_S}=1$.
\end{enumerate}
\end{corollary}
\begin{proof}
This follows from Proposition \ref{scdimprop}, Theorem \ref{Ganthetatheorem},  and Proposition \ref{Ozdimeq}; note that $\varPi^\vee|_{N} \cong \varPi|_{N}$ for irreducible, admissible representations $\varPi$ of $\GSp(4,F)$ because $\varPi^\vee \cong \omega_{\varPi}^{-1} \varPi$. 
\end{proof}

\section{Twisted Jacquet modules of induced representations}
Let $(\pi,V)$ be an irreducible, admissible representation of $\GSp(4,F)$. In view of the isomorphism \eqref{DTJacqueteq}, understanding the possible Bessel functionals of $\pi$ is equivalent to understanding the twisted Jacquet modules $V_{N,\theta}$ as $T$-modules. In this section, we will calculate the twisted Jacquet modules for representations induced from the Siegel and Klingen parabolic subgroup. This information will be used to determine the possible Bessel functionals for many of the non-supercuspidal representations of $\GSp(4,F)$; see Sect.\ \ref{maintheoremproofsec}.

The results of this section are similar to Proposition 2.1 and 2.3 of \cite{PrTa2011}. However, we prefer to redo the arguments, as those in \cite{PrTa2011} contain some inaccuracies.

\subsection{Two useful lemmas}
For a positive integer $n$ let $\mathcal{S}(F^n)$ be the Schwartz space of $F^n$, meaning the space of locally constant, compactly supported functions $F^n\rightarrow\C$. As before, $\psi$ is our fixed non-trivial character of $F$.

Let $V$ be a complex vector space. Let $\mathcal{S}(F,V)$ be the space of compactly supported, locally constant functions from $F$ to $V$. There is a canonical isomorphism $\mathcal{S}(F,V)\cong\mathcal{S}(F)\otimes V$. The functional on $\mathcal{S}(F)$ given by $f\mapsto\int_F f(x)\,dx$ gives rise to a linear map $\mathcal{S}(F)\otimes V\rightarrow V$, and hence to a linear map $\mathcal{S}(F,V)\rightarrow V$. We write this map as an integral
$$
 f\longmapsto \int\limits_Ff(x)\,dx.
$$
The following lemma will be frequently used when we calculate Jacquet modules in the subsequent sections.
\begin{lemma}\label{basicFjacquetlemma}
 Let $\rho$ denote the action of $F$ on $\mathcal{S}(F,V)$ by translation, i.e., $(\rho(x)f)(y)=f(x+y)$. Let $\rho'$ be the action of $F$ on $\mathcal{S}(F,V)$ given by $(\rho'(x)f)(y)=\psi(xy)f(y)$.
 \begin{enumerate}
  \item The map $f\mapsto\int_Ff(x)\,dx$ induces an isomorphism
   $$
    \mathcal{S}(F,V)/\langle f-\rho(x)f\::\:x\in F\rangle\cong V.
   $$
  \item The map $f\mapsto\int_F\psi(-x)f(x)\,dx$ induces an isomorphism
   $$
    \mathcal{S}(F,V)/\langle\psi(x)f-\rho(x)f\::\:x\in F\rangle\cong V.
   $$
  \item The map $f\mapsto f(0)$ induces an isomorphism
   $$
    \mathcal{S}(F,V)/\langle f-\rho'(x)f\::\:x\in F\rangle\cong V.
   $$
 \end{enumerate}
\end{lemma}
\begin{proof}
By the Proposition in 1.18 of \cite{BeZe1976}, every translation-invariant functional on $\mathcal{S}(F)$ is a multiple of the Haar measure $f\mapsto\int_Ff(x)\,dx$. This proves i) in the case where $V=\C$. The general case follows from this case by tensoring the exact sequence
$$
 0\longrightarrow\langle f-\rho(x)f\::\:x\in F,\:f\in\mathcal{S}(F)\rangle\longrightarrow \mathcal{S}(F)\longrightarrow\C\longrightarrow0
$$
by $V$. Under the isomorphism $\mathcal{S}(F)\otimes V\cong\mathcal{S}(F,V)$, the space $\langle f-\rho(x)f\::\:x\in F,\:f\in\mathcal{S}(F)\rangle\otimes V$ maps onto $\langle f-\rho(x)f\::\:x\in F,\:f\in\mathcal{S}(F,V)\rangle$.

To prove ii), observe that there is an isomorphism
$$
 \mathcal{S}(F,V)/\langle f-\rho(x)f\::\:x\in F,\:f\in\mathcal{S}(F,V)\rangle\longrightarrow\mathcal{S}(F,V)/\langle\psi(x)f-\rho(x)f\::\:x\in F,\:f\in\mathcal{S}(F,V)\rangle
$$
induced by the map $f\mapsto f'$, where $f'(x)=\psi(x)f(x)$. Hence ii) follows from i). Finally, iii) also follows from i), since the Fourier transform $f\mapsto\hat f$, where
$$
 \hat f(y)=\int\limits_F \psi(-uy)f(u)\,du,
$$
intertwines the actions $\rho$ and $\rho'$ of $F$ on $\mathcal{S}(F,V)$.
\end{proof}

\begin{lemma}\label{inducedreslemma2}
Let $G$ be an $l$-group, as in \cite{BeZe1976}, and let $H_1$ and $H_2$ be closed subgroups of $G$. Assume that $G=H_1H_2$, and that for every compact subset $K$ of $G$, there exists a compact subset $K_2$ of $H_2$ such that $K\subset H_1K_2$. Let $(\rho,V)$ be a smooth representation of $H_1$. The map $r:\cInd^G_{H_1} \rho \to \cInd^{H_2}_{H_1\cap H_2}(\rho|_{H_1\cap H_2})$ defined by restriction of functions is a well-defined isomorphism of representations of $H_2$.
\end{lemma}
\begin{proof}
This follows from straightforward verifications.
\end{proof}

\subsection{Siegel induced representations}\label{siegelindsec}
Let $\pi$ be an admissible representation of $\GL(2,F)$, let $\sigma$ be a character
of $F^\times$, and let $\pi\rtimes\sigma$ be as defined in Sect.\ \ref{representationssec}; see \eqref{Prepeq}. In this section we will calculate the twisted Jacquet modules $(\pi\rtimes\sigma)_{N,\theta}$ for any non-degenerate character $\theta$ of $N$ as a module of $T$.
Lemma \ref{siegelinducedbesselwaldspurgerlemma} below corrects an inaccuracy in Proposition 2.1 of \cite{PrTa2011}. Namely,  Proposition 2.1 of \cite{PrTa2011} does not include ii) of our lemma.

\begin{lemma}\label{casselmanfiltrationlemma2}
 Let $\sigma$ be a character of $F^\times$, and $\pi$ an admissible representation of $\GL(2,F)$. Let $I$ be the standard space of the Siegel induced representation $\pi\rtimes\sigma$. There is a filtration of $P$-modules
 $$
  I^3=0\subset I^2 \subset I^1 \subset I^0=I.
 $$
 with the quotients given as follows.
 \begin{enumerate}
  \item $ I^0/I^1 = \sigma_0$, where
   $$
    \sigma_0(\mat{A}{*}{}{cA'})=\sigma(c)\,|c^{-1}\det(A)|^{3/2}\,\pi(A)
   $$
   for $A$ in $\GL(2,F)$ and $c$ in $F^\times$.
 \item $I^1/I^2 = \cInd_{ \left[ \begin{smallmatrix} *&*&*&*\\  &*&&*\\ &&*&*\\ &&&* \end{smallmatrix} \right]  }^P \sigma_1$, where
   $$
    \sigma_1 (\begin{bmatrix} t&*&y&* \\&a&&*\\&&d&*\\&&& ad t^{-1} \end{bmatrix}) = \sigma(ad)\,|a^{-1}t|^{3/2}\,\pi(\mat{t}{y}{}{d})
   $$
   for $y$ in $F$ and $a,d,t$ in $F^\times$.
 \item $I^2/I^3 = \cInd_{ \left[ \begin{smallmatrix} *&*\\  *&*\\ &&*&*\\ &&*&* \end{smallmatrix} \right]  }^P \sigma_2$, where
   $$
    \sigma_2(\mat{A}{}{}{cA'})=\sigma(c)\,|c\det(A)^{-1}|^{3/2}\,\pi(cA')
   $$
   for $A$ in $\GL(2,F)$ and $c$ in $F^\times$.
 \end{enumerate}
\end{lemma}
\begin{proof}
This follows by going through the procedure of Sections 6.2 and 6.3 of \cite{C}. 
\end{proof}

\begin{lemma}\label{siegelinducedbesselwaldspurgerlemma}
 Let $\sigma$ be a character of $F^\times$, and let $(\pi,V)$ be an admissible representation of $\GL(2,F)$. We assume that $\pi$ admits a central character $\omega_\pi$. Let $I$ be the standard space of the Siegel induced representation $\pi\rtimes\sigma$. 
Let $\theta$ be the character of $N$ defined in \eqref{thetaSsetupeq}. Assume that $\theta$ is non-degenerate.  Let $L$ be the  quadratic extension associated to $S$ as in Sect. \ref{anotheralgebrasubsec}.
 \begin{enumerate}
  \item Assume that $L$ is a field. Then $I_{N,\theta}\cong V$ with the action of $T$ given by $\sigma\omega_\pi\pi'$. Here, $\pi'$ is the representation of $\GL(2,F)$ on $V$ given by $\pi'(g)=\pi(g')$. In particular, if $\pi$ is irreducible, then the action of $T$ is given by $\sigma\pi$.
  \item Assume that $L$ is not a field; we may arrange that $S=\left[\begin{smallmatrix} &1/2 \\ 1/2&\end{smallmatrix}\right]$. Then there is a filtration
   $$
    0\subset J_2\subset J_1=I_{N,\theta},
   $$
   with vector space isomorphisms:
   \begin{itemize}
    \item $J_1/J_2\cong V_{\left[\begin{smallmatrix} 1&*\\&1 \end{smallmatrix} \right],\psi} \oplus V_{\left[\begin{smallmatrix} 1&*\\&1 \end{smallmatrix} \right],\psi},$
    \item $J_2\cong V$.
   \end{itemize}
 The action of $T=T_S$ is given as follows,
 \begin{align*}
  {\rm diag}(a,b,a,b)(v_1 \oplus v_2) &=  \Big|\dfrac{a}{b}\Big|^{1/2}\sigma(ab)\omega_\pi(a)v_1
\oplus \Big|\dfrac{a}{b}\Big|^{-1/2}\sigma(ab)\omega_\pi(b)v_2, \\
  {\rm diag}(a,b,a,b)v&=\sigma(ab)\pi(\mat{a}{}{}{b})v,
 \end{align*}
 for  $a, b \in F^\times$, $v_1 \oplus v_2 \in J_1/J_2\cong V_{\left[\begin{smallmatrix} 1&*\\&1 \end{smallmatrix} \right],\psi} \oplus V_{\left[\begin{smallmatrix} 1&*\\&1 \end{smallmatrix} \right],\psi}$, and $v \in J_2$. In particular, if $\pi$ is one-dimensional, then $I_{N,\theta}\cong V$, with the action of $T$ given by ${\rm diag}(a,b,a,b)v=\sigma(ab)\pi(\mat{a}{}{}{b})v$.
 \end{enumerate}
\end{lemma}
\begin{proof}
We may assume that $b=0$. Since $\det (S ) \neq 0$ we have $a \neq 0$ and $c \neq 0$. 
We use the notation of Lemma \ref{casselmanfiltrationlemma2}. 
We calculate the twisted Jacquet modules $(I^i/I^{i+1})_{N,\theta}$ for $i \in \{0,1,2\}$. Since the action of $N$ on $I^0/I^1$ is trivial and $\theta$ is non-trivial, we have $(I^0/I^1)_{N,\theta}=0$. 

We consider the quotient $I^1/I^2 = \cInd_{H}^P \sigma_1$, where
$$
 H=\begin{bmatrix} *&*&*&*\\  &*&&*\\ &&*&*\\ &&&* \end{bmatrix},
$$
and with $\sigma_1$ as in ii) of Lemma \ref{casselmanfiltrationlemma2}. We first show that for each function $f$ in the standard model of this representation, the function $f^\circ:\:F\to V$, given by
$$
 f^\circ(w)=f(\begin{bmatrix}1&\\&1&w\\&&1&\\&&&1\end{bmatrix}),
$$
has compact support. Let $K$ be a compact subset of $P$ such that the support of $f$ is contained in $HK$. If
$$
 \begin{bmatrix}1&\\&1&w\\&&1&\\&&&1\end{bmatrix}=\begin{bmatrix} t&*&y&* \\&a&&*\\&&d&*\\&&& ad t^{-1} \end{bmatrix}\begin{bmatrix}k_1&k_2&x_1&x_2\\k_3&k_4&x_3&x_4\\&&k_5&k_6\\&&k_7&k_8\end{bmatrix},
$$
with the rightmost matrix being in $K$, then calculations show that $k_3=k_7=0$ and $w=k_4^{-1}x_3$. Since $k_4^{-1}$ and $x_3$ vary in bounded subsets, $w$ is confined to a compact subset of $F$. This proves our assertion that $f^\circ$ has compact support.

Next, for each function $f$ in the standard model of $\cInd_{H}^P \sigma_1$, consider the function $\tilde f:\:F^2\to V$ given by
$$
 \tilde f(u,w)=f(\begin{bmatrix}1&\\u&1&w\\&&1&\\&&-u&1\end{bmatrix})
$$
for $u,w$ in $F$. Let $W$ be the space of all such functions $\tilde f$. Since the map $f\mapsto\tilde f$ is injective, we get a vector space isomorphism $\cInd_{H}^P \sigma_1\cong W$. In this new model, the action of $N$ is given by
\begin{equation}\label{siegelinducedbesselwaldspurgerlemma20}
 (\begin{bmatrix}1&&y&z\\&1&x&y\\&&1\\&&&1\end{bmatrix}\tilde f)(u,w)=\pi(\mat{1}{y+uz}{}{1})\tilde f(u,w+x+2uy+u^2z)
\end{equation}
for $x,y,z,u,w$ in $F$.

We claim that $W$ contains $\mathcal{S}(F^2,V)$. Since $W$ is translation invariant, it is enough to prove that $W$ contains the function
$$
 f_{N,v}(u,w)=
  \begin{cases}
   v & \text{if $u,w \in \p^N$,}\\
   0 & \text{if $u \notin \p^N$ or $w \notin \p^N$,}
  \end{cases}
$$
for any $v$ in $V$ and any positive integer $N$. Again by translation invariance, we may assume that $N$ is large enough so that
\begin{equation}\label{siegelinducedbesselwaldspurgerlemma21}
 \sigma_1(h)v=v\qquad\text{for }h\in H\cap\Gamma_N,
\end{equation}
where
\begin{equation}\label{GammaNdefeq}
\Gamma_{N}=
\begin{bmatrix} 
1+\p^N&\p^N&\p^N&\p^N \\
\p^N&1+\p^N&\p^N&\p^N\\
&&1+\p^N&\p^N\\&&\p^N&1+\p^N
\end{bmatrix} \cap P.
\end{equation}
Define $f:\:P\to V$ by
$$
 f(g)=
  \begin{cases}
   \sigma_1(h)v&\text{if $g=hk$ with }h\in H,\;k\in\Gamma_N,\\
   0&g\notin H\Gamma_N.
  \end{cases}
$$
Then, by \eqref{siegelinducedbesselwaldspurgerlemma21}, $f$ is a well-defined element of $\cInd_{H}^Q \sigma_1$. It is easy to verify that $\tilde f=f_{N,v}$. This proves our claim that $W$ contains $\mathcal{S}(F^2,V)$.

Now consider the map
\begin{equation}\label{siegelinducedbesselwaldspurgerlemma22}
 W\longrightarrow\mathcal{S}(F,V),\qquad \tilde f\longmapsto\Big(w\mapsto f(\begin{bmatrix}1\\&1&w\\&&1\\&&&1\end{bmatrix}s_1)\Big),
\end{equation}
where $s_1$ is defined in \eqref{s1s2defeq}. This map is well-defined, since the function on the right is $(s_1f)^\circ$, which we showed above has compact support. Similar considerations as above show that the map \eqref{siegelinducedbesselwaldspurgerlemma22} is surjective.

We claim that the kernel of \eqref{siegelinducedbesselwaldspurgerlemma22} is $\mathcal{S}(F^2,V)$. First suppose that $\tilde f$ lies in the kernel; we have to show that $\tilde f$ has compact support. Choose $N$ large enough so that $f$ is right invariant under $\Gamma_N$. Then, for $u$ not in $\p^{-N}$ and $w$ in $F$,
\begin{align*}
 \tilde f(u,w)&=f(\begin{bmatrix}1\\u&1&w\\&&1\\&&-u&1\end{bmatrix})\\
 &=f(\begin{bmatrix}1&\\&1&w\\&&1\\&&&1\end{bmatrix}\!\begin{bmatrix}1&u^{-1}\\&1\\&&1&-u^{-1}\\&&&1\end{bmatrix}\!\begin{bmatrix}-u^{-1}\\&u\\&&u^{-1}\\&&&\!\!\!-u\end{bmatrix}\!s_1\!\begin{bmatrix}1&u^{-1}\\&1\\&&1&-u^{-1}\\&&&1\end{bmatrix})\\
 &=f(\begin{bmatrix}1&u^{-1}\\&1\\&&1&-u^{-1}\\&&&1\end{bmatrix}\begin{bmatrix}1&&-u^{-1}w&u^{-2}w\\&1&w&-u^{-1}w\\&&1\\&&&1\end{bmatrix}\begin{bmatrix}-u^{-1}\\&u\\&&u^{-1}\\&&&\!\!\!-u\end{bmatrix}s_1)\\
 &=\pi(\mat{1}{-u^{-1}w}{}{1})f(\begin{bmatrix}1&&\\&1&w&\\&&1\\&&&1\end{bmatrix}\begin{bmatrix}-u^{-1}\\&u\\&&u^{-1}\\&&&\!\!\!-u\end{bmatrix}s_1)\\ &=\pi(\mat{1}{w^2u^{-1}}{}{1})f(\begin{bmatrix}1&-w\\&1\\&&1&w\\&&&1\end{bmatrix}\begin{bmatrix}1\\&-u^{-1}\\&&-u\\&&&1\end{bmatrix}s_2)\\
 &=\pi(\mat{1}{-u^{-1}w}{}{1})f(\begin{bmatrix}-u^{-1}\\&u\\&&u^{-1}\\&&&\!\!\!-u\end{bmatrix}\begin{bmatrix}1&&\\&1&u^{-2}w&\\&&1\\&&&1\end{bmatrix}s_1)\\
 &=|u|^{-3}\pi(\mat{-u^{-1}}{-u^{-2}w}{}{u^{-1}})f(\begin{bmatrix}1&&\\&1&u^{-2}w&\\&&1\\&&&1\end{bmatrix}s_1).
\end{align*}
This last expression is zero by assumption. For fixed $u$ in $\p^{-N}$, the function $\tilde f(u,\cdot)$ has compact support; this follows because each $f^\circ$ has compact support. Combining these facts shows that $\tilde f$ has compact support. Conversely, assume $\tilde f$ is in $\mathcal{S}(F^2,V)$. Then we can find a large enough $N$ such that if $u$ has valuation $-N$, the function $\tilde f(u,\cdot)$ is zero. Looking at the above calculation, we see that, for fixed such $u$,
$$
 f(\begin{bmatrix}1&&\\&1&u^{-2}w&\\&&1\\&&&1\end{bmatrix}s_1)=0
$$
for all $w$ in $F$. This shows that $\tilde f$ is in the kernel of the map \eqref{siegelinducedbesselwaldspurgerlemma22}, completing the proof of our claim about this kernel. We now have an exact sequence
\begin{equation}\label{siegelinducedbesselwaldspurgerlemma23}
 0\longrightarrow\mathcal{S}(F^2,V)\longrightarrow W\longrightarrow\mathcal{S}(F,V)\longrightarrow0.
\end{equation}
Note that the space $\mathcal{S}(F^2,V)$ is invariant under the action \eqref{siegelinducedbesselwaldspurgerlemma20} of $N$. A calculation shows that the action of $N$ on $\mathcal{S}(F,V)$ is given by
\begin{equation}\label{siegelinducedbesselwaldspurgerlemma24}
 (\begin{bmatrix}1&&y&z\\&1&x&y\\&&1\\&&&1\end{bmatrix}f)(w)=\pi(\mat{1}{y}{}{1})f(w+z)
\end{equation}
for $x,y,z,w$ in $F$ and $f$ in $\mathcal{S}(F,V)$. Since the action of $x$ is trivial and $a\neq0$, it follows that $\mathcal{S}(F,V)_{N,\theta}=0$. Hence, by \eqref{siegelinducedbesselwaldspurgerlemma23}, we have $W_{N,\theta}\cong\mathcal{S}(F^2,V)_{N,\theta}$.

We will compute the Jacquet module $\mathcal{S}(F^2,V)_{N,\theta}$ in stages. The action of $N$ on $\mathcal{S}(F^2,V)$ is given by \eqref{siegelinducedbesselwaldspurgerlemma20}. By ii) of Lemma \ref{basicFjacquetlemma}, the map $\tilde f\mapsto f'$, where $f':\:F\to V$ is given by
$$
 f'(u)=\int\limits_F\psi^a(-u)\tilde f(u,w)\,dw,
$$
defines a vector space isomorphism
$$
 W_{\left[ \begin{smallmatrix} 1&&& \\ &1&*& \\ &&1& \\ &&&1 \end{smallmatrix} \right], \psi^a} \stackrel{\sim}{\longrightarrow} \mathcal{S}(F,V). 
$$
The transfer of the action of the remaining group $\left[\begin{smallmatrix} 1&&*&* \\ &1&&* \\ &&1& \\ &&&1 \end{smallmatrix} \right]$ to
$\mathcal{S}(F,V)$ is given by 
$$
(\begin{bmatrix} 1&&y&z \\ &1&& y \\ &&1& \\ &&&1 \end{bmatrix} f)(u) = \psi (a (2uy+u^2 z)) \pi (\mat{1}{y+uz}{}{1}) f(u)
$$
for $u,y,z \in F$ and $f \in \mathcal{S}(F,V)$.
The subspace $\mathcal{S}(F^\times,V)$ of elements $f$ of $\mathcal{S}(F,V)$ such that $f(0)=0$ is preserved under this action, so that we have an exact sequence 
\begin{equation}
\label{siegelinducedbesselwaldspurgerlemmaeqq2}
0 \longrightarrow \mathcal{S}(F^\times,V) \longrightarrow \mathcal{S}(F,V) \longrightarrow \mathcal{S}(F,V)/ \mathcal{S}(F^\times,V) \longrightarrow 0
\end{equation}
of representations of the group $\left[ \begin{smallmatrix} 1&&*&* \\ &1&&* \\ &&1& \\ &&&1 \end{smallmatrix} \right]$. There is an isomorphism
$$
\mathcal{S}(F,V)/ \mathcal{S}(F^\times,V) \stackrel{\sim}{\longrightarrow} V
$$
of vector spaces that sends $f$ to $f(0)$. The transfer of the action of the group $\left[ \begin{smallmatrix} 1&&*&* \\ &1&&* \\ &&1& \\ &&&1 \end{smallmatrix} \right]$
to $V$ is given by
\begin{equation}
\label{siegelinducedbesselwaldspurgerlemmaeqq1}
\begin{bmatrix} 1&&y&z \\ &1&& y \\ &&1& \\ &&&1 \end{bmatrix} v=  \pi (\mat{1}{y}{}{1}) v
\end{equation}
for $y, z \in F$ and $v \in V$.
The non-vanishing of $c$ and  \eqref{siegelinducedbesselwaldspurgerlemmaeqq1} imply that 
$$
(\mathcal{S}(F,V)/ \mathcal{S}(F^\times,V))_{\left[ \begin{smallmatrix} 1&&&* \\ &1&& \\ &&1& \\ &&&1 \end{smallmatrix} \right],\psi^c} =0.
$$
Therefore, 
$$
\big( \mathcal{S}(F,V)/ \mathcal{S}(F^\times,V) \big)_{\left[ \begin{smallmatrix} 1&&*&* \\ &1&&* \\ &&1& \\ &&&1 \end{smallmatrix} \right], \theta} =0. 
$$
Next, we define a vector space isomorphism of $\mathcal{S}(F^\times ,V)$
with itself and then transfer the action. For $f$ in $\mathcal{S}(F^\times ,V)$, set
$$
\tilde f(u) = \pi (\mat{u}{}{}{1})f(u)
$$
for $u$ in $F^\times$. The map defined by $f \mapsto \tilde f$ is an automorphism of vector spaces. The transfer of the action of 
$\left[ \begin{smallmatrix} 1&&*&* \\ &1&&* \\ &&1& \\ &&&1 \end{smallmatrix} \right]$ is given by
$$
(\begin{bmatrix} 1&&y&z \\ &1&&y \\ &&1& \\ &&&1 \end{bmatrix} f )(u)
=\psi(a (2uy+u^2z)) \pi (\mat{1}{uy+u^2 z}{}{1}) f(u)
$$
for $f \in \mathcal{S}(F^\times, V)$, $y,z \in F$, and $u \in F^\times$. 
Now define a linear map
$$
 p:\:\mathcal{S}(F^\times, V) \longrightarrow \mathcal{S}(F^\times, V_{\left[\begin{smallmatrix} 1&*\\&1 \end{smallmatrix} \right],\psi^{-2a}})
$$
by composing the elements of $\mathcal{S}(F^\times, V)$ with the natural projection from $V$ to 
$V_{\left[\begin{smallmatrix} 1&*\\&1 \end{smallmatrix} \right],\psi^{-2a}}=V/V(\left[\begin{smallmatrix} 1&*\\&1 \end{smallmatrix} \right],\psi^{-2a})$. 
The map $p$ is surjective. 
Let $f$ be in $\mathcal{S}(F^\times, V)$. Since $f$ has compact support and is locally constant, we see that $f$ is in the kernel of $p$ if and only if 
\begin{equation}
\label{siegelinducedbesselwaldspurgerlemmaeqq4}
\text{there exists $l>0$ such that } \int\limits_{\p^{-l}} \psi (2ay) \pi( \mat{1}{y}{}{1})f(u)\, dy =0 \qquad\text{for all }u \in F^\times. 
\end{equation}
Also, $f$ is in $\mathcal{S}(F^\times, V)(\left[ \begin{smallmatrix} 1&&*& \\ &1&&* \\ &&1& \\ &&&1 \end{smallmatrix} \right])$ if and only if
\begin{equation}
\label{siegelinducedbesselwaldspurgerlemmaeqq5}
\text{there exists $k>0$ such that } \int\limits_{\p^{-k}} \psi (2auy) \pi( \mat{1}{uy}{}{1})f(u)\, dy =0 \qquad\text{for all }u \in F^\times. 
\end{equation}
Since $f$ is locally constant and compactly supported the conditions \eqref{siegelinducedbesselwaldspurgerlemmaeqq4} and \eqref{siegelinducedbesselwaldspurgerlemmaeqq5}
are equivalent. It follows that $p$ induces an isomorphism of vector spaces:
$$
\mathcal{S}(F^\times, V)_{\left[ \begin{smallmatrix} 1&&*& \\ &1&&* \\ &&1& \\ &&&1 \end{smallmatrix} \right],\psi^b} =
\mathcal{S}(F^\times, V)_{\left[ \begin{smallmatrix} 1&&*& \\ &1&&* \\ &&1& \\ &&&1 \end{smallmatrix} \right]}
\stackrel{\sim}{ \longrightarrow }\mathcal{S}(F^\times, V_{\left[\begin{smallmatrix} 1&*\\&1 \end{smallmatrix} \right],\psi^{-2a}}).
$$
Transferring the action of $\left[ \begin{smallmatrix} 1&&&* \\ &1&& \\ &&1& \\ &&&1 \end{smallmatrix} \right]$ on the first space to the last space results
in the formula
$$
\begin{bmatrix} 1&&&z \\ &1&& \\ &&1& \\ &&&1 \end{bmatrix} f (u) = \psi(a u^2 z) \pi (\mat{1}{u^2z}{}{1}) f(u)=\psi (-au^2 z) f(u), \qquad z \in F,\ u \in F^\times,
$$
for $f$ in $\mathcal{S}(F^\times, V_{\left[\begin{smallmatrix} 1&*\\&1 \end{smallmatrix} \right],\psi^{-2a}})$.

Assume that $L$ is a field; we will prove that 
\begin{equation}
\label{siegelinducedbesselwaldspurgerlemmaeqq6}
\mathcal{S}(F^\times, V_{\left[\begin{smallmatrix} 1&*\\&1 \end{smallmatrix} \right],\psi^{-2a}})_{\left[ \begin{smallmatrix} 1&&&* \\ &1&& \\ &&1& \\ &&&1 \end{smallmatrix} \right], \psi^c} =0.
\end{equation}
Let $f$ be in $\mathcal{S}(F^\times, V_{\left[\begin{smallmatrix} 1&*\\&1 \end{smallmatrix} \right],\psi^{-2a}})$. Since the support of $f$ is compact,
and since there exists no $u$ in $F^\times$ such that $c+au^2 =0$ as $D=b^2/4-ac=-ac$ is not in $F^{\times2}$, there exists a positive integer $l$ such that 
\begin{equation}
\label{addnosuppeq}
\int\limits_{\p^{-l}} \psi(-(c+au^2) z)\,dz = 0
\end{equation}
for $u$ in the support of $f$. 
Hence, for $u$ in $F^\times$, 
\begin{equation}
\label{JLcondeq}
(\int\limits_{\p^{-l}} \psi(-cz) \begin{bmatrix} 1&&&z \\ &1&& \\ &&1& \\ &&&1 \end{bmatrix} f\, dz)(u)
= \big( \int\limits_{\p^{-l}} \psi(-(c+au^2) z)\,dz\big) f(u)=0. 
\end{equation}
This proves \eqref{siegelinducedbesselwaldspurgerlemmaeqq6}, and completes the argument that $(I^1/I^2)_{N,\theta}=0$ in the case $L$ is a field. 

Now assume that $L$ is not a field. We may further assume that $a=1$ and $c=-1$ while retaining $b=0$. The group $T=T_{\left[\begin{smallmatrix} a&b/2\\b/2&c \end{smallmatrix} \right]}=T_{\left[\begin{smallmatrix} 1&\\&-1 \end{smallmatrix} \right]}$ consists of the elements
\begin{equation}
\label{talteq}
t=\begin{bmatrix} x&y && \\ y&x && \\ && x&-y \\ && -y &x \end{bmatrix} 
\end{equation}
with $x,y \in F$ such that $x^2 \neq y^2$. 
Define
\begin{equation}
\label{lastisoeq}
\mathcal{S}(F^\times, V_{\left[\begin{smallmatrix} 1&*\\&1 \end{smallmatrix} \right],\psi^{-2a}}) \longrightarrow 
V_{\left[\begin{smallmatrix} 1&*\\&1 \end{smallmatrix} \right],\psi^{-2a}} \oplus V_{\left[\begin{smallmatrix} 1&*\\&1 \end{smallmatrix} \right],\psi^{-2a}}
\end{equation}
by $f \mapsto f(1) \oplus f(-1)$. We assert that the kernel of this linear map is 
$$
\mathcal{S}(F^\times, V_{\left[\begin{smallmatrix} 1&*\\&1 \end{smallmatrix} \right],\psi^{-2a}})(\left[ \begin{smallmatrix} 1&&&* \\ &1&& \\ &&1& \\ &&&1 \end{smallmatrix} \right], \psi^c).
$$
Evidently, this subspace is contained in the kernel. Conversely, let $f \in \mathcal{S}(F^\times, V_{\left[\begin{smallmatrix} 1&*\\&1 \end{smallmatrix} \right],\psi^{-2a}}) $ be such that $f(1)=f(-1)=0$. Then there exists a positive integer $l$ such that \eqref{addnosuppeq} holds for $u$ in the support of $f$, implying that \eqref{JLcondeq} holds. This proves our assertion. The map \eqref{lastisoeq} is clearly surjective, so that we obtain an isomorphism
$$
\mathcal{S}(F^\times, V_{\left[\begin{smallmatrix} 1&*\\&1 \end{smallmatrix} \right],\psi^{-2a}})_{\left[ \begin{smallmatrix} 1&&&* \\ &1&& \\ &&1& \\ &&&1 \end{smallmatrix} \right], \psi^c} \stackrel{\sim}{\longrightarrow} V_{\left[\begin{smallmatrix} 1&*\\&1 \end{smallmatrix} \right],\psi^{-2a}} \oplus V_{\left[\begin{smallmatrix} 1&*\\&1 \end{smallmatrix} \right],\psi^{-2a}}. 
$$
We now have an isomorphism $(I^1/I^2)_{N,\theta} \stackrel{\sim}{\longrightarrow} V_{\left[\begin{smallmatrix} 1&*\\&1 \end{smallmatrix} \right],\psi^{-2a}} \oplus V_{\left[\begin{smallmatrix} 1&*\\&1 \end{smallmatrix} \right],\psi^{-2a}}$. A calculation shows that the transfer of  the action of $T$ to $V_{\left[\begin{smallmatrix} 1&*\\&1 \end{smallmatrix} \right],\psi^{-2a}} \oplus V_{\left[\begin{smallmatrix} 1&*\\&1 \end{smallmatrix} \right],\psi^{-2a}}$ is given by
$$
t(v_1 \oplus v_2) =  \Big|\dfrac{x-y}{x+y}\Big|^{1/2}\sigma\big((x-y)(x+y)\big)\omega_\pi(x-y)v_1
\oplus \Big|\dfrac{x-y}{x+y}\Big|^{-1/2}\sigma\big((x-y)(x+y)\big)\omega_\pi(x+y)v_2
$$
for $t$ as in \eqref{talteq} and $v_1,v_2 \in V_{\left[\begin{smallmatrix} 1&*\\&1 \end{smallmatrix} \right],\psi^{-2a}}$. 
Finally, the result stated in ii) is written with respect to $S=\left[\begin{smallmatrix} &1/2\\1/2&\end{smallmatrix}\right]$. To change to this choice note that the map
$$
C:(I^1/I^2)_{N,\theta_{\left[\begin{smallmatrix} a&b/2\\b/2&c\end{smallmatrix}\right]}} \longrightarrow 
(I^1/I^2)_{N,\theta_{\left[\begin{smallmatrix} &1/2\\1/2&\end{smallmatrix}\right]}}
$$
defined by $v \mapsto \left[\begin{smallmatrix} g& \\ &g'\end{smallmatrix} \right] v$, where $g=\left[\begin{smallmatrix} -1&1\\1&1\end{smallmatrix}\right]$, is a well-defined isomorphism; recall that $a=1,b=0,c=-1$. Moreover,  $C(tv)=t'C(v)$ for $t$ as in \eqref{talteq} and 
$$
t'
=
\begin{bmatrix} x-y & \\ & x+y \\ && x-y \\ &&&x+y \end{bmatrix} \in T_{\left[\begin{smallmatrix} &1/2\\1/2 & \end{smallmatrix}\right]}. 
$$
It follows that the group $T_{\left[\begin{smallmatrix} &1/2\\1/2 & \end{smallmatrix}\right]}$ acts on the isomorphic vector spaces
$$
(I^1/I^2)_{N,\theta_{\left[\begin{smallmatrix} &1/2\\1/2&\end{smallmatrix}\right]}} 
\cong 
V_{\left[\begin{smallmatrix} 1&*\\&1 \end{smallmatrix} \right],\psi^{-2a}} 
\oplus 
V_{\left[\begin{smallmatrix} 1&*\\&1 \end{smallmatrix} \right],\psi^{-2a}}
\cong 
V_{\left[\begin{smallmatrix} 1&*\\&1 \end{smallmatrix} \right],\psi} 
\oplus 
V_{\left[\begin{smallmatrix} 1&*\\&1 \end{smallmatrix} \right],\psi}
$$
via the formula
in ii).

Next, we consider the quotient $I^2/I^3=\cInd_M^P(\sigma_2)$ from iii) of Lemma \ref{casselmanfiltrationlemma2}. By Lemma \ref{inducedreslemma2},  restriction of functions in the standard model of this representation to $N$ gives an $N$-isomorphism
$$
 \cInd_M^P(\sigma_2)\cong\mathcal{S}(N,V).
$$
An application of i) and ii) of Lemma \ref{basicFjacquetlemma} shows that
$
 \mathcal{S}(N,V)_{N,\theta}\cong V
$
via the map defined by 
$$
 f\longmapsto\int\limits_N\theta(n)^{-1}f(n)\,dn.
$$
Transferring the action of $T$ we find that $t\in T$ acts by $\sigma_2(t)$ on $V$. If $t=\mat{g}{}{}{\det(g)g'}$ as in \eqref{Tembeddingeq}, then
$$
 \sigma_2(t)=\sigma(\det(g))\omega_\pi(\det(g))\pi(g').
$$
This concludes the proof.
\end{proof}

In case of a one-dimensional representation of $M$, it follows from this lemma that
\begin{equation}\label{siegelinducedonedimjacqueteq}
 (\chi\triv_{\GL(2)}\rtimes\sigma)_{N,\theta}=\C_{(\sigma\chi)\circ \Norm_{L/F}}
\end{equation}
as $T$-modules. In case $L$ is a field and $\pi$ is irreducible, it follows from Lemma \ref{siegelinducedbesselwaldspurgerlemma} that
\begin{equation}\label{siegelinducedfieldeq}
 \Hom_T((\pi\rtimes\sigma)_{N,\theta},\C_\Lambda)=\Hom_T(\sigma\pi,\C_\Lambda).
\end{equation}
Hence, in view of \eqref{DTJacqueteq}, the space of $(\Lambda,\theta)$-Bessel functionals on $\pi\rtimes\sigma$ is isomorphic to the space of $(\Lambda,\theta)$-Waldspurger functionals on $\sigma\pi$.

\subsection{Klingen induced representations}\label{Klingendegsec}
Let $\pi$ be an admissible representation of $\GL(2,F)$, let $\chi$ be a character
of $F^\times$, and let $\chi\rtimes\pi$ be as defined in Sect.\ \ref{representationssec}; see \eqref{Qrepeq}. In this section we will calculate the twisted Jacquet modules $(\chi\rtimes\pi)_{N,\theta}$ for any non-degenerate character $\theta$ of $N$ as a module of $T$. In the split case our results make several corrections to Proposition 2.3 and Proposition 2.4 of \cite{PrTa2011}.

\begin{lemma}\label{casselmanfiltrationlemma1}
 Let $\chi$ be a character of $F^\times$ and $\pi$ an admissible representation of $\GL(2,F)$. Let $I$ be the space of the Klingen induced representation $\chi\rtimes\pi$. There is a filtration of $P$-modules
 $$
  I^2=0\subset I^1 \subset I^0=I.
 $$
 with the quotients given as follows.
 \begin{enumerate}
  \item $ I^0/I^1 = \cInd_B^P\sigma_0$, where
   $$
    \sigma_0(\begin{bmatrix}t&*&*&*\\&a&b&*\\&&d&*\\&&&adt^{-1}\end{bmatrix})=\chi(t)\,|t|^2\,|ad|^{-1}\,\pi(\mat{a}{b}{}{d})
   $$
   for $b$ in $F$ and $a,d,t$ in $F^\times$.
 \item $I^1/I^2 = \cInd_{ \left[ \begin{smallmatrix} *&*&&*\\  &*\\ &&*&*\\ &&&* \end{smallmatrix} \right]  }^P \sigma_1$, where
   $$
    \sigma_1 (\begin{bmatrix} t&*&&x \\&a\\&&d&*\\&&& ad t^{-1} \end{bmatrix}) = \chi(d)\,|a^{-1}d|\,\pi(\mat{t}{x}{}{ad t^{-1}})
   $$
   for $x$ in $F$ and $a,d,t$ in $F^\times$.
 \end{enumerate}
\end{lemma}
\begin{proof}
This follows by going through the procedure of Sections 6.2 and 6.3 of \cite{C}. 
\end{proof}

\begin{lemma}\label{Klingendegjacquetlemma1}
 Let $\chi$ be a character of $F^\times$, and let $(\pi,V)$ be an admissible representation of $\GL(2,F)$. We assume that $\pi$ has a central character $\omega_\pi$. Let $I$ be the standard space of the Klingen induced representation $\chi\rtimes\pi$. Let $N$ be the unipotent radical of the Siegel parabolic subgroup, and let $\theta$ be the character of $N$ defined in (\ref{thetaSsetupeq}). We assume that the associated quadratic extension $L$ is a field. Then, as $T$-modules,
 $$
  I_{N,\theta}\cong\bigoplus_{\Lambda|_{F^\times}=\chi\omega_\pi}d\cdot\Lambda,\qquad\text{where }d=\dim\Hom_{ \left[ \begin{smallmatrix} 1&*\\  &1\end{smallmatrix} \right]  }(\pi,\psi).
 $$
 In particular, $I_{N,\theta}=0$ if $\pi$ is one-dimensional.
\end{lemma}
\begin{proof}
We will first prove that $(I_0/I_1)_{N,\theta}=0$, where the notations are as in Lemma \ref{casselmanfiltrationlemma1}. We may assume that the element $b$ appearing in the matrix $S$ in \eqref{thetaSsetupeq} is zero. For $f$ in the standard space of the induced representation $I_0/I_1=\cInd_B^P\sigma_0$, let
$$
 \tilde f(u)=f(\begin{bmatrix}1\\u&1\\&&1\\&&-u&1\end{bmatrix}),\qquad u\in F.
$$
Let $W$ be the space of all functions $F\to\C$ of the form $\tilde f$, where $f$ runs through $\cInd_B^P\sigma_0$. Since the map $f\mapsto\tilde f$ is injective, we obtain a vector space isomorphism $\cInd_B^P\sigma_0\cong W$. The identity
$$
 \begin{bmatrix}1\\u&1\\&&1\\&&-u&1\end{bmatrix}=\begin{bmatrix}1&u^{-1}\\&1\\&&1&-u^{-1}\\&&&1\end{bmatrix}\begin{bmatrix}-u^{-1}\\&u\\&&u^{-1}\\&&&-u\end{bmatrix}s_1\begin{bmatrix}1&u^{-1}\\&1\\&&1&-u^{-1}\\&&&1\end{bmatrix},
$$
where $s_1$ is as in \eqref{s1s2defeq}, shows that $\tilde f$ satisfies
\begin{equation}\label{Klingendegjacquetlemma1eq10}
 \tilde f(u)=\chi(-u^{-1})|u|^{-2}\pi(\mat{u}{}{}{u^{-1}})f(s_1)\qquad\text{for }|u|\gg0.
\end{equation}
The space $W$ consists of locally constant functions. Furthermore, $W$ is invariant under translations, i.e., if $f'\in W$, then the function $u\mapsto f'(u+x)$ is also in $W$, for any $x$ in $F$.

We claim that $W$ contains $\mathcal{S}(F,V)$. Since $W$ is translation invariant, it is enough to prove that $W$ contains the function
$$
 f_{N,v}(u)=
  \begin{cases}
   v & \text{if $u \in \p^N$,}\\
   0 & \text{if $u \notin \p^N$,}
  \end{cases}
$$
for any $v$ in $V$ and any positive integer $N$. Again by translation invariance, we may assume that $N$ is large enough so that
\begin{equation}\label{Klingendegjacquetlemma1eq11}
 \sigma_0(b)v=v\qquad\text{for }b\in B\cap\Gamma_N,
\end{equation}
where $\Gamma_N$ is as in \eqref{GammaNdefeq}. Define $f:\:P\to V$ by
$$
 f(g)=
  \begin{cases}
   \sigma_0(b)v&\text{if $g=bk$ with }b\in B,\;k\in\Gamma_N,\\
   0&g\notin B\Gamma_N.
  \end{cases}
$$
Then, by \eqref{Klingendegjacquetlemma1eq11}, $f$ is a well-defined element of $\cInd_B^P\sigma_0$. It is easy to verify that $\tilde f=f_{N,v}$. This proves our claim that $W$ contains $\mathcal{S}(F,V)$.

We define a linear map $W\to V$ by sending $\tilde f$ to the vector $f(s_1)$ in \eqref{Klingendegjacquetlemma1eq10}. Then the kernel of this map is $\mathcal{S}(F,V)$. We claim that the map is surjective. To see this, let $v$ be in $V$. Again choose $N$ large enough so that \eqref{Klingendegjacquetlemma1eq11} holds. Then the function $f:\:P\to V$ given by
$$
 f(g)=
  \begin{cases}
   \sigma_0(b)v&\text{if $g=bs_1k$ with }b\in B,\;k\in\Gamma_N,\\
   0&g\notin Bs_1\Gamma_N.
  \end{cases}
$$
is a well-defined element of $\cInd_B^P\sigma_0$ with $f(s_1)=v$. This proves our claim that the map $W\to V$ is surjective. We therefore get an exact sequence
\begin{equation}\label{Klingendegjacquetlemma1eq12}
 0\longrightarrow\mathcal{S}(F,V)\longrightarrow W\longrightarrow V\longrightarrow0.
\end{equation}
The transfer of the action of $N$ to $W$ is given by
$$
 (\begin{bmatrix}1&&y&z\\&1&x&y\\&&1\\&&&1\end{bmatrix}\tilde f)(u)=\pi(\mat{1}{x+2uy+u^2z}{}{1})\tilde f(u)
$$
for all $x,y,z,u$ in $F$. Evidently, the subspace $\mathcal{S}(F,V)$ is invariant under $N$. Moreover, the action of $N$ on $V$ is given by
\begin{equation}\label{Klingendegjacquetlemma1eq13}
 \begin{bmatrix}1&&y&z\\&1&x&y\\&&1\\&&&1\end{bmatrix}v=\pi(\mat{1}{z}{}{1})v
\end{equation}
for all $x,y,z$ in $F$ and $v$ in $V$.

To prove that $(I_0/I_1)_{N,\theta}=0$, it suffices to show that $\mathcal{S}(F,V)_{N,\theta}=0$ and $V_{N,\theta}=0$. Since the element $a$ in the matrix $S$ is non-zero, it follows from \eqref{Klingendegjacquetlemma1eq13} that $V_{N,\theta}=0$.

To prove that $\mathcal{S}(F,V)_{N,\theta}=0$, we define a map $p$ from $\mathcal{S}(F,V)$ to
$$
 \mathcal{S}(F,V_{\left[\begin{smallmatrix}1&*\\&1\end{smallmatrix}\right],\psi^a})=\mathcal{S}(F,V/V(\mat{1}{*}{}{1},\psi^a))
$$
by sending $f$ to $f$ composed with the natural projection from $V$ to $V/V(\mat{1}{*}{}{1},\psi^a)$. This map is surjective. It is easy to see that $p$ induces an isomorphism
$$
 \mathcal{S}(F,V)_{\left[ \begin{smallmatrix}1\\ &1&* \\ &&1& \\ &&&1 \end{smallmatrix} \right],\psi^a }\cong\mathcal{S}(F,V_{\left[\begin{smallmatrix}1&*\\&1\end{smallmatrix}\right],\psi^a}).
$$
For the space on the right we have the action
$$
 (\begin{bmatrix}1&&y&z\\&1&&y\\&&1\\&&&1\end{bmatrix}f)(u)=\pi(\mat{1}{\:2uy+u^2z}{}{1})f(u),\qquad u\in F.
$$
By iii) of Lemma \ref{basicFjacquetlemma}, the map $f\mapsto f(0)$ induces an isomorphism
$$
 \mathcal{S}(F,V_{\left[\begin{smallmatrix}1&*\\&1\end{smallmatrix}\right],\psi^a})_{\left[\begin{smallmatrix}1&&*\\&1&&*\\&&1\\&&&1\end{smallmatrix}\right]}\cong V_{\left[\begin{smallmatrix}1&*\\&1\end{smallmatrix}\right],\psi^a}.
$$
For the space on the right we have the action
$$
 \begin{bmatrix}1&&&z\\&1\\&&1\\&&&1\end{bmatrix}v=v.
$$
Taking a twisted Jacquet module with respect to the character $\psi^c$ gives zero, since $c\neq0$. This concludes our proof that $(I_0/I_1)_{N,\theta}=0$.

Next let $\sigma_1$ be as in ii) of Lemma \ref{casselmanfiltrationlemma1}. Let
$$
 H_1=\begin{bmatrix} *&*&&*\\  &*\\ &&*&*\\ &&&* \end{bmatrix}
$$
and $H_2=TN$. By Lemma \ref{TFGL2lemma}, we have $P=H_1H_2$. To verify the hypotheses of Lemma \ref{inducedreslemma2}, let $K$ be a compact subset of $P$. Write $P=MN$ and let $p:P\to N$ be the resulting projection map. Since $p$ is continuous, the set $p(K)$ is compact. There exists a compact subset $K_T$ of $T$ such that $T=F^\times K_T$. Then $M\subset H_1K_T$ by Lemma \ref{TFGL2lemma}. Therefore $K\subset H_1K_2$ with $K_2=K_Tp(K)$.

By Lemma \ref{inducedreslemma2}, restriction of functions gives a $TN$ isomorphism
$$
 \cInd_{ \left[ \begin{smallmatrix} *&*&&*\\  &*\\ &&*&*\\ &&&* \end{smallmatrix} \right]  }^P \sigma_1\cong\cInd_{F^\times Z^J}^{TN}(\sigma_1\big|_{F^\times Z^J}).
$$
Note that $F^\times$ acts via the character $\chi\omega_\pi$ on this module. Since $T$ is compact modulo $F^\times$, the Jacquet module $(\cInd_{F^\times Z^J}^{TN}(\sigma_1\big|_{F^\times Z^J}))_{N,\theta}$ is a direct sum over characters of $T$.
Let $\Lambda$ be a character of $T$. It is easy to verify that
$$
 \Hom_{T}\big((\cInd_{F^\times Z^J}^{TN}(\sigma_1\big|_{F^\times Z^J}))_{N,\theta},\Lambda\big)=\Hom_{TN}\big(\cInd_{F^\times Z^J}^{TN}(\sigma_1\big|_{F^\times Z^J}),\Lambda\otimes\theta\big).
$$
By Frobenius reciprocity, the space on the right is isomorphic to
\begin{equation}\label{Klingendegjacquetlemma1eq1}
 \Hom_{F^\times Z^J}\big(\sigma_1\big|_{F^\times Z^J},(\Lambda\otimes\theta)\big|_{F^\times Z^J}\big).
\end{equation}
This space is zero unless the restriction of $\Lambda$ to $F^\times$ equals $\chi\omega_\pi$. Assume this is the case. Then \eqref{Klingendegjacquetlemma1eq1} is equal to
$$
 \Hom_{ \left[ \begin{smallmatrix} 1&*\\  &1\end{smallmatrix} \right]  }(\pi,\psi^c)\cong\Hom_{ \left[ \begin{smallmatrix} 1&*\\  &1\end{smallmatrix} \right]  }(\pi,\psi).
$$
This concludes the proof.
\end{proof}

\begin{lemma}\label{casselmanfiltrationlemma}
 Let $\chi$ be a character of $F^\times$ and $\pi$ an admissible representation of $\GL(2,F)$. Let $I$ be the space of the Klingen induced representation $\chi\rtimes\pi$. There is a filtration of $Q$-modules
 $$
  I^3=0\subset I^2 \subset I^1 \subset I^0=I.
 $$
 with the quotients given as follows.
 \begin{enumerate}
  \item $ I^0/I^1 = \sigma_0$, where
   $$
    \sigma_0 (\begin{bmatrix} t & * & *& * \\ & a&b & * \\ &c & d & * \\ & & & (ad -bc )t^{-1} \end{bmatrix}) = \chi (t)\,|t^2 (ad-bc)^{-1}|\,\pi(\mat{a}{b}{c}{d})
   $$
   for $\left[\begin{smallmatrix}a&b\\c&d\end{smallmatrix}\right]$ in $\GL(2,F)$ and $t$ in $F^\times$.
  \item $I^1/I^2 = \cInd_{ \left[ \begin{smallmatrix} *&&*&*\\  &*&*&*\\ &&*&\\ &&&* \end{smallmatrix} \right]  }^Q \sigma_1$, where
   $$
    \sigma_1 (\begin{bmatrix} t&&*&x \\&a&b&*\\&&d&\\&&& ad t^{-1} \end{bmatrix}) = \chi (a)\,|a d^{-1} |\,\pi(\mat{t}{x}{}{ad t^{-1}})
   $$
   for $b,x$ in $F$ and $a,d,t$ in $F^\times$.
  \item $I^2/I^3=I^2= \cInd_{\left[ \begin{smallmatrix} *&&& \\ &*&*& \\ &*&*& \\ &&& * \end{smallmatrix} \right] } ^Q\sigma_2$, where
   $$
    \sigma_2 (\begin{bmatrix}t&&&\\&a&b&\\&c&d&\\&&& (ad-bc) t^{-1}\end{bmatrix}) = \chi(t^{-1}(ad-bc))\,|t^{-2}(ad-bc)|\,\pi(\mat{a}{b}{c}{d})
   $$
   for $\left[\begin{smallmatrix}a&b\\c&d\end{smallmatrix}\right]$ in $\GL(2,F)$ and $t$ in $F^\times$.
 \end{enumerate}
\end{lemma}
\begin{proof}
This follows by going through the procedure of Sections 6.2 and 6.3 of \cite{C}. 
\end{proof}

\begin{lemma}\label{Klingendegjacquetlemma}
 Let $\chi$ be a character of $F^\times$, and let $(\pi,V)$ be an admissible representation of $\GL(2,F)$. Let $I$ be the standard space of the Klingen induced representation $\chi\rtimes\pi$. Let $N$ be the unipotent radical of the Siegel parabolic subgroup, and let $\theta$ be the character of $N$ defined in (\ref{splitthetaeq}) (i.e., we consider the split case). Then there is a filtration
   $$
    0\subset J_3\subset J_2\subset J_1=I_{N,\theta},
   $$
   with the quotients given as follows.
   \begin{itemize}
    \item $J_1/J_2\cong V$
    \item $J_2/J_3\cong V_{ \left[ \begin{smallmatrix} 1\\  *&1\end{smallmatrix} \right]}$.
    \item $J_3\cong \mathcal{S}(F^\times,V_{\left[\begin{smallmatrix}1\\ *&1\end{smallmatrix}\right],\psi})$.
   \end{itemize}
 The action of the stabilizer of $\theta$ is given as follows,
 \begin{align*}
  {\rm diag}(a,b,a,b)v&=\chi(a)\pi(\mat{a}{}{}{b})v\qquad\text{for }v\in J_1/J_2,\\
  {\rm diag}(a,b,a,b)v&=\chi(b)\pi(\mat{a}{}{}{b})v\qquad\text{for }v\in J_2/J_3,\\
  ({\rm diag}(a,b,a,b)f)(u)&=\chi(b)\pi(\mat{a}{}{}{a})f(a^{-1}bu)\qquad\text{for }f\in J_3,\;u\in F^\times,
 \end{align*}
 for all $a$ and $b$ in $F^\times$.
 In particular, we have the following special cases.
 \begin{enumerate}
  \item Assume that $\pi=\sigma1_{\GL(2)}$. Then the twisted Jacquet module $I_{N,\theta}=I/\langle\theta(n)v-\rho(n)v:n\in N,\,v\in I\rangle$ is two-dimensional. More precisely, there is a filtration
 $$
  0\subset J_2\subset J_1=I_{N,\theta},
 $$
 where $J_2$ and $J_1/J_2$ are both one-dimensional, and the action of the stabilizer of $\theta$ is given as follows,
 \begin{align*}
  {\rm diag}(a,b,a,b)v&=\chi(a)\sigma(ab)v\qquad\text{for }v\in J_1/J_2,\\
  {\rm diag}(a,b,a,b)v&=\chi(b)\sigma(ab)v\qquad\text{for }v\in J_2,
 \end{align*}
 for all $a$ and $b$ in $F^\times$.
  \item Assume that $\pi$ is infinite-dimensional and irreducible. Then  there is a filtration
   $$
    0\subset J_3\subset J_2\subset J_1=I_{N,\theta},
   $$
   with the quotients given as follows.
   \begin{itemize}
    \item $J_1/J_2\cong V$
    \item $J_2/J_3\cong V_{ \left[ \begin{smallmatrix} 1\\  *&1\end{smallmatrix} \right]}$. Hence, $J_2/J_3$ is $2$-dimensional if $\pi$ is a principal series representation, $1$-dimensional if $\pi$ is a twist of the Steinberg representation, and $0$ if $\pi$ is supercuspidal.
    \item $J_3\cong \mathcal{S}(F^\times)$.
   \end{itemize}
 The action of the stabilizer of $\theta$ is given as follows,
 \begin{align*}
  {\rm diag}(a,b,a,b)v&=\chi(a)\pi(\mat{a}{}{}{b})v\qquad\text{for }v\in J_1/J_2,\\
  {\rm diag}(a,b,a,b)v&=\chi(b)\pi(\mat{a}{}{}{b})v\qquad\text{for }v\in J_2/J_3,\\
  ({\rm diag}(a,b,a,b)f)(u)&=\chi(b)\omega_\pi(a)f(a^{-1}bu)\qquad\text{for }f\in J_3,\;u\in F^\times,
 \end{align*}
 for all $a$ and $b$ in $F^\times$.
 \end{enumerate}

\end{lemma}
\begin{proof}
It will be easier to work with the conjugate subgroup $N_{\mathrm{alt}}$ and the character $\theta_{\mathrm{alt}}$ of $N_{\mathrm{alt}}$ defined in (\ref{splitthetaconjeq}). For the top quotient from i) of Lemma \ref{casselmanfiltrationlemma} we have
$$
 (I^0/I^1)_{N_{\mathrm{alt}},\theta_{\mathrm{alt}}}=0,
$$
since the subgroup
$$
 \begin{bmatrix}1&*\\&1\\&&1&*\\&&&1\end{bmatrix}
$$
acts trivially on $I^0/I^1$, but $\theta_{\mathrm{alt}}$ is not trivial on this subgroup. We consider the quotient $I^1/I^2 = \cInd_{H}^Q \sigma_1$, where
$$
 H=\begin{bmatrix} *&&*&*\\  &*&*&*\\ &&*&\\ &&&* \end{bmatrix},
$$
and with $\sigma_1$ as in ii) of Lemma \ref{casselmanfiltrationlemma}. We first show that for each function $f$ in the standard model of this representation, the function $f^\circ:\:F\to V$, given by
$$
 f^\circ(w)=f(\begin{bmatrix}1&-w\\&1\\&&1&w\\&&&1\end{bmatrix}),
$$
has compact support. Let $K$ be a compact subset of $Q$ such that the support of $f$ is contained in $HK$. If
$$
 \begin{bmatrix}1&-w\\&1\\&&1&w\\&&&1\end{bmatrix}=\begin{bmatrix} t&&*&x \\&a&b&*\\&&d&\\&&& ad t^{-1} \end{bmatrix}\begin{bmatrix}k_0&x_1&x_2&x_3\\&k_1&k_2&x_4\\&k_3&k_4&x_5\\&&&k_5\end{bmatrix},
$$
with the rightmost matrix being in $K$, then calculations show that $k_3=0$ and $w=k_4^{-1}x_5$. Since $k_4^{-1}$ and $x_5$ vary in bounded subsets, $w$ is confined to a compact subset of $F$. This proves our assertion that $f^\circ$ has compact support.

Next, for each function $f$ in the standard model of $\cInd_{H}^Q \sigma_1$, consider the function $\tilde f:\:F^2\to V$ given by
$$
 \tilde f(u,w)=f(\begin{bmatrix}1&-w\\&1\\&u&1&w\\&&&1\end{bmatrix}).
$$
Let $W$ be the space of all such functions $\tilde f$. Since the map $f\mapsto\tilde f$ is injective, we get a vector space isomorphism $\cInd_{H}^Q \sigma_1\cong W$. Evidently, in this new model, the action of $N_{\mathrm{alt}}$ is given by
\begin{equation}\label{Klingendegjacquetlemmaeq10}
 (\begin{bmatrix}1&-y&&z\\&1&&\\&x&1&y\\&&&1\end{bmatrix}\tilde f)(u,w)=\tilde f(u+x,w+y).
\end{equation}

We claim that $W$ contains $\mathcal{S}(F^2,V)$. Since $W$ is translation invariant, it is enough to prove that $W$ contains the function
$$
 f_{N,v}(u,w)=
  \begin{cases}
   v & \text{if $u,w \in \p^N$,}\\
   0 & \text{if $u \notin \p^N$ or $w \notin \p^N$,}
  \end{cases}
$$
for any $v$ in $V$ and any positive integer $N$. Again by translation invariance, we may assume that $N$ is large enough so that
\begin{equation}\label{Klingendegjacquetlemmaeq11}
 \sigma_1(h)v=v\qquad\text{for }h\in H\cap\Gamma_N,
\end{equation}
where
\begin{equation}\label{GammaNdefeq2}
 \Gamma_{N}=
  \begin{bmatrix} 
   1+\p^N&\p^N&\p^N&\p^N \\
   &1+\p^N&\p^N&\p^N\\
   &\p^N&1+\p^N&\p^N\\
   &&&1+\p^N
  \end{bmatrix} \cap Q.
\end{equation}
Define $f:\:Q\to V$ by
$$
 f(g)=
  \begin{cases}
   \sigma_1(h)v&\text{if $g=hk$ with }h\in H,\;k\in\Gamma_N,\\
   0&g\notin H\Gamma_N.
  \end{cases}
$$
Then, by \eqref{Klingendegjacquetlemmaeq11}, $f$ is a well-defined element of $\cInd_{H}^Q \sigma_1$. It is easy to verify that $\tilde f=f_{N,v}$. This proves our claim that $W$ contains $\mathcal{S}(F^2,V)$.

Now consider the map
\begin{equation}\label{Klingendegjacquetlemmaeq12}
 W\longrightarrow\mathcal{S}(F,V),\qquad \tilde f\longmapsto\Big(w\mapsto f(\begin{bmatrix}1&-w\\&1\\&&1&w\\&&&1\end{bmatrix}s_2)\Big),
\end{equation}
where $s_2$ is defined in \eqref{s1s2defeq}. This map is well-defined, since the function on the right is $(s_2f)^\circ$, which we showed above has compact support. Similar considerations as above show that the map \eqref{Klingendegjacquetlemmaeq12} is surjective.

We claim that the kernel of \eqref{Klingendegjacquetlemmaeq12} is $\mathcal{S}(F^2,V)$. First suppose that $\tilde f$ lies in the kernel; we have to show that $\tilde f$ has compact support. Choose $N$ large enough so that $f$ is right invariant under $\Gamma_N$. Then, for $u$ not in $\p^{-N}$ and $w$ in $F$,
\begin{align*}
 \tilde f(u,w)&=f(\begin{bmatrix}1&-w\\&1\\&u&1&w\\&&&1\end{bmatrix})\\
 &=f(\begin{bmatrix}1&-w\\&1\\&&1&w\\&&&1\end{bmatrix}\begin{bmatrix}1\\&1&u^{-1}\\&&1\\&&&1\end{bmatrix}\begin{bmatrix}1\\&-u^{-1}\\&&-u\\&&&1\end{bmatrix}s_2\begin{bmatrix}1\\&1&u^{-1}\\&&1\\&&&1\end{bmatrix})\\
 &=f(\begin{bmatrix}1&&-wu^{-1}&w^2u^{-1}\\&1&u^{-1}&-wu^{-1}\\&&1\\&&&1\end{bmatrix}\begin{bmatrix}1&-w\\&1\\&&1&w\\&&&1\end{bmatrix}\begin{bmatrix}1\\&-u^{-1}\\&&-u\\&&&1\end{bmatrix}s_2)\\
 &=\pi(\mat{1}{w^2u^{-1}}{}{1})f(\begin{bmatrix}1&-w\\&1\\&&1&w\\&&&1\end{bmatrix}\begin{bmatrix}1\\&-u^{-1}\\&&-u\\&&&1\end{bmatrix}s_2)\\
 &=\chi(-u^{-1})|u|^{-2}\pi(\mat{1}{w^2u^{-1}}{}{1})f(\begin{bmatrix}1&wu^{-1}\\&1\\&&1&-wu^{-1}\\&&&1\end{bmatrix}s_2).
\end{align*}
This last expression is zero by assumption. For fixed $u$ in $\p^{-N}$, the function $\tilde f(u,\cdot)$ has compact support; this follows because each $f^\circ$ has compact support. Combining these facts shows that $\tilde f$ has compact support. Conversely, assume $\tilde f$ is in $\mathcal{S}(F^2,V)$. Then we can find a large enough $N$ such that if $u$ has valuation $-N$, the function $\tilde f(u,\cdot)$ is zero. Looking at the above calculation, we see that, for fixed such $u$,
$$
 f(\begin{bmatrix}1&wu^{-1}\\&1\\&&1&-wu^{-1}\\&&&1\end{bmatrix}s_2)=0
$$
for all $w$ in $F$. This shows that $\tilde f$ is in the kernel of the map \eqref{Klingendegjacquetlemmaeq12}, completing the proof of our claim about this kernel. We now have an exact sequence
\begin{equation}\label{Klingendegjacquetlemmaeq13}
 0\longrightarrow\mathcal{S}(F^2,V)\longrightarrow W\longrightarrow\mathcal{S}(F,V)\longrightarrow0.
\end{equation}
Note that the space $\mathcal{S}(F^2,V)$ is invariant under the action \eqref{Klingendegjacquetlemmaeq10} of $N_{\mathrm{alt}}$. A calculation shows that the action of $N_{\mathrm{alt}}$ on $\mathcal{S}(F,V)$ is given by
\begin{equation}\label{Klingendegjacquetlemmaeq14}
 (\begin{bmatrix}1&-y&&z\\&1&&\\&x&1&y\\&&&1\end{bmatrix}f)(w)=\pi(\mat{1\;}{z-2wy-w^2x}{}{1})f(w)
\end{equation}
for $x,y,z,w$ in $F$ and $f$ in $\mathcal{S}(F,V)$.

We claim that $\mathcal{S}(F,V)_{N_{\mathrm{alt}},\theta_{\mathrm{alt}}}=0$. To prove this, we calculate this Jacquet module in stages. We define a map $p$ from $\mathcal{S}(F,V)$ to
$$
 \mathcal{S}(F,V_{\left[\begin{smallmatrix}1&*\\&1\end{smallmatrix}\right]})=\mathcal{S}(F,V/V(\mat{1}{*}{}{1}))
$$
by sending $f$ to $f$ composed with the natural projection from $V$ to $V/V(\mat{1}{*}{}{1})$. This map is surjective and has kernel $\mathcal{S}(F,V)(\left[\begin{smallmatrix}1&&&*\\&1\\&&1\\&&&1\end{smallmatrix}\right])$. Hence, we obtain an isomorphism
$$
 \mathcal{S}(F,V)_{\left[\begin{smallmatrix}1&&&*\\&1\\&&1\\&&&1\end{smallmatrix}\right]}\cong\mathcal{S}(F,V_{\left[\begin{smallmatrix}1&*\\&1\end{smallmatrix}\right]}).
$$
The action of the group $\left[\begin{smallmatrix}1&*&&\\&1\\&*&1&*\\&&&1\end{smallmatrix}\right]$ on these spaces is trivial. Since $\theta_{\mathrm{alt}}$ is not trivial on this group, this proves our claim that $\mathcal{S}(F,V)_{N_{\mathrm{alt}},\theta_{\mathrm{alt}}}=0$.

By \eqref{Klingendegjacquetlemmaeq13}, we now have $W_{N_{\mathrm{alt}},\theta_{\mathrm{alt}}}\cong\mathcal{S}(F^2,V)_{N_{\mathrm{alt}},\theta_{\mathrm{alt}}}$. The action of $N_{\mathrm{alt}}$ on $\mathcal{S}(F^2,V)$ is given by \eqref{Klingendegjacquetlemmaeq10}. Since $\mathcal{S}(F^2,V)=\mathcal{S}(F)\otimes\mathcal{S}(F)\otimes V$, Lemma \ref{basicFjacquetlemma} implies that the map
$$
 f\longmapsto\int\limits_F\int\limits_F f(u,w)\psi(-w)\,du\,dw
$$
induces an isomorphism $\mathcal{S}(F^2,V)_{N_{\mathrm{alt}},\theta_{\mathrm{alt}}}\cong V$. Moreover, 
a calculation shows that ${\rm diag}(a,a,b,b)$ acts on $\mathcal{S}(F^2,V)_{N_{\mathrm{alt}},\theta_{\mathrm{alt}}}\cong V$ by $\chi(a)\pi(\mat{a}{}{}{b})$.

Finally, we consider the bottom quotient $I^2/I^3=\cInd_{\left[ \begin{smallmatrix} *&&& \\ &*&*& \\ &*&*& \\ &&& * \end{smallmatrix} \right] } ^Q\sigma_2$ with $\sigma_2$ as in iii) of Lemma \ref{casselmanfiltrationlemma}. If we associate with a function $f$ in the standard model of this induced representation the function
$$
 \tilde f(u,v,w)=f(\begin{bmatrix}1&-v&u&w\\&1&&u\\&&1&v\\&&&1\end{bmatrix}),
$$
then, by Lemma \ref{inducedreslemma2}, we obtain an isomorphism $I^2/I^3\cong\mathcal{S}(F^3,V)$. A calculation shows that the action of $N_{\mathrm{alt}}$ on $\mathcal{S}(F^3,V)$ is given by
\begin{equation}\label{Klingendegjacquetlemmaeq15}
 (\begin{bmatrix}1&-y&&z\\&1&&\\&x&1&y\\&&&1\end{bmatrix}f)(u,v,w)=\pi(\mat{1}{}{x}{1})f(u,v+y-ux,w+z+uy)
\end{equation}
for $x,y,z,u,v,w$ in $F$ and $f$ in $\mathcal{S}(F^3,V)$. This time we take Jacquet modules step by step, starting with the $z$-variable. Lemma \ref{basicFjacquetlemma} shows that the map
$$
 f\longmapsto \bigg((u,v)\mapsto\int\limits_Ff(u,v,w)\,dw\bigg)
$$
induces an isomorphism $\mathcal{S}(F^3,V)_{\left[\begin{smallmatrix}1&&&*\\&1\\&&1\\&&&1\end{smallmatrix}\right]}\cong\mathcal{S}(F^2,V)$. On $\mathcal{S}(F^2,V)$ we have the action
$$
 (\begin{bmatrix}1&-y\\&1\\&x&1&y\\&&&1\end{bmatrix}f)(u,v)=\pi(\mat{1}{}{x}{1})f(u,v+y-ux)
$$
for $x,y,u,v$ in $F$ and $f$ in $\mathcal{S}(F^2,V)$.
Part ii) of Lemma \ref{basicFjacquetlemma} shows that the map
$$
 f\longmapsto \bigg(u\mapsto\int\limits_Ff(u,v)\psi(-v)\,dv\bigg)
$$
induces an isomorphism $\mathcal{S}(F^2,V)_{\left[\begin{smallmatrix}1&*\\&1\\&&1&*\\&&&1\end{smallmatrix}\right],\psi}\cong\mathcal{S}(F,V)$. A calculation shows that on $\mathcal{S}(F,V)$ we have the actions
\begin{equation}\label{Klingendegjacquetlemmaeq1}
 (\begin{bmatrix}1\\&1\\&x&1\\&&&1\end{bmatrix}f)(u)=\psi(-ux)\pi(\mat{1}{}{x}{1})f(u)
\end{equation}
for $x,u$ in $F$, and
\begin{equation}\label{Klingendegjacquetlemmaeq1b}
 (\begin{bmatrix}a\\&a\\&&b\\&&&b\end{bmatrix}f)(u)=\chi(b)\pi(\mat{a}{}{}{b})f(a^{-1}bu)
\end{equation}
for $u$ in $F$ and $a,b$ in $F^\times$. The subspace $\mathcal{S}(F^\times,V)$ consisting of functions that vanish at zero is invariant under these actions. We consider the exact sequence
$$
 0\longrightarrow\mathcal{S}(F^\times,V)\longrightarrow\mathcal{S}(F,V)\longrightarrow\mathcal{S}(F,V)/\mathcal{S}(F^\times,V)\longrightarrow 0.
$$
The quotient $\mathcal{S}(F,V)/\mathcal{S}(F^\times,V)$ is isomorphic to $V$ via the map $f\mapsto f(0)$. The actions of the above subgroups on $V$ are given by
\begin{equation}\label{Klingendegjacquetlemmaeq1c}
 \begin{bmatrix}1\\&1\\&x&1\\&&&1\end{bmatrix}v=\pi(\mat{1}{}{x}{1})v
\end{equation}
and
\begin{equation}\label{Klingendegjacquetlemmaeq1d}
 \begin{bmatrix}a\\&a\\&&b\\&&&b\end{bmatrix}v=\chi(b)\pi(\mat{a}{}{}{b})v.
\end{equation}
Taking Jacquet modules on the above sequence gives
$$
 0\longrightarrow\mathcal{S}(F^\times,V)_{ \left[ \begin{smallmatrix} 1\\  &1\\ &*&1\\ &&&1 \end{smallmatrix} \right]}\longrightarrow\mathcal{S}(F,V)_{ \left[ \begin{smallmatrix} 1\\  &1\\ &*&1\\ &&&1 \end{smallmatrix} \right]}\longrightarrow\big(\mathcal{S}(F,V)/\mathcal{S}(F^\times,V)\big)_{ \left[ \begin{smallmatrix} 1\\  &1\\ &*&1\\ &&&1 \end{smallmatrix} \right]}\longrightarrow 0.
$$
In view of \eqref{Klingendegjacquetlemmaeq1c}, the Jacquet module on the right is isomorphic to $V_{ \left[ \begin{smallmatrix} 1\\  *&1\end{smallmatrix} \right]}$. The action of the diagonal subgroup on $V_{ \left[ \begin{smallmatrix} 1\\  *&1\end{smallmatrix} \right]}$ is given by the same formula as in \eqref{Klingendegjacquetlemmaeq1d}.

We consider the map from $\mathcal{S}(F^\times,V)$ to itself given by
$$
 f\longmapsto\Big(u\mapsto\pi(\mat{1}{}{}{u})f(u)\Big).
$$
This map is an isomorphism of vector spaces. The actions \eqref{Klingendegjacquetlemmaeq1}  and \eqref{Klingendegjacquetlemmaeq1b} turn into
\begin{equation}\label{Klingendegjacquetlemmaeq2}
 (\begin{bmatrix}1\\&1\\&x&1\\&&&1\end{bmatrix}f)(u)=\psi(-ux)\pi(\mat{1}{}{ux}{1})f(u)
\end{equation}
and
\begin{equation}\label{Klingendegjacquetlemmaeq2b}
 (\begin{bmatrix}a\\&a\\&&b\\&&&b\end{bmatrix}f)(u)=\chi(b)\pi(\mat{a}{}{}{a})f(a^{-1}bu).
\end{equation}
We define a map $p$ from $\mathcal{S}(F^\times,V)$ to
$$
 \mathcal{S}(F^\times,V_{\left[\begin{smallmatrix}1\\ *&1\end{smallmatrix}\right],\psi})=\mathcal{S}(F^\times,V/V(\mat{1}{}{ *}{1},\psi))
$$
by sending $f$ to $f$ composed with the natural projection from $V$ to $V/V(\mat{1}{}{ *}{1},\psi)$. This map is surjective. The kernel of $p$ consists of all $f$ in $\mathcal{S}(F^\times,V)$ for which there exists a positive integer $l$ such that
\begin{equation}\label{Klingendegjacquetlemmaeq3}
 \int\limits_{\p^{-l}}\psi(-x)\pi(\mat{1}{}{x}{1})f(u)\,dx=0\qquad\text{for all }u\in F^\times.
\end{equation}
Let $W$ be the space of $f$ in $\mathcal{S}(F^\times,V)$ for which there exists a positive integer $k$ such that
\begin{equation}\label{Klingendegjacquetlemmaeq4}
 \int\limits_{\p^{-k}}\begin{bmatrix}1\\&1\\&x&1\\&&&1\end{bmatrix}f\,dx=0,
\end{equation}
so that $\mathcal{S}(F^\times,V)/W=\mathcal{S}(F^\times,V)_{ \left[ \begin{smallmatrix} 1\\  &1\\ &*&1\\ &&&1 \end{smallmatrix} \right]}$. Let $f$ be in $W$. The condition \eqref{Klingendegjacquetlemmaeq4} means that
\begin{equation}\label{Klingendegjacquetlemmaeq5}
 \int\limits_{\p^{-k}}\psi(-ux)\pi(\mat{1}{}{ux}{1})f(u)\,dx=0\qquad\text{for all }u\in F^\times.
\end{equation}
Since $f$ has compact support in $F^\times$, the conditions \eqref{Klingendegjacquetlemmaeq3} and \eqref{Klingendegjacquetlemmaeq5} are equivalent. It follows that
$$
 \mathcal{S}(F^\times,V)_{ \left[ \begin{smallmatrix} 1\\  &1\\ &*&1\\ &&&1 \end{smallmatrix} \right]}\cong\mathcal{S}(F^\times,V_{\left[\begin{smallmatrix}1\\ *&1\end{smallmatrix}\right],\psi}).
$$
The diagonal subgroup acts on $\mathcal{S}(F^\times,V_{\left[\begin{smallmatrix}1\\ *&1\end{smallmatrix}\right],\psi})$ by the same formula as in \eqref{Klingendegjacquetlemmaeq2b}.
\end{proof}

\section{The main results}\label{besseltablesec}
Having assembled all the required tools, we are now ready to prove the three main results of this paper mentioned in the introduction.
\subsection{Existence of Bessel functionals}
In this section we prove that every irreducible, admissible representation $(\pi,V)$ of $\GSp(4,F)$ which is not a twist of the trivial representation admits a Bessel functional. The proof uses the $P_3$-module $V_{Z^J}$ and the $G^J$-module $V_{Z^J,\psi}$. The first module is closely related to the theory of zeta integrals. The second module, on the other hand, is related to the theory of representations of the metaplectic group $\meta(2,F)$.

\begin{lemma}\label{Nthetaexistslemma}
 Let $(\pi,V)$ be a smooth representation of $N$. Then there exists a character $\theta$ of $N$ such that $V_{N,\theta}\neq0$.
\end{lemma}
\begin{proof}
This follows immediately from Lemma 1.6 of \cite{RoSc2011}.
\end{proof}

Let $\meta(2,F)$ be the metaplectic group, defined as in Sect.\ 1 of \cite{RoSc2011}. Let $m$ be in $F^\times$. We will use the \emph{Weil representation} $\piw^m$ of $\meta(2,F)$ on $\mathcal{S}(F)$ associated to the quadratic form $q(x)=x^2$ and $\psi^m$. This is as defined on pp.~3-4 of \cite{Wald1980} and p.~223 of \cite{Wald1991}. The only explicit property of $\piw^m$ we will use is
\begin{align}
(\piw^m (\mat{1}{b}{}{1}, 1)f)(x)&=\psi(mbx^2)f(x), \label{upperweileq}
\end{align}
for $b$ in $F$ and $f$ in $\mathcal{S}(F)$. We define an action of $N_Q$, introduced in \eqref{NQdefeq}, on the Schwartz space $\mathcal{S}(F)$ by
\begin{equation}\label{piswmformula1}
 \pis^m(\begin{bmatrix}1&\lambda&\mu&\kappa\\&1&&\mu\\&&1&-\lambda\\&&&1\end{bmatrix}f)(x)=\psi^m(\kappa+(2x+\lambda)\mu)f(x+\lambda)
\end{equation}
for $f$ in $\mathcal{S}(F)$. This representation of $N_Q$ is called the \emph{Schr\"odinger representation}. 

Given a smooth, genuine representation $(\tau,W)$ of $\meta(2,F)$, we define a representation $\tau^J$ of $G^J$ on the space $W\otimes\mathcal{S}(F)$ by the formulas
\begin{align}
 \label{GJrepeq1} \tau^J(\begin{bmatrix}1\\&a&b\\&c&d\\&&&1\end{bmatrix})(v\otimes f)&=\tau(\mat{a}{b}{c}{d},1)v\otimes\piw^m(\mat{a}{b}{c}{d},1)f,\\
 \label{GJrepeq2} \tau^J(\begin{bmatrix}1&\lambda&\mu&\kappa\\&1&&\mu\\&&1&-\lambda\\&&&1\end{bmatrix})(v\otimes f)&=v\otimes\pis^m(\begin{bmatrix}1&\lambda&\mu&\kappa\\&1&&\mu\\&&1&-\lambda\\&&&1\end{bmatrix})f.
\end{align}
Computations show that $\tau^J$ is a smooth representation of $G^J$. Moreover, the map that sends $\tau$ to $\tau^J$ is a bijection between the set of equivalence classes of smooth, genuine representations of $\meta(2,F)$, and smooth representations of $G^J$ with central character $\psi^m$. The proof of this fact is based on the Stone-von Neumann Theorem; see Theorem 2.6.2 of \cite{BS1998}. Under this bijection, irreducible $\tau$ correspond to irreducible $\tau^J$.

\begin{lemma} \label{GJWhitlemma}
 Let $m$ be in $F^\times$. Let $(\tau^J,W^J)$ be a non-zero, irreducible, smooth representation of $G^J$ with central character $\psi^m$. Then $\dim W^J_{N,\theta_{a,0,m}} \leq1$ for all $a$ in $F^\times$ and $\dim W^J_{N,\theta_{a,0,m}} =1$ for some $a$ in $F^\times$. This dimension depends only on the class of $a$ in $F^\times/F^{\times2}$.
\end{lemma}
\begin{proof}
By the above discussion, there exists an irreducible, genuine, admissible representation $\tau$ of $\meta(2,F)$ such that $\tau^J\cong \tau \otimes \pisw^m$.
Using \eqref{upperweileq}, \eqref{piswmformula1} and iii) of Lemma \ref{basicFjacquetlemma}, an easy calculation shows that
$$
 W^J_{\left[\begin{smallmatrix}1&&*&*\\&1&*&*\\&&1\\&&&1\end{smallmatrix}\right],\theta_{a,0,m}}\cong W_{\left[\begin{smallmatrix}1&*\\&1\end{smallmatrix}\right],\psi^a}.
$$
By Lemme 2 on p.~226 of \cite{Wald1991}, the space on the right is at most one-dimensional, and is one-dimensional for some $a$ in $F^\times$. Moreover, the dimension depends only on the class of $a$ in $F^\times/F^{\times2}$.
\end{proof}

\begin{proposition}\label{Nthetaprop}
 Let $(\pi,V)$ be an irreducible, admissible representation of $\GSp(4,F)$. Then the following statements are equivalent.
 \begin{enumerate}
  \item $\pi$ is not a twist of the trivial representation.
  \item There exists a non-trivial character $\theta$ of $N$ such that $V_{N,\theta}\neq0$.
  \item There exists a non-degenerate character $\theta$ of $N$ such that $V_{N,\theta}\neq0$.  
 \end{enumerate}
\end{proposition}
\begin{proof}
i) $\Rightarrow$ ii) Assume that $V_{N,\theta}=0$ for all non-trivial $\theta$. By Lemma \ref{Nthetaexistslemma}, it follows that $V_{N,1}\neq0$. In particular, the $P_3$-module $V_{Z^J}$ is non-zero. By using Theorem \ref{finitelength} and inspecting tables A.5 and A.6 in \cite{NF}, one can see that $V_{Z^J}$ contains an irreducible subquotient $\tau$ of the form $\tau^{P_3}_{\GL(0)}(1)$, or $\tau^{P_3}_{\GL(1)}(\chi)$ for a character $\chi$ of $F^\times$, or $\tau^{P_3}_{\GL(2)}(\rho)$ for an irreducible, admissible, infinite-dimensional representation $\rho$ of $\GL(2,F)$; it is here that we use the hypothesis that $\pi$ is not one-dimensional. For $a,b$ in $F$ we define a character of the subgroup $\left[\begin{smallmatrix}1&*&*\\&1\\&&1\end{smallmatrix}\right]$ of $P_3$ by
\begin{equation}\label{thetaabdefeq}
 \theta_{a,b}(\begin{bmatrix}1&x&y\\&1\\&&1\end{bmatrix})=\psi(ax+by).
\end{equation}
By Lemma 2.5.4 or Lemma 2.5.5 of \cite{NF}, or the infinite-dimensionality of $\rho$ if $\tau=\tau^{P_3}_{\GL(2)}(\rho)$,
$$
 \tau_{\left[\begin{smallmatrix}1&*&*\\&1\\&&1\end{smallmatrix}\right],\theta_{a,b}}\neq0
$$
for some $(a,b)\neq(0,0)$. This implies that $V_{N,\theta_{a,b,0}}\neq0$, contradicting our assumption.

ii) $\Rightarrow$ iii) The hypothesis implies that $V_{Z^J,\psi^m}$ is non-zero for some $m$ in $F^\times$. We observe that $V_{Z^J,\psi^m}$ is a smooth $G^J$ representation. By Lemma 2.6 of \cite{BeZe1976}, there exists an irreducible subquotient $(\tau^J,W^J)$ of this $G^J$ module. By Lemma \ref{GJWhitlemma}, we have $\dim W^J_{N,\theta_{a,0,m}}=1$ for some $a$ in $F^\times$. This implies that $V_{N,\theta_{a,0,m}}\neq0$.

iii) $\Rightarrow$ i) is obvious.
\end{proof}

\begin{theorem}\label{existencetheorem}
 Let $(\pi,V)$ be an irreducible, admissible representation of $\GSp(4,F)$. Assume that $\pi$ is not one-dimensional. Then $\pi$ admits a $(\Lambda,\theta)$-Bessel functional for some non-degenerate character $\theta$ of $N$ and some character $\Lambda$ of $T$. If $\pi$ is non-generic and supercuspidal, then every Bessel functional for $\pi$ is non-split.
\end{theorem}
\begin{proof}
By Proposition \ref{Nthetaprop}, there exists a non-degenerate $\theta$ such that $V_{N,\theta}\neq0$. Assume that $\theta$ is non-split. Then, since the center $F^\times$ of $\GSp(4,F)$ acts by a character on $V_{N,\theta}$ and $T/F^\times$ is compact, $V_{N,\theta}$ decomposes as a direct sum over characters of $T$. It follows that a $(\Lambda,\theta)$-Bessel functional exists for some character $\Lambda$ of $T$.

Now assume that $\theta$ is split. We may assume that $S$ is the matrix in \eqref{splitSeq}.  Let $V_0,V_1,V_2$ be the modules appearing in the $P_3$-filtration, as in Theorem \ref{finitelength}. Since $V_{N,\theta}\neq0$, we must have
$$
 (V_0/V_1)_{\left[\begin{smallmatrix}1&*&*\\&1\\&&1\end{smallmatrix}\right],\theta_{0,1}}\neq0,\qquad
 (V_1/V_2)_{\left[\begin{smallmatrix}1&*&*\\&1\\&&1\end{smallmatrix}\right],\theta_{0,1}}\neq0,\qquad\text{or}\qquad
 (V_2)_{\left[\begin{smallmatrix}1&*&*\\&1\\&&1\end{smallmatrix}\right],\theta_{0,1}}\neq0,
$$
where we use the notation \eqref{thetaabdefeq}. It is immediate from \eqref{tauP3GL2eq} that the first space is zero. If the second space is non-zero, then $\pi$ admits a split Bessel functional by iii) of Proposition \ref{nongenericsplitproposition}. If the third space is non-zero, then $\pi$ is generic by Theorem \ref{finitelength}, and hence, by Proposition \ref{GSp4genericprop}, admits a split Bessel functional.

For the last statement, assume that $\pi$ is non-generic and supercuspidal. Then $V_{Z^J}=0$ by Theorem \ref{finitelength}. Hence, $V_{N,\theta}=0$ for any split $\theta$. It follows that all Bessel functionals for $\pi$ are non-split.
\end{proof}

\subsection{The table of Bessel functionals}\label{maintheoremproofsec}
In this section, given a non-supercuspidal representation $\pi$, or a $\pi$ that is in an $L$-packet with a non-supercuspidal representation, we determine the set of $(\Lambda,\theta)$ for which $\pi$ admits a $(\Lambda,\theta)$-Bessel functional.
\begin{lemma}\label{Kcompactexactlemma}
 Let $\theta$ be as in \eqref{thetaSsetupeq}, and let $T$ be the corresponding torus. Assume that the associated quadratic extension $L$ is a field. Let $V_1$, $V_2$, $V_3$ and $W$ be smooth representations of $T$. Assume that these four representations all have the same central character. Assume further that there is an exact sequence of $T$-modules
 $$
  0\longrightarrow V_1\longrightarrow V_2\longrightarrow V_3\longrightarrow0.
 $$
 Then the sequence of $T$-modules
 $$
  0\longrightarrow\Hom_{T}(V_3,W)\longrightarrow\Hom_{T}(V_2,W)\longrightarrow\Hom_{T}(V_1,W)\longrightarrow0
 $$
 is exact.
\end{lemma}
\begin{proof}
It is easy to see that the sequence
 $$
  0\longrightarrow\Hom_{T}(V_3,W)\longrightarrow\Hom_{T}(V_2,W)\longrightarrow\Hom_{T}(V_1,W)
 $$
is exact. We will prove the surjectivity of the last map. Let $f$ be in $\Hom_{T}(V_1,W)$. We extend $f$ to a linear map $f_1$ from $V_2$ to $W$. We define another linear map $f_2$ from $V_2$ to $W$ by
$$
 f_2(v)=\int\limits_{T/F^\times}t^{-1}\cdot f_1(t\cdot v)\,dt.
$$
This is well-defined by the condition on the central characters, the compactness of $T/F^\times$, and the smoothness hypothesis. Evidently, $f_2$ is in $\Hom_{T}(V_2,W)$ and maps to a multiple of~$f$.
\end{proof}

\begin{theorem}\label{mainnonsupercuspidaltheorem}
 The following table shows the Bessel functionals admitted by the irreducible, admissible, non-supercuspidal representations of $\GSp(4,F)$. The column ``$L\leftrightarrow\xi$'' indicates that the field $L$ is the quadratic extension of $F$ corresponding to the non-trivial, quadratic character $\xi$ of $F^\times$; this is only relevant for representations in groups V and IX. The pairs of characters $(\chi_1,\chi_2)$ in the ``$L=F\times F$'' column for types IIIb and IVc refer to the characters of $T=\{{\rm diag}(a,b,a,b):\:a,b\in F^\times\}$ given by ${\rm diag}(a,b,a,b)\mapsto\chi_1(a)\chi_2(b)$. In representations of group IX, the symbol $\mu$ denotes a non-Galois-invariant character of $L^\times$, where $L$ is the quadratic extension corresponding to $\xi$. The Galois conjugate of $\mu$ is denoted by $\mu'$. The irreducible, admissible, supercuspidal representation of $\GL(2,F)$ corresponding to $\mu$ is denoted by $\pi(\mu)$. Finally, the symbol $\Norm$ in the table stands for the norm map $\Norm_{L/F}$.
 In the split case, the character $\sigma\circ\Norm$ is the same as $(\sigma,\sigma)$. In the table, the phrase ``all $\Lambda$'' means all characters $\Lambda$ of $T$ whose restriction to $F^\times$ is the central character of the representation of $\GSp(4,F)$.
\end{theorem}
$$\renewcommand{\arraystretch}{1.09}\renewcommand{\arraycolsep}{0.07cm}
 \begin{array}{cccccc}
  \toprule
  &&\text{representation}&
  \multicolumn{3}{c}{(\Lambda,\theta)\text{-Bessel functional exists exactly for \ldots}}
  \\
  \cmidrule{4-6}
  &&&L=F\times F&\multicolumn{2}{c}{L/F\text{ a field extension}}\\
  \cmidrule{5-6}
  &&&&L\leftrightarrow\xi&L\not\leftrightarrow\xi\\
  \toprule
  {\rm I}&& \chi_1 \times \chi_2 \rtimes \sigma\ 
  \mathrm{(irreducible)}&\text{all }\Lambda&\multicolumn{2}{c}{\text{all }\Lambda}\\
  \midrule
  \mbox{II}&\mbox{a}&\chi \St_{\GL(2)} \rtimes \sigma&
   \text{all }\Lambda&\multicolumn{2}{c}{\Lambda\neq(\chi\sigma)\circ\Norm}\\
  \cmidrule{2-6}
  &\mbox{b}&\chi \triv_{\GL(2)} \rtimes \sigma
   &\Lambda=(\chi\sigma)\circ\Norm&\multicolumn{2}{c}{\Lambda=(\chi\sigma)\circ\Norm}\\
  \midrule
  \mbox{III}&\mbox{a}&\chi \rtimes \sigma \St_{\GSp(2)}&\text{all }\Lambda
   &\multicolumn{2}{c}{\text{all }\Lambda}\\\cmidrule{2-6}
  &\mbox{b}&\chi \rtimes \sigma \triv_{\GSp(2)}
   &\Lambda\in\{(\chi\sigma,\sigma),(\sigma,\chi\sigma)\}&\multicolumn{2}{c}{\text{---}}\\
  \midrule
  \mbox{IV}&\mbox{a}&\sigma\St_{\GSp(4)}&\text{all }\Lambda&
  \multicolumn{2}{c}{\Lambda\neq\sigma\circ\Norm}\\\cmidrule{2-6}
  &\mbox{b}&L(\nu^2,\nu^{-1}\sigma\St_{\GSp(2)})&\Lambda=\sigma\circ\Norm 
   &\multicolumn{2}{c}{\Lambda=\sigma\circ\Norm}\\\cmidrule{2-6}
  &\mbox{c}&L(\nu^{3/2}\St_{\GL(2)},\nu^{-3/2}\sigma)
   &\Lambda=(\nu^{\pm1}\sigma,\nu^{\mp1}\sigma)&\multicolumn{2}{c}{\text{---}}\\\cmidrule{2-6}
  &\mbox{d}&\sigma\triv_{\GSp(4)}&\text{---}&\multicolumn{2}{c}{\text{---}}\\
  \midrule
  \mbox{V}&\mbox{a}&\delta([\xi,\nu \xi], \nu^{-1/2} \sigma)&\text{all }\Lambda
   &\Lambda\neq\sigma\circ\Norm&\sigma\circ\Norm\neq\Lambda\neq(\xi\sigma)\circ\Norm \\\cmidrule{2-6}
  &\mbox{b}&L(\nu^{1/2}\xi\St_{\GL(2)},\nu^{-1/2} \sigma)&\Lambda=\sigma\circ\Norm 
   &\text{---}&\Lambda=\sigma\circ\Norm \\\cmidrule{2-6}
  &\mbox{c}&L(\nu^{1/2}\xi\St_{\GL(2)},\xi\nu^{-1/2} \sigma)
   &\Lambda=(\xi\sigma)\circ\Norm
   &\text{---}&\Lambda=(\xi\sigma)\circ\Norm\\\cmidrule{2-6}
  &\mbox{d}&L(\nu\xi,\xi\rtimes\nu^{-1/2}\sigma)&\text{---}
   &\Lambda=\sigma\circ\Norm&\text{---}\\
  \midrule
  \mbox{VI}&\mbox{a}&\tau(S, \nu^{-1/2}\sigma)&\text{all }\Lambda
   &\multicolumn{2}{c}{\Lambda\neq\sigma\circ\Norm}\\\cmidrule{2-6}
  &\mbox{b}&\tau(T, \nu^{-1/2}\sigma)&\text{---}&\multicolumn{2}{c}{\Lambda=\sigma\circ\Norm}\\\cmidrule{2-6}
  &\mbox{c}&L(\nu^{1/2}\St_{\GL(2)},\nu^{-1/2}\sigma)
   &\Lambda=\sigma\circ\Norm&\multicolumn{2}{c}{\text{---}}\\\cmidrule{2-6}
  &\mbox{d}&L(\nu,1_{F^\times}\rtimes\nu^{-1/2}\sigma)&
   \Lambda=\sigma\circ\Norm&\multicolumn{2}{c}{\text{---}}\\
  \toprule
  \mbox{VII}&&\chi \rtimes \pi&\text{all }\Lambda
   &\multicolumn{2}{c}{\text{all }\Lambda}\\\midrule
  \mbox{VIII}&\mbox{a}&\tau(S, \pi)&\text{all }\Lambda&\multicolumn{2}{c}{{\rm Hom}_T(\pi,\C_\Lambda)\neq0}\\\cmidrule{2-6}
  &\mbox{b}&\tau(T, \pi)
   &\text{---}&\multicolumn{2}{c}{{\rm Hom}_T(\pi,\C_\Lambda)=0}\\
  \midrule
  \mbox{IX}&\mbox{a}&\delta(\nu\xi,\nu^{-1/2}\pi(\mu))&\text{all }\Lambda
   &\mu\neq\Lambda\neq\mu'&\text{all }\Lambda\\\cmidrule{2-6}
  &\mbox{b}&L(\nu\xi,\nu^{-1/2}\pi(\mu))
   &\text{---}&\Lambda=\mu\text{ or }\Lambda=\mu'&\text{---}\\
  \toprule
  \mbox{X}&&\pi \rtimes \sigma&\text{all }\Lambda
   &\multicolumn{2}{c}{{\rm Hom}_T(\sigma\pi,\C_\Lambda)\neq0}\\
  \midrule
  \mbox{XI}&\mbox{a}&\delta (\nu^{1/2}\pi,\nu^{-1/2}\sigma)&\text{all }\Lambda&\multicolumn{2}{c}{\Lambda\neq\sigma\circ\Norm\:\text{ and }{\rm Hom}_T(\sigma\pi,\C_{\Lambda})\neq0}\\
  \cmidrule{2-6}
  &\mbox{b}&L(\nu^{1/2}\pi,\nu^{-1/2}\sigma)&\Lambda=\sigma\circ\Norm&\multicolumn{2}{c}{\Lambda=\sigma\circ\Norm\:\text{ and } {\rm Hom}_T(\pi,\C_1)\neq0}\\
  \toprule
  \mbox{Va$^*$}&&\delta^*([\xi,\nu\xi],\nu^{-1/2}\sigma)&\text{---}&\Lambda=\sigma\circ\Norm&\text{---}\\
  \midrule
  \mbox{XIa$^*$}&&\delta^*(\nu^{1/2}\pi,\nu^{-1/2}\sigma)&\text{---}&\multicolumn{2}{c}{\Lambda=\sigma\circ\Norm\:\text{ and }{\rm Hom}_T(\pi^{\mathrm{JL}},\C_1)\neq0}\\
  \toprule
 \end{array}
$$
\begin{proof}
We will go through all representations in the table and explain how the statements follow from our preparatory sections.

\underline{I}:
This follows from Proposition \ref{GSp4genericprop} and Lemma \ref{Klingendegjacquetlemma1}.

\underline{IIa}: In the split case this follows from Proposition \ref{GSp4genericprop}. In the non-split case it follows from Lemma \ref{siegelinducedbesselwaldspurgerlemma} together
with (\ref{StGL2Waldspurgereq2}).

\underline{IIb}: This follows from Lemma \ref{siegelinducedbesselwaldspurgerlemma}; see \eqref{siegelinducedonedimjacqueteq}.

\underline{IIIa}: This follows from Proposition \ref{GSp4genericprop} and Lemma \ref{Klingendegjacquetlemma1}.

\underline{IIIb}: It follows from Lemma \ref{Klingendegjacquetlemma1} that IIIb type representations have no non-split Bessel functionals. The split case follows from either Proposition \ref{nongenericsplitproposition} or i) of Lemma \ref{Klingendegjacquetlemma}. Note that the characters $(\chi\sigma,\sigma)$ and $(\sigma,\chi\sigma)$ are Galois conjugates of each other.

\underline{IVd}:
It is easy to see that the twisted Jacquet modules of the trivial representation
are zero.

\underline{IVb}:
By (2.9) of \cite{NF} there is a short exact sequence
$$
 0\longrightarrow{\rm IVb}\longrightarrow\nu^{3/2}1_{\GL(2)}\rtimes\nu^{-3/2}\sigma
 \longrightarrow\sigma1_{\GSp(4)}\longrightarrow0.
$$
Taking twisted Jacquet modules and observing \eqref{siegelinducedonedimjacqueteq}, we get
$$
 ({\rm IVb})_{N,\theta}\cong(\nu^{3/2}1_{\GL(2)}\rtimes\nu^{-3/2}\sigma)_{N,\theta}=\C_{\sigma\circ \Norm_{L/F}}
$$
as $T$-modules.

\underline{IVc}:
By (2.9) of \cite{NF} there is a short exact sequence
$$
 0\longrightarrow{\rm IVc}\longrightarrow\nu^2\rtimes\nu^{-1}\sigma1_{\GSp(2)}
 \longrightarrow\sigma1_{\GSp(4)}\longrightarrow0.
$$
Taking twisted Jacquet modules gives
$$
 ({\rm IVc})_{N,\theta}\cong(\nu^2\rtimes\nu^{-1}\sigma1_{\GSp(2)})_{N,\theta}.
$$
Hence IVc admits the same Bessel functionals as the full induced representation
$\nu^2\rtimes\nu^{-1}\sigma1_{\GSp(2)}$. By Lemma \ref{Klingendegjacquetlemma1}, any such
Bessel functional is necessarily split. Assume that $\theta$ is as in (\ref{splitthetaeq}). Then, using Lemma \ref{Klingendegjacquetlemma}, it follows that IVc admits the $(\Lambda,\theta)$-Bessel functional for
\begin{equation}\label{IVcpossibleLambdaeq}
 \Lambda(\begin{bmatrix}a\\&b\\&&a\\&&&b\end{bmatrix})=\nu(ab^{-1})\sigma(ab).
\end{equation}
which we write as $(\nu\sigma,\nu^{-1}\sigma)$.
By \eqref{besselgaloiseq}, IVc also admits a $(\Lambda,\theta)$-Bessel functional for $\Lambda=(\nu^{-1}\sigma,\nu\sigma)$. Again by Lemma \ref{Klingendegjacquetlemma}, IVc does not admit a $(\Lambda,\theta)$-Bessel functional for any other $\Lambda$.

\underline{IVa}: In the split case this follows from Proposition \ref{GSp4genericprop}. Assume $\theta$ is non-split. By (2.9) of \cite{NF}, there is an exact sequence
$$
 0\longrightarrow\sigma\St_{\GSp(4)}\longrightarrow\nu^2\rtimes\nu^{-1}\sigma\St_{\GSp(2)}\longrightarrow {\rm IVb}\longrightarrow0.
$$
Taking Jacquet modules, we get
$$
 0\longrightarrow(\sigma\St_{\GSp(4)})_{N,\theta}\longrightarrow(\nu^2\rtimes\nu^{-1}\sigma\St_{\GSp(2)})_{N,\theta}\longrightarrow ({\rm IVb})_{N,\theta}\longrightarrow0.
$$
Keeping in mind Lemma \ref{Kcompactexactlemma}, the result now follows from Lemma \ref{Klingendegjacquetlemma1} and the result for IVb.

\underline{Vd}: This was proved in Corollary \ref{bsummary}.

\underline{Vb and Vc}: Let $\xi$ be a non-trivial quadratic character of $F^\times$. By (2.10) of \cite{NF}, there are exact sequences
$$
 0\longrightarrow{\rm Vb}\longrightarrow\nu^{1/2}\xi\triv_{\GL(2)}\rtimes\xi\nu^{-1/2}\sigma
 \longrightarrow{\rm Vd}\longrightarrow0
$$
and
$$
 0\longrightarrow{\rm Vc}\longrightarrow\nu^{1/2}\xi\triv_{\GL(2)}\rtimes\nu^{-1/2}\sigma
 \longrightarrow{\rm Vd}\longrightarrow0.
$$
Taking Jacquet modules and observing (\ref{siegelinducedonedimjacqueteq}), we get
\begin{equation}\label{VdBesseleq1}
 0\longrightarrow({\rm Vb})_{N,\theta}\longrightarrow\C_{\sigma\circ \Norm_{L/F}}
 \longrightarrow({\rm Vd})_{N,\theta}\longrightarrow0
\end{equation}
and
\begin{equation}\label{VdBesseleq2}
 0\longrightarrow({\rm Vc})_{N,\theta}\longrightarrow\C_{(\xi\sigma)\circ \Norm_{L/F}}
 \longrightarrow({\rm Vd})_{N,\theta}\longrightarrow0.
\end{equation}
Hence the results for Vb and Vc follow from the result for Vd.

\underline{Va}: In the split case this follows from Proposition \ref{GSp4genericprop}. Assume $\theta$ is non-split. Assume first that $\xi$ corresponds to the quadratic extension $L/F$.
As we just saw, $({\rm Vb})_{N,\theta}=0$ in this case. By (2.10) of \cite{NF}, there is an exact sequence
$$
 0\longrightarrow{\rm Va}\longrightarrow\nu^{1/2}\xi\St_{\GL(2)}\rtimes\nu^{-1/2}\sigma\longrightarrow{\rm Vb}\longrightarrow0.
$$
Taking Jacquet modules, it follows that
$$
 ({\rm Va})_{N,\theta}=(\nu^{1/2}\xi\St_{\GL(2)}\rtimes\nu^{-1/2}\sigma)_{N,\theta}.
$$
By Lemma \ref{siegelinducedbesselwaldspurgerlemma}, the space of $(\Lambda,\theta)$-Bessel functionals on the representation Va is isomorphic to ${\rm Hom}_T(\sigma\xi\St_{\GL(2)},\C_\Lambda)$. Using (\ref{StGL2Waldspurgereq2}), it follows that Va admits a $(\Lambda,\theta)$-Bessel functional if and only if $\Lambda\neq(\sigma\xi)\circ \Norm_{L/F}=\sigma\circ \Norm_{L/F}$.

Now assume that $\xi$ does not correspond to the quadratic extension $L/F$. Then, by what we already proved for Vb, we have an exact sequence
\begin{equation}\label{VdBesseleq5}
 0\longrightarrow({\rm Va})_{N,\theta}\longrightarrow
 (\nu^{1/2}\xi\St_{\GL(2)}\rtimes\nu^{-1/2}\sigma)_{N,\theta}
 \longrightarrow\C_{\sigma\circ \Norm_{L/F}}\longrightarrow0.
\end{equation}
Using Lemma \ref{Kcompactexactlemma}, it follows that the possible characters $\Lambda$ for Va are those of $(\nu^{1/2}\xi\St_{\GL(2)}\rtimes\nu^{-1/2}\sigma)_{N,\theta}$ with the exception of $\sigma\circ \Norm_{L/F}$. By Lemma \ref{siegelinducedbesselwaldspurgerlemma} and (\ref{StGL2Waldspurgereq2}), these are all characters other than $\sigma\circ \Norm_{L/F}$ and $(\xi\sigma)\circ \Norm_{L/F}$.

\underline{VIc and VId}: By (2.11) of \cite{NF}, there is an exact sequence
$$
 0\longrightarrow{\rm VIc}\longrightarrow\triv_{F^\times}\rtimes\sigma\triv_{\GSp(2)}
 \longrightarrow{\rm VId}\longrightarrow0.
$$
It follows from Lemma \ref{Klingendegjacquetlemma1} that VIc and VId have no non-split Bessel functionals. The split case follows from Proposition \ref{nongenericsplitproposition}.

\underline{VIa}: In the split case this follows from Proposition \ref{GSp4genericprop}. Assume that $\theta$ is non-split. By (2.11) of \cite{NF}, there is an exact sequence
\begin{equation}\label{VIcexacteq1}
 0\longrightarrow{\rm VIa}\longrightarrow
 \nu^{1/2}\St_{\GL(2)}\rtimes\nu^{-1/2}\sigma
 \longrightarrow{\rm VIc}\longrightarrow0.
\end{equation}
Taking Jacquet modules and observing the result for VIc, we get $({\rm VIa})_{N,\theta}=(\nu^{1/2}\St_{\GL(2)}\rtimes\nu^{-1/2}\sigma)_{N,\theta}$. Hence the result follows from Lemma \ref{siegelinducedbesselwaldspurgerlemma} and (\ref{StGL2Waldspurgereq2}).

\underline{VIb}: By (2.11) of \cite{NF}, there is an exact sequence
\begin{equation}\label{VIdBesseleq3}
 0\longrightarrow({\rm VIb})_{N,\theta}\longrightarrow
 (\nu^{1/2}\triv_{\GL(2)}\rtimes\nu^{-1/2}\sigma)_{N,\theta}
 \longrightarrow({\rm VId})_{N,\theta}\longrightarrow0.
\end{equation}
By (\ref{siegelinducedonedimjacqueteq}), the middle term equals $\C_{\sigma\circ \Norm_{L/F}}$.
One-dimensionality implies that the sequence splits, so that
\begin{equation}\label{VIdBesseleq4}
 {\rm Hom}_T(\C_{\sigma\circ \Norm_{L/F}},\C_\Lambda)=
 {\rm Hom}_D({\rm VIb},\C_{\Lambda\otimes\theta})\oplus{\rm Hom}_D({\rm VId},\C_{\Lambda\otimes\theta})
\end{equation}
($D$ is the Bessel subgroup defined in \eqref{Ddefeq}). Hence the VIb case follows from the known result for VId.

\underline{VII}:
This follows from Proposition \ref{GSp4genericprop} and Lemma \ref{Klingendegjacquetlemma1}.

\underline{VIIIa and VIIIb}:
In the split case this follows from Proposition \ref{GSp4genericprop} and v) of Proposition \ref{nongenericsplitproposition}. Assume that $\theta$ is non-split. Since we are in a unitarizable situation, the sequence
$$
 0\longrightarrow{\rm VIIIa}\longrightarrow
 1_{F^\times}\rtimes\pi\longrightarrow
 {\rm VIIIb}\longrightarrow0
$$
splits. It follows that
\begin{equation}\label{VIIIbBesseleq1}
 \Hom_D(1_{F^\times}\rtimes\pi,\C_{\Lambda\otimes\theta})
 =\Hom_D({\rm VIIIa},\C_{\Lambda\otimes\theta})\oplus
 \Hom_D({\rm VIIIb},\C_{\Lambda\otimes\theta}).
\end{equation}
By Lemma \ref{Klingendegjacquetlemma1}, the space on the left is one-dimensional for any $\Lambda$. Therefore the Bessel functionals of VIIIb are complementary to those of VIIIa.

Assume that VIIIa admits a $(\Lambda,\theta)$-Bessel functional. Then, by Corollary \ref{fourdimthetatheoremcor1} and Theorem \ref{Ganthetatheorem}, we have $\Hom_T(\pi,\C_\Lambda)\neq0$. Conversely, assume that $\Hom_T(\pi,\C_\Lambda)\neq0$ and assume that VIIIa does not admit a $(\Lambda,\theta)$-Bessel functional; we will obtain a contradiction.
By \eqref{VIIIbBesseleq1}, we have $\Hom_D({\rm VIIIb},\C_{\Lambda\otimes\theta})\neq0$. By Corollary \ref{fourdimthetatheoremcor1} and Theorem \ref{Ganthetatheorem}, we have $\Hom_T(\pi^{\mathrm{JL}},\C_\Lambda)\neq0$. This contradicts \eqref{waldspurgerpropeq2}.

The result for VIIIb now follows from \eqref{VIIIbBesseleq1}.

\underline{IXb}: This was proved in Corollary \ref{bsummary}.

\underline{IXa}:
In the split case this follows from Proposition \ref{GSp4genericprop}. Assume that $\theta$ is non-split. We have an exact sequence
$$
 0\longrightarrow{\rm IXa}\longrightarrow\nu\xi\rtimes\nu^{-1/2}\pi\longrightarrow{\rm IXb}\longrightarrow0.
$$
By Lemma \ref{Klingendegjacquetlemma1}, the space ${\rm Hom}_D(\nu\xi\rtimes\nu^{-1/2}\pi,\C_{\Lambda\otimes\theta})$ is one-dimensional, for any character $\Lambda$ of $L^\times$ satisfying the central character condition. It follows that the possible Bessel functionals of IXa are complementary to those of IXb.

\underline{X}:
In the split case this follows from Proposition \ref{GSp4genericprop}. In the non-split case it follows from Lemma \ref{siegelinducedbesselwaldspurgerlemma}.

\underline{XIa and XIb}:
In the split case this follows from Proposition \ref{GSp4genericprop} and Proposition \ref{nongenericsplitproposition}; note that the $V_1/V_2$ quotient of XIb equals $\tau_{\GL(1)}^{P_3}(\nu\sigma)$ by Table A.6 of \cite{NF}. Assume that $L/F$ is not split, and consider the exact sequence
\begin{equation}\label{XIaBesseleq1}
 0\longrightarrow({\rm XIa})_{N,\theta}\longrightarrow
 (\nu^{1/2}\pi\rtimes\nu^{-1/2}\sigma)_{N,\theta}\longrightarrow
 ({\rm XIb})_{N,\theta}\longrightarrow0.
\end{equation}
It follows from Lemma \ref{Kcompactexactlemma} that
\begin{equation}\label{XIaBesseleq2}
 \Hom_D(\nu^{1/2}\pi\rtimes\nu^{-1/2}\sigma,\C_{\Lambda\otimes\theta})
 =\Hom_D({\rm XIa},\C_{\Lambda\otimes\theta})\oplus
 \Hom_D({\rm XIb},\C_{\Lambda\otimes\theta}).
\end{equation}
Observe here that, by Lemma \ref{siegelinducedbesselwaldspurgerlemma}, the left side equals
${\rm Hom}_T(\sigma\pi,\C_\Lambda)$, which is at most one-dimensional.

Assume that the representation XIa admits a $(\Lambda,\theta)$-Bessel functional. Then $\Lambda\neq\sigma\circ\Norm_{L/F}$ and $\Hom_T(\sigma\pi,\C_{\Lambda})\neq0$ by Corollary \ref{fourdimthetatheoremcor1} and Theorem \ref{Ganthetatheorem}. Conversely, assume that  $\Lambda\neq\sigma\circ\Norm_{L/F}$ and $\Hom_T(\sigma\pi,\C_{\Lambda})\neq0$. Assume also that XIa does not admit a $(\Lambda,\theta)$-Bessel functional; we will obtain a contradiction. By the one-dimensionality of the space on the left hand side of \eqref{XIaBesseleq2}, we have $\Hom_D({\rm XIb},\C_{\Lambda\otimes\theta})\neq0$. By Corollary \ref{fourdimthetatheoremcor1} and Theorem \ref{Ganthetatheorem}, we conclude $\Lambda=\sigma\circ\Norm_{L/F}$, contradicting our assumption.

Assume that the representation XIb admits a $(\Lambda,\theta)$-Bessel functional. Then $\Lambda=\sigma\circ\Norm_{L/F}$ and $\Hom_T(\pi,\C_1)\neq0$ by Corollary \ref{fourdimthetatheoremcor1} and Theorem \ref{Ganthetatheorem}. Conversely, assume that $\Lambda=\sigma\circ\Norm_{L/F}$ and $\Hom_T(\pi,\C_1)\neq0$. Assume also that XIb does not admit a $(\Lambda,\theta)$-Bessel functional; we will obtain a contradiction. By our assumption, the space on the left hand side of \eqref{XIaBesseleq2} is one-dimensional. Hence $\Hom_D({\rm XIa},\C_{\Lambda\otimes\theta})\neq0$. By what we have already proven, this implies $\Lambda\neq\sigma\circ\Norm_{L/F}$, a contradiction.

\underline{Va$^*$}:
This was proved in Corollary \ref{bsummary}.

\underline{XIa$^*$}:
By Proposition \ref{nongenericsplitproposition}, the representation XIa$^*$ has no split Bessel functionals. Assume that $\theta$ is non-split. By Corollary \ref{fourdimthetatheoremcor1} and Theorem \ref{Ganthetatheorem}, if XIa$^*$ admits a $(\Lambda,\theta)$-Bessel functional, then $\Lambda=\sigma\circ\Norm$ and $\Hom_T(\pi^{\mathrm{JL}},\C_1)\neq0$. Conversely, assume that $\Lambda=\sigma\circ\Norm$ and $\Hom_T(\pi^{\mathrm{JL}},\C_1)\neq0$. By Corollary \ref{Vastarlemma}, the twisted Jacquet module $\delta^*(\nu^{1/2}\pi,\nu^{-1/2}\sigma)_{N,\theta}$ is one-dimensional. Therefore, XIa$^*$ does admit a $(\Lambda',\theta)$-Bessel functional for some $\Lambda'$. By what we already proved, $\Lambda'=\Lambda$.

This concludes the proof.
\end{proof}
\subsection{Some cases of uniqueness}\label{uniquenesssec}
Let $(\pi,V)$ be an irreducible, admissible representation of $\GSp(4,F)$. Using the notations from Sect.\ \ref{besselsec}, consider $(\Lambda,\theta)$-Bessel functionals for $\pi$. We say that such functionals are \emph{unique} if the dimension of the space $\Hom_D(V,\C_{\Lambda\otimes\theta})$ is at most $1$. In this section we will prove the uniqueness of split Bessel functionals for all representations, and the uniqueness of non-split Bessel functionals for all non-supercuspidal representations.

As far as we know, a complete proof that Bessel functionals are unique for all $(\Lambda,\theta)$ and all representations $\pi$ has not yet appeared in the literature. In \cite{NovoPia1973} it is proved that $(1,\theta)$-Bessel functionals are unique if $\pi$ has trivial central character. The main ingredient for this proof is Theorem 1' of \cite{GelfandKazhdan1975}. In \cite{Novodvorski1973} it is proved that $(\Lambda,\theta)$-Bessel functionals are unique if $\pi$ has trivial central character. The proof is based on a generalization of Theorem 1' of \cite{GelfandKazhdan1975}. In \cite{Rodier1976} it is stated, without proof, that $(\Lambda,\theta)$-Bessel functionals are unique if $\pi$ is supercuspidal and has trivial central character.

\begin{lemma}\label{splituniquenesslemma}
 Let $\sigma_1$ be a character of $F^\times$, and let $(\pi_1,V_1)$ be an irreducible, admissible representation of $\GL(2,F)$. Let the matrix $S$ be as in \eqref{splitSeq}, and $\theta$ be as in \eqref{splitthetaTeq}. The resulting group $T$ is then given by \eqref{splitthetaTeq}. Let $(\pi,V)$ be an irreducible, admissible representation of $\GSp(4,F)$. Assume there is an exact sequence
 \begin{equation}\label{splituniquenesslemmaeq}
  \pi_1\rtimes\sigma_1\longrightarrow\pi\longrightarrow0.
 \end{equation}
 Let $\Lambda$ be a character of $T$. If $\Lambda$ is not equal to one of the characters $\Lambda_1$ or $\Lambda_2$, given by
 \begin{align}
  \label{splituniquenesslemmaeq2}\Lambda_1({\rm diag}(a,b,a,b))&=\nu^{1/2}(a)\nu^{-1/2}(b)\sigma_1(ab)\omega_{\pi_1}(a),\\
  \label{splituniquenesslemmaeq3}\Lambda_2({\rm diag}(a,b,a,b))&=\nu^{-1/2}(a)\nu^{1/2}(b)\sigma_1(ab)\omega_{\pi_1}(b),
 \end{align}
 then $(\Lambda,\theta)$-Bessel functionals are unique.
\end{lemma}
\begin{proof}
Since $\pi$ is a quotient of $\pi_1\rtimes\sigma_1$, it suffices to prove that $\Hom_D(\pi_1\rtimes\sigma_1,\C_{\Lambda\otimes\theta})$ is at most one-dimensional. Any element $\beta$ of this space factors through the Jacquet module $(\pi_1\rtimes\sigma_1)_{N,\theta}$. These Jacquet modules were calculated in Lemma \ref{siegelinducedbesselwaldspurgerlemma} ii). Using the notation of this lemma, the assumption about $\Lambda$ implies that restriction of $\beta$ to $J_2$ establishes an injection
$$
 \Hom_D(\pi_1\rtimes\sigma_1,\C_{\Lambda\otimes\theta})\longrightarrow\Hom_{\left[\begin{smallmatrix} *\\&*\end{smallmatrix}\right]}(\sigma_1\pi_1,\C_\Lambda).
$$
The space on the right is at most one-dimensional; see Sect.\ \ref{siegelindsec}. This proves our statement.
\end{proof}

\begin{theorem}\label{splituniquenesstheorem}
 Let $(\pi,V)$ be an irreducible, admissible representation of $\GSp(4,F)$.
 \begin{enumerate}
  \item Split Bessel functionals for $\pi$ are unique.
  \item Non-split Bessel functionals for $\pi$ are unique, if $\pi$ is not supercuspidal, or if $\pi$ is of type $\mathrm{Va^*}$ or $\mathrm{XIa^*}$.
 \end{enumerate}
\end{theorem}
\begin{proof}
i) By Proposition \ref{nongenericsplitproposition}, we may assume that $\pi$ is generic. Let the matrix $S$ be as in \eqref{splitSeq}, and $\theta$ be as in \eqref{splitthetaeq}. The resulting group $T$ is then given by \eqref{splitthetaTeq}. Let $\Lambda$ be a character of $T$. We use the fact that any $(\Lambda,\theta)$-Bessel functional $\beta$ on $V$ factors through the $P_3$-module $V_{Z^J}$.

Assume that $\pi$ is supercuspidal. Then, by Theorem \ref{finitelength}, the associated $P_3$-module $V_{Z^J}$ equals $\tau_{\GL(0)}^{P_3}(1)$. Therefore, the space of $(\Lambda,\theta)$-Bessel functionals on $V$ equals the space of linear functionals considered in Lemma 2.5.4 of \cite{NF}. By this lemma, this space is one-dimensional.

Now assume that $\pi$ is non-supercuspidal. As in the proof of Proposition \ref{nongenericsplitproposition}, we write the semisimplification of the quotient $V_1/V_2$ in the $P_3$-filtration as $\sum_{i=1}^n\tau^{P_3}_{\GL(1)}(\chi_i)$ with characters $\chi_i$ of $F^\times$. Let $C(\pi)$ be the set of characters $\chi_i$. Proposition 2.5.7 of \cite{NF} states that if the character $a\mapsto\Lambda({\rm diag}(a,1,a,1))$ is not contained in the set $\nu^{-1}C(\pi)$, then the set of $(\Lambda,\theta)$-Bessel functionals is at most one-dimensional (note that the arguments in the proof of this proposition do not require the hypothesis of trivial central character). The table below lists the sets $\nu^{-1}C(\pi)$ for all generic non-supercuspidal representations. This table implies that $(\Lambda,\theta)$-Bessel functionals for types VII, VIIIa and IXa are unique.

Assume that $\pi$ is not one these types. Then there exists a sequence as in \eqref{splituniquenesslemmaeq} for some irreducible, admissible representation $\pi_1$ of $\GL(2,F)$ and some character $\sigma_1$ of $F^\times$. These $\pi_1$ and $\sigma_1$ are listed in the table below. Let $\Lambda_1$, $\Lambda_2$ be the characters defined in
\eqref{splituniquenesslemmaeq2} and \eqref{splituniquenesslemmaeq3}. Note that, since $\Lambda_1$ and $\Lambda_2$ are Galois conjugate, we have
\begin{equation}\label{splituniquenesstheoremeq1}
 \dim\Hom_D(\pi,\C_{\Lambda_1\otimes\theta})=\dim\Hom_D(\pi,\C_{\Lambda_2\otimes\theta})
\end{equation}
by \eqref{besselgaloiseq}. By Lemma \ref{splituniquenesslemma}, it suffices to prove that these spaces are one-dimensional. Define characters $\lambda_1,\lambda_2$ of $F^\times$ by
\begin{align*}
 \lambda_1(a)&=\Lambda_1({\rm diag}(a,1,a,1))=\nu^{1/2}(a)\sigma_1(a)\omega_{\pi_1}(a),\\
 \lambda_2(a)&=\Lambda_2({\rm diag}(a,1,a,1))=\nu^{-1/2}(a)\sigma_1(a).
\end{align*}
The set $\{\lambda_1,\lambda_2\}$ is listed in the table below for each representation. By the previous paragraph, the spaces \eqref{splituniquenesstheoremeq1} are one-dimensional if $\{\lambda_1,\lambda_2\}$ is not a subset of $\nu^{-1}C(\pi)$. This can easily be verified using the table below.

$$
 \begin{array}{ccccc}
  \toprule
  \pi&\pi_1&\sigma_1&\{\lambda_1,\lambda_2\}&\nu^{-1}C(\pi)\\
  \toprule
  \text{I}&\chi_1\times\chi_2&\sigma&\{\nu^{1/2}\chi_1\chi_2\sigma,\nu^{-1/2}\sigma\}&\{\nu^{1/2}\chi_1\chi_2\sigma,\nu^{1/2}\chi_1\sigma,\nu^{1/2}\chi_2\sigma,\nu^{1/2}\sigma\}\\
  \midrule
  \text{IIa}&\chi\St_{\GL(2)}&\sigma&\{\nu^{1/2}\chi^2\sigma,\nu^{-1/2}\sigma\}&\{\nu^{1/2}\chi^2\sigma,\nu^{1/2}\sigma,\nu\chi\sigma\}\\
  \midrule
  \text{IIIa}&\chi^{-1}\times\nu^{-1}&\nu^{1/2}\chi\sigma&\{\sigma,\chi\sigma\}&\{\nu\chi\sigma,\nu\sigma\}\\
  \midrule
  \text{IVa}&\nu^{-3/2}\St_{\GL(2)}&\nu^{3/2}\sigma&\{\nu^{-1}\sigma,\nu\sigma\}&\{\nu^2\sigma\}\\
  \midrule
  \text{Va}&\nu^{-1/2}\xi\St_{\GL(2)}&\nu^{1/2}\xi\sigma&\{\xi\sigma\}&\{\nu\sigma,\nu\xi\sigma\}\\
  \midrule
  \text{VIa}&\nu^{-1/2}\St_{\GL(2)}&\nu^{1/2}\sigma&\{\sigma\}&\{\nu\sigma\}\\
  \midrule
  \text{VII}&\text{---}&\text{---}&\text{---}&\emptyset\\
  \midrule
  \text{VIIIa}&\text{---}&\text{---}&\text{---}&\emptyset\\
  \midrule
  \text{IXa}&\text{---}&\text{---}&\text{---}&\emptyset\\
  \midrule
  \text{X}&\pi&\sigma&\{\nu^{1/2}\omega_\pi\sigma,\nu^{-1/2}\sigma\}&\{\nu^{1/2}\omega_\pi\sigma,\nu^{1/2}\sigma\}\\
  \midrule
  \text{XIa}&\nu^{-1/2}\pi&\nu^{1/2}\sigma&\{\sigma\}&\{\nu\sigma\}\\
  \bottomrule
 \end{array}
$$

ii) Assume first that $\pi$ is not supercuspidal. Then there exist an irreducible, admissible representation $\pi_1$ of $\GL(2,F)$ and a character $\sigma$ of $F^\times$ such that $\pi$ is either a quotient of $\pi_1\rtimes\sigma$, or a quotient of $\sigma\rtimes\pi_1$. The assertion of ii) now follows from i) of Lemma \ref{siegelinducedbesselwaldspurgerlemma} and Lemma \ref{Klingendegjacquetlemma1}.

Now assume that $\pi=\delta^*([\xi,\nu\xi],\nu^{-1/2}\sigma)$ is of type Va$^*$. Suppose that $\Hom_D(\pi,\C_{\Lambda\otimes\theta})$ is non-zero for some $\theta$ and $\Lambda$, with $L$ being a field. By our main result Theorem \ref{mainnonsupercuspidaltheorem}, the quadratic extension $L$ is the field corresponding to $\xi$ and $\Lambda=\sigma\circ\Norm_{L/F}$. By Corollary \ref{Vastarlemma}, the Jacquet module $\pi_{N,\theta}$ is one-dimensional. This implies that $\Hom_D(\pi,\C_{\Lambda\otimes\theta})$ is one-dimensional.

Finally, assume that $\pi=\delta^*(\nu^{1/2}\pi,\nu^{-1/2}\sigma)$ is of type XIa$^*$. Suppose that $\Hom_D(\pi,\C_{\Lambda\otimes\theta})$ is non-zero for some $\theta$ and $\Lambda$, with $L$ being a field. By our main result Theorem \ref{mainnonsupercuspidaltheorem}, we have $\Lambda=\sigma\circ\Norm_{L/F}$ and $\Hom_T(\pi^{\mathrm{JL}},\C_1)\neq0$. By Corollary \ref{Vastarlemma}, the Jacquet module $\pi_{N,\theta}$ is one-dimensional. This implies that $\Hom_D(\pi,\C_{\Lambda\otimes\theta})$ is one-dimensional.
\end{proof}

\section{Some applications}
We present two applications that result from the methods used in this paper. The first application is a characterization of irreducible, admissible, non-generic representations of $\GSp(4,F)$ in terms of their twisted Jacquet modules and their Fourier-Jacobi quotient. The second application concerns the existence of certain vectors with good invariance properties.
\subsection{Characterizations of non-generic representations}
As before, we fix a non-trivial character $\psi$ of $F$.
\begin{lemma} \label{ZJfinitelengthlemma}
  Let $(\pi,V)$ be a non-generic, supercuspidal, irreducible, admissible representation of  $\GSp(4,F)$. Then $\dim V_{N,\theta}<\infty$ for all non-degenerate $\theta$.
\end{lemma}
\begin{proof}
If $\theta$ is split, then $V_{N,\theta}=0$ by Theorem \ref{finitelength}. Assume that $\theta$ is not split. Let $\theta=\theta_S$ with $S$ as in \eqref{Sdefeq}. We may assume that $\dim V_{N,\theta}\neq0$. Let $X$ be as in \eqref{XHeq}. By Theorem 5.6 of \cite{GaTa2011}, there exists an irreducible, admissible representation $\sigma$ of $\GO(X)$ such that $\Hom_R(\omega,\pi\otimes\sigma)\neq0$; here, $\omega$ is the Weil representation defined in Sect.\ \ref{thetabesselsubsec}. By i) of Theorem \ref{fourdimthetatheorem}, the set $\Omega_S$ is non-empty. By Proposition \ref{scdimprop}, the dimension of $V_{N,\theta}$ is finite.
\end{proof}

Let $W$ be a smooth representation of $N$. We will consider the dimensions of the complex vector spaces
$
 W_{N, \theta_{a,b,c}}.
$
Fix representatives $a_1,\dots, a_t$ for $F^\times / F^{\times 2}$. We define
$$
 d(W)=\sum_{i=1}^t \dim W_{N, \theta_{a_i,0,1}}.
$$
If $0=W_0 \subset W_1 \subset W_2 \subset \dots \subset W_k=W$ is a chain of $N$ subspaces, then 
\begin{equation}\label{dWsumeq}
 d(W) = \sum_{j=1}^k d(W_j/W_{j-1}).
\end{equation}
If one of the spaces $W_{N, \theta_{a_i,0,1}}$ is infinite-dimensional, then this equality still holds in the sense that both sides are infinite.

\begin{lemma} \label{detectlemma} Let $W^J$ be a non-zero, irreducible, smooth representation of $G^J$ admitting $\psi$ as a central character. Then $1\leq d(W^J) \leq\#F^\times/F^{\times2}$.
\end{lemma}
\begin{proof}
This follows immediately from Lemma \ref{GJWhitlemma}.
\end{proof}

\begin{lemma}\label{kdboundlemma}
 Let $(\tau^J,W^J)$ be a smooth representation of $G^J$. Then $W^J$ has finite length if and only if $d(W^J)$ is finite. If it has finite length, then
 $$
  {\rm length}(W^J)\leq d(W^J)\leq{\rm length}(W^J)\cdot\#F^\times/F^{\times2}.
 $$
\end{lemma}
\begin{proof}
Assume that $W^J$ has finite length. Let 
$$
 0=W_0 \subset W_1 \subset W_2 \subset \dots \subset W_k=W^J
$$
be a chain of $G^J$ spaces such that $W_j/W_{j-1}$ is not zero and irreducible. By \eqref{dWsumeq}, we have 
$$
d(W^J)=\sum_{j=1}^k d(W_j/W_{j-1}).
$$
By Lemma \ref{detectlemma}, $1\leq d(W_j/W_{j-1}) \leq\#F^\times/F^{\times2}$ for $j=1,\dots,k$. It follows that $d(W^J)$ is finite, and that the asserted inequalities hold.

If $W^J$ has infinite length, a similar argument shows that $d(W)$ is infinite.
\end{proof}

\begin{theorem}\label{nongenchartheorem}
 Let $(\pi,V)$ be an irreducible, admissible representation of $\GSp(4,F)$. The following statements are equivalent.
 \begin{enumerate}
  \item $\pi$ is not generic.
  \item $\dim V_{N,\theta}<\infty$ for all split $\theta$.
  \item $\dim V_{N,\theta}<\infty$ for all non-degenerate $\theta$.
  \item The $G^J$-representation $V_{Z^J,\psi}$ has finite length.
 \end{enumerate}
\end{theorem}
\begin{proof}
i) $\Rightarrow$ iii) Assume that $\pi$ is not generic. Let $\theta$ be a non-degenerate character of $N$. Assume first that $\theta$ is split. Then $V_{N,\theta}$ can be calculated from the $P_3$-filtration of $\pi$. As in the proof of Proposition \ref{nongenericsplitproposition} we see that $V_{N,\theta}$ is finite-dimensional.

Now assume that $\theta$ is not split. If $\pi$ is supercuspidal, then $\dim V_{N,\theta}<\infty$ by Lemma \ref{ZJfinitelengthlemma}. Assume that $\pi$ is not supercuspidal. Then the table of Bessel functionals shows that $\pi$ admits $(\Lambda,\theta)$-Bessel functionals only for finitely many $\Lambda$. Since every $\Lambda$ can occur in $V_{N,\theta}$ at most once by the uniqueness of Bessel functionals (Theorem \ref{splituniquenesstheorem}), this implies that $V_{N,\theta}$ is finite-dimensional.

iii) $\Rightarrow$ ii) is trivial.

ii) $\Rightarrow$ i) Assume that $\pi$ is generic. Then the subspace $V_2$ of the $P_3$-module $V_{Z^J}$ from Theorem \ref{finitelength} is non-zero. In fact, this subspace is isomorphic to the representation $\tau^{P_3}_{\GL(0)}(1)$ defined in \eqref{tauP3GL0eq}. By Lemma 2.5.4 of \cite{NF}, the space
$$
 (V_2)_{\left[\begin{smallmatrix}1&*&*\\&1\\&&1\end{smallmatrix}\right],\theta_{0,1}},
$$
where $\theta_{a,b}$ is defined in \eqref{thetaabdefeq}, is infinite-dimensional. This implies that $V_{N,\theta_{0,1,0}}$ is infinite-dimensional, contradicting the hypothesis in ii).

iii) $\Leftrightarrow$ iv) 
Let $W^J=V_{Z^J,\psi}$. Then $W^J_{N,\theta_{a,0,1}}=V_{N,\theta_{a,0,1}}$ for any $a$ in $F^\times$, so that $d(W^J)=d(V)$. Lemma \ref{kdboundlemma} therefore implies that iii) and iv) are equivalent.
\end{proof}

For more thoughts on $V_{Z^J,\psi}$, see \cite{AdlerPrasad2006}. Theorem \ref{nongenchartheorem} answers one of the questions mentioned at the end of this paper.

\subsection{Invariant vectors}
Let $(\pi,V)$ be an irreducible, admissible representation of $\GSp(4,F)$. In this section we will prove the existence of a vector $v$ in $V$ such that ${\rm diag}(1,1,c,c)v=v$ for all units $c$ in the ring of integers $\OF$ of $F$. This result was motivated by a question of Abhishek Saha; see \cite{Sa2013}.

Our main tool will be the $G^J$-module $V_{Z^J,\psi}$ for a smooth representation $(\pi,V)$ of $\GSp(4,F)$. Throughout this section we will make a convenient assumption about the character $\psi$ of $F$, namely that $\psi$ has conductor $\OF$. By definition, this means that $\psi$ is trivial on $\OF$, but not on $\p^{-1}$, where $\p$ is the maximal ideal of $\OF$. We normalize the Haar measure on $F$ such that $\OF$ has volume $1$. Let $q$ be the cardinality of the residue class field $\OF/\p$.

In this section, we will abbreviate
$$
 d(c)=\begin{bmatrix}1\\&1\\&&c\\&&&c\end{bmatrix},\qquad z(x)=\begin{bmatrix}1&&&x\\&1\\&&1\\&&&1\end{bmatrix}
$$
for $c$ in $F^\times$ and $x$ in $F$.

\begin{lemma}\label{VZJlemma1}
 Let $(\pi,V)$ be a smooth representation of $\GSp(4,F)$. Let $p:\:V\rightarrow V_{Z^J,\psi}$ be the projection map, and let $w$ in $V_{Z^J,\psi}$ be non-zero. Then there exists a positive integer $m$ and a non-zero vector $v$ in $V$ with the following properties.
 \begin{enumerate}
  \item $p(v)=w$.
  \item $\pi(z(x))v=\psi(x)v$ for all $x\in\p^{-m}$.
  \item $\pi(d(c))v=v$ for all $c\in1+\p^m$.
 \end{enumerate}
\end{lemma}
\begin{proof}
Let $v_0$ in $V$ be such that $p(v_0)=w$. Let $m$ be a positive integer such that $\pi(d(c))v_0=v_0$ for all $c\in1+\p^m$. Set $v=q^{-m}\int\limits_{\p^{-m}}\psi(-x)\pi(z(x))v_0\,dx$. Then $p(v)=w$.
In particular, $v$ is not zero. Evidently, $v$ has property ii). Moreover, for $c$ in $1+\p^m$,
\begin{align*}
 \pi(d(c))(v)&=q^{-m}\int\limits_{\p^{-m}}\psi(-x)\pi(z(xc^{-1})d(c))v_0\,dx\\
 &=q^{-m}\int\limits_{\p^{-m}}\psi(-xc)\pi(z(x))v_0\,dx\\
 &=q^{-m}\int\limits_{\p^{-m}}\psi(-x)\pi(z(x))v_0\,dx\\
 &=v.
\end{align*}
This concludes the proof.
\end{proof}

\begin{lemma}\label{VZJlemma2}
 Let $(\pi,V)$ be a smooth representation of $\GSp(4,F)$. Let $p:\:V\rightarrow V_{Z^J,\psi}$ be the projection map. Let $m$ be a positive integer. Assume that $v$ in $V$ is such that $\pi(z(x))v=\psi(x)v$ for all $x\in\p^{-m}$. If $c$ is in $\OF^\times$ but not in $1+\p^m$, then $p(\pi(d(c))v)=0$.
\end{lemma}
\begin{proof}
Let $w=\pi(d(c))v$. To show that $p(w)=0$, it is enough to show that
$$
 \int\limits_{\p^{-m}}\psi(-x)\pi(z(x))w\,dx=0.
$$
Indeed,
\begin{align*}
 \int\limits_{\p^{-m}}\psi(-x)\pi(z(x))w\,dx&=\int\limits_{\p^{-m}}\psi(-x)\pi(z(x)d(c))v\,dx\\
 &=\pi(d(c))\int\limits_{\p^{-m}}\psi(-x)\pi(z(xc))v\,dx\\
 &=\pi(d(c))\int\limits_{\p^{-m}}\psi(-x)\psi(xc)v\,dx\\
 &=\Big(\int\limits_{\p^{-m}}\psi(x(c-1))\,dx\Big)\pi(d(c))v\\
 &=0,
\end{align*}
since $c\notin1+\p^m$ and $\psi$ has conductor $\OF$.
\end{proof}

\begin{proposition}\label{VZJprop}
 Let $(\pi,V)$ be a smooth representation of $\GSp(4,F)$. Let $p:\:V\rightarrow V_{Z^J,\psi}$ be the projection map. Let $w$ be in $V_{Z^J,\psi}$. Then there exists a unique vector $v$ in $V$ with the following properties.
 \begin{enumerate}
  \item $p(v)=w$.
  \item $\pi(z(x))v=v$ for all $x\in\OF$.
  \item $\int\limits_{\p^{-1}}\pi(z(x))v\,dx=0$.
  \item $\pi(d(c))v=v$ for all $c\in\OF^\times$.
 \end{enumerate}
\end{proposition}
\begin{proof}
For the existence part we may assume that $w$ is non-zero. Let the positive integer $m$ and $v$ in $V$ be as in Lemma \ref{VZJlemma1}. Define $v_1=q^m\int\limits_{\OF^\times}\pi(d(c))v\,dc$.
Then, by Lemma \ref{VZJlemma2},
\begin{align*}
 p(v_1)&=q^m\int\limits_{\OF^\times}p(\pi(d(c))v)\,dc\\
 &=q^m\int\limits_{1+\p^m}p(\pi(d(c))v)\,dc\\
 &=q^m\int\limits_{1+\p^m}p(v)\,dc\\
 &=w.
\end{align*}
Evidently, $v_1$ has property iv). To see properties ii) and iii), let $x$ be in $\p^{-1}$. By ii) of Lemma \ref{VZJlemma1},
\begin{align*}
 \pi(z(x))v_1&=q^m\int\limits_{\OF^\times}\pi(d(c)z(xc))v\,dc=q^m\int\limits_{\OF^\times}\psi(xc)\pi(d(c))v\,dc.
\end{align*}
It follows that $v_1$ has property ii). Integrating over $x$ in $\p^{-1}$ shows that $v_1$ has property iii) as well.

To prove that $v_1$ is unique, let $V_1$ be the subspace of $V$ consisting of vectors $v$ satisfying properties ii), iii) and iv). We will prove that the restriction of $p$ to $V_1$ is injective (so that $p$ induces an isomorphism $V_1\cong V_{Z^J,\psi}$). Let $v$ be in $V_1$ and assume that $p(v)=0$. Then there exists a positive integer $m$ such that
$$
 \int\limits_{\p^{-m}}\psi(-x)\pi(z(x))v\,dx=0.
$$
Applying $d(c)$ to this equation, where $c$ is in $\OF^\times$, leads to
$$
 \int\limits_{\p^{-m}}\psi(-cx)\pi(z(x))v\,dx=0.
$$
Integrating over $c$ in $\OF^\times$, we obtain
$$
 q^{-1}\int\limits_{\p^{-1}}\pi(z(x))v\,dx=\int\limits_\OF\pi(z(x))v\,dx.
$$
Using properties ii) and iii) it follows that $v=0$. This concludes the proof.
\end{proof}
\begin{corollary}\label{VZJpropcor}
 Let $(\pi,V)$ be an irreducible, admissible representation of $\GSp(4,F)$ that is not a twist of the trivial representation. Then there exists a vector $v$ in $V$ that is invariant under all elements $d(c)$ with $c$ in $\OF^\times$.
\end{corollary}
\begin{proof}
By Proposition \ref{VZJprop}, it is enough to show that $V_{Z^J,\psi}$ is non-zero. By Proposition \ref{Nthetaprop}, there exists a non-trivial character $\theta$ of $N$ such that $V_{N,\theta}\neq0$. We may assume that $\theta$ is of the form \eqref{thetaSsetupeq} with $c=1$. The assertion follows.
\end{proof}
\section*{}
\phantomsection
\label{refsec}
\addcontentsline{toc}{section}{References}
\bibliography{bessel_functionals}{}
\bibliographystyle{plain}

\end{document}